\documentclass[11pt]{article}
\usepackage{amsmath}
\usepackage{amsfonts}
\usepackage{amssymb}
\usepackage{mathrsfs}
\usepackage{amsthm}
\usepackage{palatino}
\usepackage[margin=2.5cm, vmargin={1.5cm}]{geometry}

\usepackage{dcolumn}

\usepackage{times}
\usepackage{graphicx}
\usepackage{color}

\usepackage{pstricks}
\usepackage{pst-plot}

\DeclareMathOperator{\oplambda}{\Lambda}
\DeclareMathOperator{\opv}{\nu}
\DeclareMathOperator{\opd}{\delta}
\DeclareMathOperator{\opS}{{\it S}}
\DeclareMathOperator{\opsmall}{{\it s}}

\title{Hyperbolic Fourier coefficients of Poincar\'e series}

\author{Cormac O'Sullivan,  Karen Taylor\footnote{
\newline
2010 Mathematics Subject Classification 11F11, 11F30
\newline
Support for this project was provided by a PSC-CUNY Award, jointly funded by The Professional Staff Congress and The City University of New York.}
}

\begin{document}

\maketitle

\vspace{-10mm}
\begin{center}
{\em In memory of Marvin Knopp}
\end{center}

\def\s#1#2{\langle \,#1 , #2 \,\rangle}

\def\H{{\mathbb H}}
\def\F{{\mathfrak F}}
\def\C{{\mathbb C}}
\def\R{{\mathbb R}}
\def\Z{{\mathbb Z}}
\def\Q{{\mathbb Q}}
\def\N{{\mathbb N}}
\def\O{{\mathbb O}}
\def\st{{\mathbb S}}
\def\D{{\mathbb D}}
\def\B{{\mathbb B}}
\def\fun{{\mathbb F}}
\def\G{{\Gamma}}
\def\GH{{\G \backslash \H}}
\def\g{{\gamma}}
\def\L{{\Lambda}}
\def\ee{{\varepsilon}}
\def\er{\eqref}
\def\K{{\mathcal K}}
\def\Re{\text{\rm Re}}
\def\Im{\text{\rm Im}}
\def\SL{\text{\rm SL}}
\def\GL{\text{\rm GL}}
\def\PSL{\text{\rm PSL}}
\def\sgn{\text{\rm sgn}}
\def\tr{\text{\rm tr}}
\def\lqs{\leqslant}
\def\gqs{\geqslant}
\def\res{\operatornamewithlimits{Res}}
\def\vol{\text{\rm Vol}}

\def\e{\mathbf{e}}

\def\ca{{\mathfrak a}}
\def\cb{{\mathfrak b}}
\def\cc{{\mathfrak c}}
\def\cd{{\mathfrak d}}
\def\ci{{\infty}}

\def\sa{{\sigma_\mathfrak a}}
\def\sb{{\sigma_\mathfrak b}}
\def\sc{{\sigma_\mathfrak c}}
\def\sd{{\sigma_\mathfrak d}}
\def\si{{\sigma_\infty}}

\def\se{{\sigma_\eta}}
\def\sep{{\sigma_{\eta'}}}
\def\sz{{\sigma_{z_0}}}

\def\red{\color{red} }

\newcolumntype{d}[1]{D{.}{.}{#1} }


\newtheorem{theorem}{Theorem}[section]
\newtheorem{lemma}[theorem]{Lemma}
\newtheorem{prop}[theorem]{Proposition}
\newtheorem{cor}[theorem]{Corollary}
\newtheorem{conj}[theorem]{Conjecture}
\newtheorem{remark}[theorem]{Remark}

\newcounter{coundef}
\newtheorem{adef}[coundef]{Definition}

\newcounter{thm1count}
\newtheorem{thm1}[thm1count]{Theorem}

\renewcommand{\labelenumi}{(\roman{enumi})}

\numberwithin{equation}{section}

\bibliographystyle{alpha}

\vspace{0mm}

\begin{abstract}

Poincar\'e in 1912 and Petersson in 1932 gave the now classical  expression for the parabolic Fourier coefficients of holomorphic Poincar\'{e} series in terms of Bessel functions and Kloosterman sums. Later,
in 1941, Petersson introduced hyperbolic and elliptic Fourier expansions of modular forms and the associated hyperbolic and elliptic Poincar\'{e} series.
In this paper we express the hyperbolic Fourier coefficients of Poincar\'e series, of both parabolic and hyperbolic type, in terms of hypergeometric series and Good's generalized Kloosterman sums. In an explicit example for the modular group,  we see that the hyperbolic Kloosterman sum corresponds to a sum over lattice points on a hyperbola contained in an ellipse. This allows for numerical computation of the hyperbolic Fourier coefficients.
\end{abstract}

\section{Introduction}
The group $\SL_2(\R)$ acts by linear fractional transformations on $\H \cup \R \cup\{ \infty\}$ with $\H$ denoting the upper half plane.
Let $\G \subset \SL_2(\R)$ be a  Fuchsian
group of the first kind, i.e. a discrete subgroup of $\SL_2(\R)$ so that $\G\backslash\H$ has finite hyperbolic volume. Write $Z:=\{\pm I\} \cap \G$ for $I$ the identity matrix.  Elements in $\G-Z$ may be classified as parabolic, elliptic or hyperbolic according to their types of fixed points. A function $f$ on $\H$ transforms with weight $k$ with respect to $\G$ if $
(f|_k \g)(z) = f(z)$ for all $\g \in \G$, where
$
(f|_k \g)(z)$ indicates $j(\g,z)^{-k}f(\g z)$ for $j(\g,z):=cz+d$ when $\g = \left(\smallmatrix a & b \\ c & d
\endsmallmatrix\right)$. Unless stated otherwise, we assume throughout that $k$  is even and at least $4$.

The usual way to describe such an $f$ is in terms of its Fourier expansion. For example,  the modular discriminant function is of weight $12$ for $\G=\SL_2(\Z)$, see Section \ref{num}, and its expansion begins
\begin{equation}\label{dl0}
\Delta(z)    =  q -24q^2+ 252q^3 -1472q^4+ 4830q^5  + \cdots \quad (q=e^{2\pi i z}).
\end{equation}
To describe generalizations of this Fourier expansion, we first review some basic notation and results for modular forms as described in \cite{S71}, \cite{R77} and \cite{IwTo}, for example.

The series \eqref{dl0} is the Fourier expansion corresponding to the cusp (parabolic fixed point) at $\infty$. In general, for a cusp $\ca$  for $\G$, let $\G_\ca$ be the subgroup fixing $\ca$. Then $\overline{\G}_\ca$ is isomorphic to $\Z$, where the bar means the image  under the map $\SL_2(\R) \to \SL_2(\R)/\pm I$. This isomorphism can be seen explicitly as there exists a scaling matrix $\sa \in \SL_2(\R)$ such that $\sa \infty = \ca$ and
$$
 \sa^{-1}\overline{\G}_\ca \sa  = \left\{\left. \pm \begin{pmatrix}  1 & m \\ 0 & 1  \end{pmatrix} \ \right| \  m\in \Z\right\}.
$$
The matrix $\sa$ is unique up to multiplication on the right by $\pm\left(\smallmatrix 1
& t \\ 0 & 1 \endsmallmatrix\right)$ for any  $t\in \R$.

\begin{adef} Let $f$ be holomorphic on $\H$ and of weight $k$ with respect to  $\G$. Its  {\em Fourier expansion  at $\ca$} is
\begin{equation}\label{expnpar}
\left(f|_k  \sa\right) (z)=\sum_{m \in \Z} c_\ca(m;f) e^{2\pi i m z}.
\end{equation}
\end{adef}

\begin{adef} Let $S_k(\G)$ be the set of holomorphic functions on $\H$, of weight $k>0$ with respect to $\G$, such that $y^{k/2} f(x+iy)$ is bounded for all $x+iy \in \H$.
\end{adef}

If $\G$ has cusps then $S_k(\G)$ consists of cusp forms $f$ whose coefficients $c_\ca(m;f)$ are zero at every cusp $\ca$ when $m\lqs 0$, see for example \cite[Sect. 5.1]{IwTo}. Relaxing this condition to allow $c_\ca(0;f)$ to be non-zero gives the set $M_k(\G)$ of modular forms, and allowing a finite number of $c_\ca(-m;f)$ to be non-zero for $-m<0$ gives the set $M^!_k(\G)$ of weakly holomorphic forms.  

If $\G$ has no cusps then $\G\backslash\H$ is compact and $S_k(\G)$ is the set of all holomorphic functions on $\H$ with weight $k$, since the condition that $y^{k/2} f(x+iy)$ is bounded  is automatically satisfied. Whether $\G$ has cusps or not, $S_k(\G)$ is a finite dimensional vector space over $\C$, equipped with the Petersson inner product given by $\s{f}{g}:=\int_{\G \backslash\H} y^k f(z) \overline{g(z)} \, d\mu z$
where $d\mu z := y^{-2} dx dy$.

Another result of Petersson \cite{P41} is that alongside the parabolic expansions \eqref{expnpar} there are also  elliptic Fourier expansions associated to each point in $\H$ and  hyperbolic Fourier expansions associated to each pair of hyperbolic fixed points in $\R \cup\{ \infty\}$. For example, the elliptic expansion of $\Delta$ at $i\in \H$ is given in \cite{OSR} as
\begin{equation}\label{dl1}
\bigl( \Delta|_{12} \sigma_i \bigr) (z)
 =   -64\Delta(i)\left(  1-12\frac{(r_iz)^2}{2!}+ 216\frac{(r_iz)^4}{4!}+10368\frac{(r_iz)^6}{6!}     + \dots \right)
\end{equation}
where
$r_i=-\Gamma(1/4)^{4}/(8\sqrt{3}\pi^{2})$ and $\sigma_i =\frac{1-i}{2}\left(\smallmatrix i & i \\ -1 & 1
\endsmallmatrix\right)$.

In this paper we develop the theory of hyperbolic expansions of modular forms, with the aim of expressing the hyperbolic coefficients as explicitly as possible. For example, we show that the expansion of $\Delta$ at the hyperbolic pair $\eta=(-\sqrt{2},\sqrt{2})$ is given numerically by
\begin{multline}\label{dl2}
   \frac{\bigl(\Delta|_{12} \sigma_{\eta} \bigr) (z)}{1721.23 z^{-6}}
 \approx   \cdots -3.47\times 10^{-7}q^{-4}
 +1.20\times 10^{-7} q^{-3}
 +0.00176 q^{-2}
 -0.0937 q^{-1} \\
 + 1 + 25.31 q^{1}
  + 128.12 q^2
   -2.37q^3
    -1849.07q^4 + \cdots
  \qquad(q=z^{2\pi i /\ell_\eta})
\end{multline}
for the scaling matrix $\se$  given in \eqref{d14s} and $\ell_\eta = 2\log(3+2\sqrt{2})$. (We divided by $1721.23$ to make the zeroth coefficient $\approx  1$ and the other coefficients more visible.)

Some examples of hyperbolic expansions have already appeared in the literature. Siegel in \cite[Chap. II, Sect. 3]{Si65} worked out the hyperbolic expansions of parabolic non-holomorphic Eisenstein series in terms of Hecke grossencharacter $L$-functions. In \cite[Prop. 4.2.2]{vP10}, von Pippich  computed the hyperbolic Fourier coefficients of non-holomorphic Eisenstein series of elliptic type. Legendre functions (examples of ${_2}F_1$ hypergeometric functions) appear in these coefficients. Good, in the book \cite{G83}, found the hyperbolic expansions of certain non-holomorphic Poincar\'e series. We will use much of the theory he developed, and expand some of his results that appear there in condensed form. Hiramatsu  in \cite{Hir70} worked in the holomorphic setting. He gave the hyperbolic expansion of an $f$ in $S_k(\G(p,q))$ derived from a Hilbert modular form associated to a real quadratic field. The group $\G(p,q)$ is coming from a quaternion algebra and has no cusps.
In \cite{Hir} he also  found basic bounds on the size of  hyperbolic coefficients for elements of $S_k(\G)$, as we see Subsection \ref{numb}.

\subsection{Hyperbolic definitions}

For most of the definitions and results in this subsection, see \cite{Ka92}, \cite{P41}, \cite{Hir70} and \cite{IO09}.
Let $\eta=(\eta_1,\eta_2)$ be an ordered hyperbolic fixed pair for $\G$, i.e.\ $\eta_1$, $\eta_2$ are distinct elements of $\R \cup\{ \infty\}$ so that there exists a  hyperbolic $\g \in \G$ with $\g \eta_1 = \eta_1$ and $\g \eta_2 = \eta_2$. Let $\G_\eta$ be the subgroup of all such $\g$ fixing $\eta_1$ and $\eta_2$.  There exists a  scaling matrix $\se \in \SL_2(\R)$ such that $\se 0 = \eta_1$, $\se \infty = \eta_2$ and $\se$ is unique up to multiplication on the right by $\left(\smallmatrix t
& 0 \\ 0 & 1/t \endsmallmatrix\right)$ for any $t\in \R_{\neq 0}$.
That $\overline{\G}_\eta$ is isomorphic to $\Z$ may be seen with
\begin{equation}\label{sgsh}
    \se^{-1}\overline{\G}_\eta \se = \left\{\left. \pm \begin{pmatrix}  e^{m\ell_\eta/2} & 0 \\ 0 & e^{-m\ell_\eta/2}  \end{pmatrix}\ \right| \ m\in \Z\right\}.
\end{equation}
The number $\ell_\eta$ is the hyperbolic length of the geodesic from $z$ to $\g_\eta z$ for any $z \in \H$ where $\g_\eta$ is a generator of $\overline{\G}_\eta$. We also set
\begin{equation*}
    \varepsilon_\eta:=e^{\ell_\eta/2}>1.
\end{equation*}

If $f$ has weight $k$ then $e^{k \ell_\eta w/2}\left( f|_k \se \right) (e^{\ell_\eta w})$ has period 1 in $w$ and a Fourier expansion. Rewrite this expansion with $z=e^{\ell_\eta w}$ to get the following. (Here and throughout, the expression $z^s$ for $z,$ $s\in \C$ with $z \neq 0$ means $e^{s\log z}$ using the principal branch of $\log$ with argument convention $-\pi<\arg z \lqs \pi$.)
\begin{adef} Let $f$ be holomorphic on $\H$ and of weight $k$. Its   {\em hyperbolic Fourier expansion  at $\eta$} is
\begin{equation}\label{hypexp}
\left( f|_k \se \right) (z)= \sum_{m \in \Z} c_{\eta}(m;f) z^{-k/2+2\pi i m/\ell_\eta},
\end{equation}
valid for all $z \in \H$.
\end{adef}
The coefficients $c_{\eta}(m;f)$ depend on $\se$ in a simple way:
\begin{equation} \label{hypch}
    \se \ \to \ \se \left(\smallmatrix t
& 0 \\ 0 & 1/t \endsmallmatrix\right) \qquad \implies \qquad c_{\eta}(m;f) \ \to \  c_{\eta}(m;f) \cdot (t^2)^{2\pi i m/\ell_\eta}.
\end{equation}
Also note that the expansions at $\eta$ and $\g \eta$ for $\g \in \G$ might differ by this type of $(t^2)^{2\pi i m/\ell_\eta}$ factor  unless $\sigma_{\g \eta}$ is chosen as $\g \se$. For example, with $-\infty<\eta_1 < \eta_2 < \infty$, a simple choice for the scaling matrix is
\begin{equation}\label{hypscat}
\hat\sigma_\eta :=\frac{1}{\sqrt{\eta_2-\eta_1}}\begin{pmatrix}  \eta_2 & \eta_1 \\ 1 & 1  \end{pmatrix}.
\end{equation}

With
$$
c_{\eta}(m;f) = \int_{w_0}^{w_0+1} e^{k \ell_\eta w/2}\left( f|_k \se \right) (e^{\ell_\eta w}) \cdot e^{-2\pi i m w}\, dw
$$
 we may recover the hyperbolic coefficients
for any $w_0$ satisfying $0<\Im(w_0)<\pi/\ell_\eta$. Writing this as
$$
c_{\eta}(m;f) = \int_{0}^{1} \left( f|_k \se \right) (e^{\ell_\eta (w_0+t)}) \cdot e^{(w_0+t)(k \ell_\eta/2-2\pi i m)}\, dt
$$
and using the change of variables
$
r_0 e^{i\theta_0} = e^{\ell_\eta w_0},$ $r=r_0 e^{\ell_\eta t}
$
then gives (with $\varepsilon_\eta^2 = e^{\ell_\eta}$)
\begin{equation}\label{kcq}
c_{\eta}(m;f) = \frac{e^{i\theta_0(k/2-2\pi i m/ \ell_\eta)}}{\ell_\eta} \int_{r_0}^{\varepsilon_\eta^2 \cdot r_0} \left( f|_k \se \right) (r e^{i\theta_0}) \cdot r^{k/2-2\pi i m/ \ell_\eta}\, \frac{dr}{r}
\end{equation}
valid for arbitrary $r_0>0$ and $0<\theta_0<\pi$.

\begin{adef} The (weight $k$) {\em hyperbolic  Poincar\'e series $P_{\eta, m}$}
  is defined for $m \in \Z$ as
\begin{equation}\label{poinhyp}
    P_{\eta,m}(z):=  \sum_{\g \in \G_\eta \backslash \G} z^{-k/2+2\pi i m /\ell_\eta} \left|_k \se^{-1}\g  \right.
    = \sum_{\g \in \G_\eta \backslash \G}
    \frac{(\se^{-1}\g z)^{-k/2+2\pi i m /\ell_\eta}}
    {j(\se^{-1}\g , z)^{k}}.
\end{equation}
\end{adef}
The convergence is absolute for $k>2$ and uniform for $z$ in compact sets in $\H$.
We have  $P_{\eta, m} \in S_k(\G)$ for $m \in \Z$. For $f\in S_k(\G)$ and $m \in \Z$ the Petersson inner product of $f$ with $P_{\eta, m}$ yields
\begin{equation}\label{epe}
\s{f}{P_{\eta,m}} =  c_{\eta}(m;f) \left[ \frac{\pi  \G(k-1) \ell_\eta   e^{-2\pi^2 m/\ell_\eta}}{2^{k-2} \left|\G\left(k/2+2\pi i m/\ell_\eta \right)\right|^2  } \right].
\end{equation}
It follows from \eqref{epe}
that, for fixed $\eta$ and $m \in \Z$, the series $P_{\eta, m}$   span the space $S_k(\G)$.

These hyperbolic Poincar\'e series, at least in the case $m=0$, have appeared for example in the works of Kohnen and Zagier \cite{KZ}  and Katok \cite{Ka85}, obtaining  hyperbolic rational structures on $S_k(\G)$. See  the related discussion in \cite[Sect. 3]{IO09}.  In \cite{BKK} they discover an interesting generalization of $P_{\eta, 0}$ to a locally harmonic hyperbolic Poincar\'e series of negative weight.

One advantage of the expansion \eqref{hypexp} and the series \eqref{poinhyp} is that they are always available since $\G$ always has hyperbolic elements and hyperbolic fixed points.
If $\G$ has no cusps then there are no expansions of the form \eqref{expnpar}.
The more familiar parabolic Poincar\'e series, defined next, also requires a cusp for its construction.

\begin{adef} For $m \in \Z$, the  {\em Poincar\'e series $P_{\ca, m}$
associated to the cusp $\ca$}    is defined as
\begin{equation}\label{poinpar}
    P_{\ca,m}(z):=  \sum_{\g \in \G_\ca \backslash \G} e^{2\pi i m z} \left|_k \sa^{-1}\g \right.
    = \sum_{\g \in \G_\ca \backslash \G} \frac{e^{2\pi i m (\sa^{-1}\g z)}}{j(\sa^{-1}\g, z)^k}.
\end{equation}
\end{adef}
This series converges absolutely for $k>2$ with the convergence uniform for $z$ in compact sets in $\H$. We have  $P_{\ca, m} \in S_k(\G)$ for $m \gqs 1$, $P_{\ca, 0} \in M_k(\G)$ and $P_{\ca, m} \in M^!_k(\G)$ if $m \lqs -1$. For $f\in S_k(\G)$ and $m \in \Z_{\gqs 1}$
\begin{equation}\label{pin}
    \s{f}{P_{\ca, m}} = c_\ca(m;f) \left[\frac{ \G(k-1)}{(4\pi m)^{k-1}}\right]
\end{equation}
and the series $P_{\ca, m}$ for fixed $\ca$ and $m \in \Z_{\gqs 1}$ span  $S_k(\G)$.

\subsection{Main results}
In this paper we calculate the parabolic and hyperbolic Fourier expansions of the parabolic and hyperbolic Poincar\'e series. The parabolic Fourier expansion of
$P_{\ca, m}$  for $m \in \Z$ was first found by Poincar\'e himself in \cite{P11} for $\SL_2(\Z)$, see the discussions in \cite{Pr,K10}. This was generalized  by Petersson in \cite{P30,P32} to general groups. The coefficients are expressed as series involving  Kloosterman sums, denoted $S_{\ca\cb}(m,n;C)$, multiplied by Bessel functions. To establish the first instance of the pattern we will see in the other cases, we rewrite the coefficients in terms of the ${_0}F_1$ hypergeometric function. Doing this has the added benefit of making the statement very concise, independent of  the signs of $m$ and $n$. Recall that the general hypergeometric function is given by
\begin{equation} \label{hypfn}
    {_p}F_q(a_1, \dots,a_p;b_1,\dots,b_q;x) := \sum_{n=0}^\infty \frac{(a_1)_n \cdots (a_p)_n}{(b_1)_n \cdots (b_q)_n} \frac{x^n}{n!},
\end{equation}
where $(a)_n:=a(a+1) \cdots (a+n-1)$ and $b_i \not\in \Z_{\lqs 0}$. The series \eqref{hypfn} is absolutely convergent for all $x\in \C$ if $p\lqs q$, and absolutely convergent for all  $|x|<1$ if $p= q+1$. See \cite[Chap. 2]{AAR}.

\begin{theorem}[Poincar\'e, Petersson] \label{CISpp}
For $m$, $n \in \Z$, the $n$th  coefficient in the parabolic Fourier expansion at  $\cb$ of the parabolic Poincar\'e series $P_{\ca,m}$ is given by
\begin{multline}
c_\cb(n;P_{\ca,m})  = \begin{cases}\displaystyle \frac{(2\pi i)^k n^{k-1}}{\G(k)} \sum_{C \in C_{\ca\cb} }
 {_0}F_1 \left( ;k; - \frac{4\pi^2 mn}{C^2} \right)
 \frac{S_{\ca\cb}(m,n;C)}{C^k} \text{ \ \ \ if \ $n \gqs 1$}
 \end{cases}
 \\
 +\begin{cases}
1 \text{ \ \ if \ $m=n$ \ and $\ca \equiv \cb \bmod \G$,}
\end{cases}
 \label{sum3pp}
\end{multline}
where we understand $0$ when a condition is not met.
Here, if $\ca$ and $\cb$ are $\G$-equivalent we choose $\sb=\g\sa$ for some $\g \in \G$ with $ \cb=\g \ca$.
\end{theorem}

See Section \ref{secpp} for all the details. Petersson worked more generally with real weight $k$ and an associated multiplier system.

To describe the parabolic Fourier expansion of the hyperbolic series
$P_{\eta, m}$ we need the following notation.
Put $C_{\eta\ca} := \left\{ac \ \left| \ \left(\smallmatrix a
& b \\ c & d
\endsmallmatrix\right) \in \se^{-1}\G\sa \right. \right\}$. We will see later that $0 \not\in C_{\eta\ca}$.
For $C \in C_{\eta\ca}$ and $\e(z):=e^{2\pi i z}$ define
\begin{equation}\label{kloo}
S_{\eta\ca}(m,n;C):= \sum_{\substack{\g \in \G_\eta \backslash \G / \G_\ca  \\ \left(\smallmatrix a
& b \\ c & d
\endsmallmatrix\right) = \se^{-1}\g\sa, \ ac=C} } \e\left(\frac{m}{\ell_\eta} \log \left|\frac ac\right| + n \left( \frac{b}{2a}+\frac{d}{2c}\right)\right).
\end{equation}
This generalized Kloosterman sum was first identified and studied by Good in \cite{G83}.
Renormalizing \eqref{kloo} by multiplying it by $\exp\left(\pi^2 m  (\sgn(C)-1)/\ell_\eta - \pi i n/C \right)$
gives the  variant
\begin{equation}\label{kloos}
S^\star_{\eta\ca}(m,n;C):= \sum_{\substack{\g \in \G_\eta \backslash \G / \G_\ca  \\ \left(\smallmatrix a
& b \\ c & d
\endsmallmatrix\right) = \se^{-1}\g\sa, \ ac=C} } \e\left(\frac{m}{\ell_\eta} \log \left(\frac ac\right) + n  \frac{b}{a}\right)
\end{equation}
where the logarithm takes its principal value. The next theorem is proved in Section \ref{sechp}.

\begin{theorem} \label{CIShp}
For $m \in \Z$ and $n \in \Z_{\gqs 1}$, the $n$th coefficient in the parabolic Fourier expansion at  $\ca$ of the hyperbolic Poincar\'e series $P_{\eta,m}$ has the formula
\begin{equation} \label{sum3}
c_\ca(n;P_{\eta,m})  =  \frac{(2\pi i)^k n^{k-1}}{\G(k)}\sum_{C \in C_{\eta\ca} }
   {_1}F_1 \left(\frac k2 + \frac{2\pi i m}{\ell_\eta};k; \frac{2\pi i n}{C} \right)  \frac{S^\star_{\eta\ca}(m,n;C)}{ C^{k/2}}.
\end{equation}
\end{theorem}

In the case that  $\G =\SL_2(\Z)$, $\ca=\infty$ and $\eta=(-\sqrt{D},\sqrt{D})$ for $D$  a positive integer that is not a perfect square, we can give a very explicit expression for $S_{\eta\ca}(m,n;C)$. First, choose $\se=\hat\sigma_\eta$ and $\si=I$ so that $C_{\eta\ci} \subset \Z/(2\sqrt{D})$. Let $(a_0,c_0)=(a,c)$ be the minimal positive integer solution
to Pell's equation
\begin{equation}\label{pell}
     a^2-Dc^2=1.
\end{equation}
Such a solution always exists and may be found from the continued fraction expansion of $\sqrt{D}$.
 Set $\varepsilon_D:=a_0+\sqrt{D}c_0$,
$\ell_\eta :=2\log \varepsilon_D$ and
write
\begin{equation*}
    \frac{a_0+ 1}{c_0}=\frac{u_+}{v_+}, \qquad \frac{a_0- 1}{c_0}=\frac{u_-}{v_-}
\end{equation*}
in lowest terms. Also set $D_+:=u_+^2-D v_+^2$, $D_-:=u_-^2-D v_-^2$;
 we will see later that  $D_+>0$ and $D_-<0$. Define
\begin{equation*}
    \psi_D(m,n;N) := \begin{cases} (-1)^{m+c_0 \cdot n} & \quad \text{if} \quad N = D_+ \text{ or } D_-\\
    0 & \quad \text{otherwise}
    \end{cases}
\end{equation*}
and put
\begin{equation} \label{rdnst}
    R^*_D(N) := \Bigl\{  (e,g)\in \Z^2 \ \Big| \ \gcd(e,g)=1, \ e^2-Dg^2=N, \ e^2+Dg^2 \lqs a_0|N|  \Bigr\}.
\end{equation}
See Figure \ref{bfig} for an example of $R^*_D(N)$. The next result is proved in Section \ref{sect_ex}.

\begin{theorem} \label{final_k}
Let $\ci$ be the  cusp and $\eta=(-\sqrt{D},\sqrt{D})$  a hyperbolic fixed pair  for $\SL_2(\Z)$ with scaling matrices $I$  and $\hat\sigma_\eta$ respectively. Then for all $m$, $n \in \Z$
\begin{equation} \label{kloost4}
S_{\eta\infty}\left(m,n;\frac{-N}{2\sqrt{D}}\right) =  -\psi_D(m,n;N)
+ \frac 12
\sum_{(e,g)\in R^*_D(N)}
\e  \left(\frac{m}{\ell_\eta} \log \left|\frac{e+g\sqrt{D}}{e-g\sqrt{D}}\right| -  \frac{n e g^{-1}}{N} \right)
\end{equation}
where $g^{-1}$ denotes the inverse of $g \bmod N$.
If $g=0$ then $N=1$ and we may set $g^{-1} = 0$.
\end{theorem}


\SpecialCoor
\psset{griddots=5,subgriddiv=0,gridlabels=0pt}
\psset{xunit=0.25cm, yunit=0.2cm}
\psset{linewidth=1pt}
\psset{dotsize=4pt 0,dotstyle=*}

\begin{figure}[h]
\begin{center}
\begin{pspicture}(-21,-10)(21,10) 

\savedata{\mydata}[
{{19.8997, 0.}, {19.8605, 0.5588}, {19.7428, 1.1154}, {19.5473,
  1.66759}, {19.2746, 2.2132}, {18.9258, 2.75008}, {18.5023,
  3.2761}, {18.0058, 3.7892}, {17.4383, 4.28734}, {16.8019,
  4.76856}, {16.0992, 5.23096}, {15.333, 5.67272}, {14.5063,
  6.09208}, {13.6223, 6.48741}, {12.6846, 6.85713}, {11.6968,
  7.1998}, {10.6628, 7.51404}, {9.58678, 7.79864}, {8.4729,
  8.05245}, {7.32559, 8.27449}, {6.14936, 8.46387}, {4.94887,
  8.61985}, {3.72884, 8.7418}, {2.4941, 8.82926}, {1.24952,
  8.88188}, {0., 8.89944}, {-1.24952, 8.88188}, {-2.4941,
  8.82926}, {-3.72884, 8.7418}, {-4.94887, 8.61985}, {-6.14936,
  8.46387}, {-7.32559, 8.27449}, {-8.4729, 8.05245}, {-9.58678,
  7.79864}, {-10.6628, 7.51404}, {-11.6968, 7.1998}, {-12.6846,
  6.85713}, {-13.6223, 6.48741}, {-14.5063, 6.09208}, {-15.333,
  5.67272}, {-16.0992, 5.23096}, {-16.8019, 4.76856}, {-17.4383,
  4.28734}, {-18.0058, 3.7892}, {-18.5023, 3.2761}, {-18.9258,
  2.75008}, {-19.2746, 2.2132}, {-19.5473, 1.66759}, {-19.7428,
  1.1154}, {-19.8605, 0.5588}, {-19.8997,
  0.}, {-19.8605, -0.5588}, {-19.7428, -1.1154}, {-19.5473,
-1.66759}, {-19.2746, -2.2132}, {-18.9258, -2.75008}, {-18.5023,
-3.2761}, {-18.0058, -3.7892}, {-17.4383, -4.28734}, {-16.8019,
-4.76856}, {-16.0992, -5.23096}, {-15.333, -5.67272}, {-14.5063,
-6.09208}, {-13.6223, -6.48741}, {-12.6846, -6.85713}, {-11.6968,
-7.1998}, {-10.6628, -7.51404}, {-9.58678, -7.79864}, {-8.4729,
-8.05245}, {-7.32559, -8.27449}, {-6.14936, -8.46387}, {-4.94887,
-8.61985}, {-3.72884, -8.7418}, {-2.4941, -8.82926}, {-1.24952,
-8.88188}, {0., -8.89944}, {1.24952, -8.88188}, {2.4941, -8.82926},
{3.72884, -8.7418}, {4.94887, -8.61985}, {6.14936, -8.46387},
{7.32559, -8.27449}, {8.4729, -8.05245}, {9.58678, -7.79864},
{10.6628, -7.51404}, {11.6968, -7.1998}, {12.6846, -6.85713},
{13.6223, -6.48741}, {14.5063, -6.09208}, {15.333, -5.67272},
{16.0992, -5.23096}, {16.8019, -4.76856}, {17.4383, -4.28734},
{18.0058, -3.7892}, {18.5023, -3.2761}, {18.9258, -2.75008},
{19.2746, -2.2132}, {19.5473, -1.66759}, {19.7428, -1.1154},
{19.8605, -0.5588}, {19.8997, 0.}}
]

\savedata{\mydatb}[
{{21.1896, -9.}, {20.8713, -8.85}, {20.5536, -8.7}, {20.2364, -8.55},
{19.9198, -8.4}, {19.6039, -8.25}, {19.2886, -8.1}, {18.974, -7.95},
{18.6601, -7.8}, {18.347, -7.65}, {18.0347, -7.5}, {17.7232, -7.35},
{17.4126, -7.2}, {17.103, -7.05}, {16.7943, -6.9}, {16.4867, -6.75},
{16.1802, -6.6}, {15.8749, -6.45}, {15.5708, -6.3}, {15.268, -6.15},
{14.9666, -6.}, {14.6667, -5.85}, {14.3684, -5.7}, {14.0717, -5.55},
{13.7768, -5.4}, {13.4838, -5.25}, {13.1928, -5.1}, {12.904, -4.95},
{12.6174, -4.8}, {12.3334, -4.65}, {12.052, -4.5}, {11.7734, -4.35},
{11.4978, -4.2}, {11.2255, -4.05}, {10.9567, -3.9}, {10.6917, -3.75},
{10.4307, -3.6}, {10.1741, -3.45}, {9.9222, -3.3}, {9.67536, -3.15},
{9.43398, -3.}, {9.19851, -2.85}, {8.96939, -2.7}, {8.74714, -2.55},
{8.53229, -2.4}, {8.32541, -2.25}, {8.12712, -2.1}, {7.93804, -1.95},
{7.75887, -1.8}, {7.59029, -1.65}, {7.43303, -1.5}, {7.28783, -1.35},
{7.15542, -1.2}, {7.03651, -1.05}, {6.93181, -0.9}, {6.84197, -0.75},
{6.76757, -0.6}, {6.70914, -0.45}, {6.66708, -0.3}, {6.64172, -0.15},
{6.63325, 0.}, {6.64172, 0.15}, {6.66708, 0.3}, {6.70914,
  0.45}, {6.76757, 0.6}, {6.84197, 0.75}, {6.93181, 0.9}, {7.03651,
  1.05}, {7.15542, 1.2}, {7.28783, 1.35}, {7.43303, 1.5}, {7.59029,
  1.65}, {7.75887, 1.8}, {7.93804, 1.95}, {8.12712, 2.1}, {8.32541,
  2.25}, {8.53229, 2.4}, {8.74714, 2.55}, {8.96939, 2.7}, {9.19851,
  2.85}, {9.43398, 3.}, {9.67536, 3.15}, {9.9222, 3.3}, {10.1741,
  3.45}, {10.4307, 3.6}, {10.6917, 3.75}, {10.9567, 3.9}, {11.2255,
  4.05}, {11.4978, 4.2}, {11.7734, 4.35}, {12.052, 4.5}, {12.3334,
  4.65}, {12.6174, 4.8}, {12.904, 4.95}, {13.1928, 5.1}, {13.4838,
  5.25}, {13.7768, 5.4}, {14.0717, 5.55}, {14.3684, 5.7}, {14.6667,
  5.85}, {14.9666, 6.}, {15.268, 6.15}, {15.5708, 6.3}, {15.8749,
  6.45}, {16.1802, 6.6}, {16.4867, 6.75}, {16.7943, 6.9}, {17.103,
  7.05}, {17.4126, 7.2}, {17.7232, 7.35}, {18.0347, 7.5}, {18.347,
  7.65}, {18.6601, 7.8}, {18.974, 7.95}, {19.2886, 8.1}, {19.6039,
  8.25}, {19.9198, 8.4}, {20.2364, 8.55}, {20.5536, 8.7}, {20.8713,
  8.85}, {21.1896, 9.}}
]

\savedata{\mydatc}[
{{-21.1896, -9.}, {-20.8713, -8.85}, {-20.5536, -8.7}, {-20.2364,
-8.55}, {-19.9198, -8.4}, {-19.6039, -8.25}, {-19.2886, -8.1},
{-18.974, -7.95}, {-18.6601, -7.8}, {-18.347, -7.65}, {-18.0347,
-7.5}, {-17.7232, -7.35}, {-17.4126, -7.2}, {-17.103, -7.05},
{-16.7943, -6.9}, {-16.4867, -6.75}, {-16.1802, -6.6}, {-15.8749,
-6.45}, {-15.5708, -6.3}, {-15.268, -6.15}, {-14.9666, -6.},
{-14.6667, -5.85}, {-14.3684, -5.7}, {-14.0717, -5.55}, {-13.7768,
-5.4}, {-13.4838, -5.25}, {-13.1928, -5.1}, {-12.904, -4.95},
{-12.6174, -4.8}, {-12.3334, -4.65}, {-12.052, -4.5}, {-11.7734,
-4.35}, {-11.4978, -4.2}, {-11.2255, -4.05}, {-10.9567, -3.9},
{-10.6917, -3.75}, {-10.4307, -3.6}, {-10.1741, -3.45}, {-9.9222,
-3.3}, {-9.67536, -3.15}, {-9.43398, -3.}, {-9.19851, -2.85},
{-8.96939, -2.7}, {-8.74714, -2.55}, {-8.53229, -2.4}, {-8.32541,
-2.25}, {-8.12712, -2.1}, {-7.93804, -1.95}, {-7.75887, -1.8},
{-7.59029, -1.65}, {-7.43303, -1.5}, {-7.28783, -1.35}, {-7.15542,
-1.2}, {-7.03651, -1.05}, {-6.93181, -0.9}, {-6.84197, -0.75},
{-6.76757, -0.6}, {-6.70914, -0.45}, {-6.66708, -0.3}, {-6.64172,
-0.15}, {-6.63325, 0.}, {-6.64172, 0.15}, {-6.66708, 0.3}, {-6.70914,
  0.45}, {-6.76757, 0.6}, {-6.84197, 0.75}, {-6.93181,
  0.9}, {-7.03651, 1.05}, {-7.15542, 1.2}, {-7.28783,
  1.35}, {-7.43303, 1.5}, {-7.59029, 1.65}, {-7.75887,
  1.8}, {-7.93804, 1.95}, {-8.12712, 2.1}, {-8.32541,
  2.25}, {-8.53229, 2.4}, {-8.74714, 2.55}, {-8.96939,
  2.7}, {-9.19851, 2.85}, {-9.43398, 3.}, {-9.67536, 3.15}, {-9.9222,
  3.3}, {-10.1741, 3.45}, {-10.4307, 3.6}, {-10.6917,
  3.75}, {-10.9567, 3.9}, {-11.2255, 4.05}, {-11.4978,
  4.2}, {-11.7734, 4.35}, {-12.052, 4.5}, {-12.3334, 4.65}, {-12.6174,
   4.8}, {-12.904, 4.95}, {-13.1928, 5.1}, {-13.4838,
  5.25}, {-13.7768, 5.4}, {-14.0717, 5.55}, {-14.3684,
  5.7}, {-14.6667, 5.85}, {-14.9666, 6.}, {-15.268, 6.15}, {-15.5708,
  6.3}, {-15.8749, 6.45}, {-16.1802, 6.6}, {-16.4867,
  6.75}, {-16.7943, 6.9}, {-17.103, 7.05}, {-17.4126, 7.2}, {-17.7232,
   7.35}, {-18.0347, 7.5}, {-18.347, 7.65}, {-18.6601, 7.8}, {-18.974,
   7.95}, {-19.2886, 8.1}, {-19.6039, 8.25}, {-19.9198,
  8.4}, {-20.2364, 8.55}, {-20.5536, 8.7}, {-20.8713,
  8.85}, {-21.1896, 9.}}
]

\psline[linecolor=gray]{->}(-21,0)(21.5,0)
\psline[linecolor=gray]{->}(0,-10)(0,10.5)
\psline[linecolor=gray](-0.15,2)(0.15,2)
\psline[linecolor=gray](-0.15,-2)(0.15,-2)
\multirput(-20,-0.3)(2,0){21}{\psline[linecolor=gray](0,0)(0,0.6)}
\multirput(-0.2,-8)(0,2){9}{\psline[linecolor=gray](0,0)(0.4,0)}

\dataplot[linecolor=orange,linewidth=0.8pt,plotstyle=line]{\mydata}
\dataplot[linecolor=red,linewidth=0.8pt,plotstyle=line]{\mydatb}
\dataplot[linecolor=red,linewidth=0.8pt,plotstyle=line]{\mydatc}

\rput(10,5.9){$(13,5)$}
\rput(4.8,2){$(7,1)$}

\psdots(7,1)(13,5)(-7,1)(-13,5)(7,-1)(13,-5)(-7,-1)(-13,-5)

\end{pspicture}
\caption{The eight elements of $R_5^*(44)$}\label{bfig}
\end{center}
\end{figure}
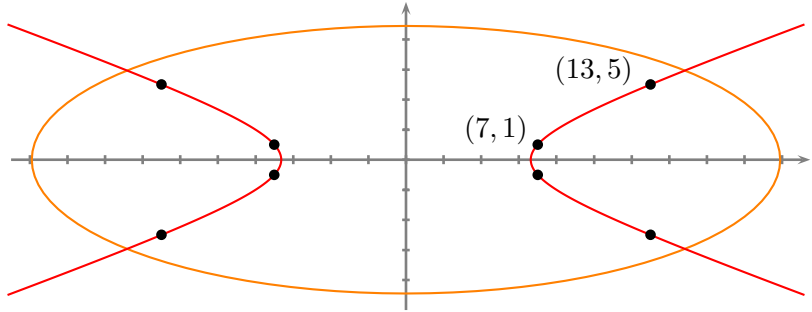



In Theorems \ref{CISph} and \ref{CISph2} of Section \ref{secph}  we also give the hyperbolic expansion of $P_{\ca,m}$, which is similar to Theorem \ref{CIShp}.
Section \ref{num} explores our formulas numerically with the result \eqref{dl2}  calculated there.
Finally,  in Section \ref{sechh} the hyperbolic expansion of $P_{\eta,m}$ is obtained. For this, first
put
\begin{equation}\label{cetet}
    C_{\eta\eta'} := \Bigl\{ad \ \Bigl| \ \left(\smallmatrix a
& b \\ c & d
\endsmallmatrix\right) \in \se^{-1}\G\sep, \ abcd \neq 0 \Bigr\}.
\end{equation}
When $C \in C_{\eta\eta'}$ and $\alpha=\pm1$  define
\begin{equation}\label{klooshh}
S^\star_{\eta\eta'}(m,n;C,\alpha):= \sum_{\substack{\g \in \G_\eta \backslash \G / \G_{\eta'}  \\ \left(\smallmatrix a
& b \\ c & d
\endsmallmatrix\right) = \se^{-1}\g\sep, \ ad=C, \ \sgn(ac)=\alpha} } \e\left(\frac{m}{\ell_\eta} \log \left(\frac ac\right) +\frac{n}{\ell_{\eta'}} \log \left(-\frac cd\right) \right).
\end{equation}
For $\eta=(\eta_1,\eta_2)$, let $\eta^*$ be the reversed pair $(\eta_2,\eta_1)$. It is easy to see that if $\se$ is a scaling matrix for $\eta$, then $\se S$ is a possible scaling matrix for $\eta^*$ where $S:=
\left(\smallmatrix
0 & -1 \\ 1 & 0
\endsmallmatrix\right)$. Also we recall the beta function $B(u,v) := \G(u)\G(v)/\G(u+v)$.

\begin{theorem} \label{CIShh}
For any $m$, $n \in \Z$, the $n$th  coefficient in the hyperbolic expansion at $\eta'$ of the hyperbolic Poincar\'e series $P_{\eta,m}$  is given by
\begin{align}\label{mull}
    c_{\eta'}(n;P_{\eta,m}) = \frac{e^{2\pi^2 n/\ell_{\eta'}}}{\ell_{\eta'}}
    & B\left( \frac{k}{2} - \frac{2\pi i n}{\ell_{\eta'}},\right.  \hspace{-1.5mm}\left.\frac{k}{2} + \frac{2\pi i n}{\ell_{\eta'}}\right)
    \left(\sum\nolimits_1+\sum\nolimits_2+\sum\nolimits_3\right)
    \\
    + & \begin{cases}
 (a^2)^{2\pi i n/\ell_{\eta'}}  \text{ \ if \ }\eta' \equiv \eta \bmod \G  \text{ and } n=m
 \end{cases} \label{xab1q}\\
  + & \begin{cases}
 (-1)^{k/2} e^{2\pi^2 n/\ell_{\eta'}}(b^2)^{-2\pi i n/\ell_{\eta'}}   \text{ \ if \ }\eta' \equiv \eta^* \bmod \G  \text{ and }  n=-m
 \end{cases} \label{yab1q}
\end{align}
where $\sum\nolimits_1$, $\sum\nolimits_2$ and $\sum\nolimits_3$ are given by
\begin{gather*}
    \sum_{C \in C_{\eta\eta'}, \ C\not\in (0,1)}
     {_2}F_1\left( \frac{k}{2} - \frac{2\pi i m}{\ell_{\eta}}, \frac{k}{2} + \frac{2\pi i n}{\ell_{\eta'}};k;\frac{1}{C}\right)
     \frac{S^\star_{\eta\eta'}(m,n;C,1) + S^\star_{\eta\eta'}(m,n;C,-1)}{C^{k/2}},
     \\
  \sum_{C \in C_{\eta\eta'} \cap (0,1)}
     {_2}F_1\left( \frac{k}{2} - \frac{2\pi i m}{\ell_{\eta}}, \frac{k}{2} + \frac{2\pi i n}{\ell_{\eta'}};k;\frac{1}{C}\right)
     \frac{S^\star_{\eta\eta'}(m,n;C,1)}{C^{k/2}},
     \\
      \sum_{C \in C_{\eta\eta'} \cap (0,1)}
     \left(\frac {C}{C-1}\right)^{2\pi i n/\ell_{\eta'}}
     {_2}F_1\left( \frac{k}{2} + \frac{2\pi i m}{\ell_{\eta}}, \frac{k}{2} + \frac{2\pi i n}{\ell_{\eta'}};k;\frac{-1}{C-1}\right)
     \frac{S^\star_{\eta\eta'}(m,n;C,-1)}{(C-1)^{k/2}},
\end{gather*}
respectively.  The sums $\sum\nolimits_2$ and $\sum\nolimits_3$ are finite. The real numbers $a$ and $b$ in \eqref{xab1q}, \eqref{yab1q} depend on the choice of scaling matrices $\se$ and $\sep$.
\end{theorem}

\begin{remark}{\rm
With specific choices of $\se$ and $\sep$ we can make $a$ and $b$ in \eqref{xab1q}, \eqref{yab1q} explicit. For example, suppose $\eta$ and $\eta^*$ are not $\G$-equivalent. If  $\eta' =  \rho \eta$ for some $\rho \in \G$ put $\sep =  \rho \se$ and if $\eta' =  \rho \eta^*$ for some $\rho \in \G$ put $\sep =  \rho \se S$. Then  \eqref{xab1q}, \eqref{yab1q} become
\begin{align*} 
    + & \begin{cases}
 1  \text{ \ if \ }\eta' \equiv \eta \bmod \G  \text{ and } n=m
 \end{cases}\\
  + & \begin{cases}
 (-1)^{k/2} e^{2\pi^2 n/\ell_{\eta'}}  \text{ \ if \ }\eta' \equiv \eta^* \bmod \G  \text{ and }  n=-m
 \end{cases} 
\end{align*}
respectively. See Proposition \ref{trpr} for the proof of this and the general result.}
\end{remark}

In Subsection \ref{last} we test Theorem \ref{CIShh} numerically. We also  show there that a simple special case of the theorem allows us to naturally detect  when the negative Pell equation
\begin{equation}\label{npell}
    x^2-Dy^2=-1 \qquad (\text{$D \in \Z_{\gqs 1}$, non-square})
\end{equation}
has integer solutions.

\section{Good's generalized Kloosterman sums} \label{good}
The Kloosterman sums that arise in all the cases we need are covered by Good's theory as described in \cite{G83}.
Following his notation,
let $\xi$ and $\chi$ each denote either a cusp such as $\ca$ or a hyperbolic fixed pair $\eta$. If the object we are defining is independent of the particular cusp or hyperbolic fixed pair we sometimes write {\em par}  or {\em hyp}, respectively, instead.  For $M=\left(\smallmatrix
a & b \\ c & d
\endsmallmatrix\right) \in \SL_2(\R)$ define the functions $\sideset{_\xi}{_{\chi}}{\oplambda}(M)$  as follows:
\begin{align*}
    \sideset{_\text{par}}{_\text{par}}{\oplambda}(M) & := \frac ac &
    \sideset{_\text{par}}{_\text{hyp}}{\oplambda}(M) & :=  \frac a{2c} + \frac b{2d}
    \\
    \sideset{_\text{hyp}}{_\text{par}}{\oplambda}(M) & := \log \left|\frac ac \right| &
     \sideset{_\text{hyp}}{_\text{hyp}}{\oplambda}(M) & := \frac 12\log \left|\frac{ab}{cd} \right| .
\end{align*}
Good parameterized his sums with $\sideset{_\xi}{_{\chi}}{\opv}(M)$, $\sideset{_\xi}{_{\chi}}{\opd}(M)$ and $\sideset{_\xi}{_{\chi}}{\opd'}(M)$, defined as
\begin{align*}
    \sideset{_\text{par}}{_\text{par}}{\opv}(M) & := |c| &
    \sideset{_\text{par}}{_\text{par}}{\opd}(M) & := 0 &
    \sideset{_\text{par}}{_\text{par}}{\opd'}(M) & := 0
    \\
    \sideset{_\text{hyp}}{_\text{par}}{\opv}(M) & := \left|2 ac \right|^{1/2} &
    \sideset{_\text{hyp}}{_\text{par}}{\opd}(M) & := \frac{1-\sgn(ac)}2 &
    \sideset{_\text{hyp}}{_\text{par}}{\opd'}(M) & := 0
    \\
    \sideset{_\text{par}}{_\text{hyp}}{\opv}(M) & :=  \left|2 cd \right|^{1/2} &
    \sideset{_\text{par}}{_\text{hyp}}{\opd}(M) & := 0 &
    \sideset{_\text{par}}{_\text{hyp}}{\opd'}(M) & := \frac{1+\sgn(cd)}2
    \\
    \sideset{_\text{hyp}}{_\text{hyp}}{\opv}(M) & := |ad|^{1/2}+ |bc|^{1/2} &
    \sideset{_\text{hyp}}{_\text{hyp}}{\opd}(M) & := \frac{1-\sgn(ac)}2 &
    \sideset{_\text{hyp}}{_\text{hyp}}{\opd'}(M) & := \frac{1+\sgn(cd)}2 .
\end{align*}
The functions $\sideset{_\xi}{_{\chi}}{\oplambda}(M)$ and $\sideset{_\xi}{_{\chi}}{\opv}(M)$ are derived from the geometry of the fixed points of $\SL_2(\R)$ in $\H$ and double coset decompositions of $\SL_2(\R)$, see \cite[Sect. 3]{G83}.  The Iwasawa and Bruhat decompositions are generalized in Lemma 1 of \cite{G83}. The four  cases of this Lemma   we  need are given explicitly in our Lemmas \ref{brupp}, \ref{bruhp}, \ref{bruph} and \ref{bruhh}.

Let $\ell_\eta$ be as in \eqref{sgsh} and put  $\ell_\ca:=1$ for any cusp $\ca$.
For $\delta$, $\delta' \in \{0,1\}$ define the generalized Kloosterman sum, \cite[Eq. (5.10)]{G83}, as
\begin{equation}\label{kloost}
\sideset{_\xi^{\delta}}{_{\chi}^{\delta'}}{\opS}(m,n;\nu):= \sum_{\genfrac{}{}{0pt}{0}{\g \in \G_\xi \backslash \G / \G_\chi}{M = \sigma_{\xi}^{-1}\g\sigma_{\chi}}} \e\left(\frac{m}{\ell_\xi} \sideset{_\xi}{_{\chi}}{\oplambda}(M) - \frac{n}{\ell_\chi} \sideset{_\chi}{_{\xi}}{\oplambda}(M^{-1})\right)
\end{equation}
where the sum is restricted to $M$ such that
$$
   \sideset{_\xi}{_{\chi}}{\opv}(M)=\nu, \quad \sideset{_\xi}{_{\chi}}{\opd}(M)=\delta, \quad \sideset{_\xi}{_{\chi}}{\opd'}(M)=\delta'.
$$
The usual Kloosterman sum corresponds to the parabolic/parabolic combination $\sideset{_\ca^{0}}{_{\cb}^{0}}{\opS}(m,n;\nu)$  in \eqref{kloost}, see Subsections \ref{klpp} and \ref{numa}. We use the  three other families of sums with parabolic and hyperbolic combinations in our Fourier expansions in Sections \ref{sechp}, \ref{secph} and \ref{sechh}.
Including the elliptic case gives five further combinations which Good also fit into the formalism \eqref{kloost}.

In \cite{G83} these generalized Kloosterman sums are required for  the Fourier expansions of the non-holomorphic Poincar\'e series
\begin{equation*}
    P_\xi(z,s,m):= \sum_{\g \in \G_\xi\backslash \G} V_\xi(\sigma_\xi^{-1} \g z,s,m/\ell_\xi)
\end{equation*}
for $z \in \H$ and $\Re(s)>1$ where
\begin{align*}
    V_{par}(z,s,\lambda) & := \frac 1i \int_{-z}^{-\overline{z}} \e(-\lambda \rho) \left( \frac{y}{(\rho+z)(\rho+\overline{z})}\right)^{1-s} \, d\rho,\\
    V_{hyp}(z,s,\lambda) & := \frac 1i \int_{-\log z}^{-\overline{\log z}} \e(-\lambda \rho) \left( \frac{2y e^\rho}{(z e^\rho -1)(\overline{z}e^\rho -1)}\right)^{1-s} \, d\rho \qquad (\Re(z)>0).
\end{align*}
These  series are constructed to be  eigenfunctions of the hyperbolic Laplacian:
\begin{equation*}
    \Delta P_\xi(z,s,m) = -s(1-s)P_\xi(z,s,m) \qquad \text{for} \qquad \Delta := y^2 \left( \frac{\partial^2}{\partial x^2} + \frac{\partial^2}{\partial y^2}  \right), \quad z=x+iy.
\end{equation*}
See \cite[Sect. 7]{G83} for the details.

\section{Parabolic Poincar\'e series and their parabolic Fourier expansions} \label{secpp}
Let  $\ca$ and $\cb$ be two  cusps for $\G$ and let $m$ and $n$ be any two integers.
In this section we give a detailed review of the  computation of the coefficients $c_\cb(n; P_{\ca,m})$ in the parabolic Fourier expansion of $P_{\ca,m}$ at $\cb$:
$$
\left(P_{\ca,m}|_k \sb\right)(z) = \sum_{n\in \Z} c_\cb(n; P_{\ca,m}) e^{2\pi i n z}.
$$
See for example \cite[Chap. 2, 3]{IwTo} and \cite[Chap. 5]{R77} for similar treatments.   Sections \ref{sechp}, \ref{secph} and \ref{sechh}  will extend these calculations to the  cases when $\ca$ or $\cb$ equals $\eta$ or $\eta'$.
We also remark that in \cite{Pr} the Fourier expansion is computed for a very general kind of parabolic
Poincar\'e series with complex `weight' and separate multiplier system.

\subsection{An integral for the parabolic/parabolic case}
For $m$, $n \in \Z$ and $r \in \R_{\neq 0}$ define
\begin{equation}\label{ipp}
    I_{par \, par}(m,n;r) := \int_{-\infty+iy}^{\infty+iy} \e\left(-\frac{m}{r^2 u} - nu\right) u^{-k}  \, du \qquad (y>0, \ k>1).
\end{equation}
This is the integral we will  need shortly, see \eqref{inffpp} in the proof of Theorem \ref{CISpp}, and we study it here first.

\begin{prop}
The integral \eqref{ipp} is absolutely convergent. For an implied constant depending only on $k>1$,
\begin{align}
    I_{par \, par}(m,n;r) & = 0   &(n \lqs 0), \label{ippx1}\\
    I_{par \, par}(m,n;r) & \ll n^{(k-1)/2}  \exp\left(2\pi n^{1/2} \left(1+\frac{|m|-m}{2r^2} \right) \right)  & (n >0), \label{ippx2}\\
    I_{par \, par}(m,n;r) & \ll n^{k-1}  & (m, \ n >0). \label{ippx3}
\end{align}
\end{prop}
\begin{proof}
Bounding the absolute value of the integrand in \eqref{ipp} when $u=x+iy$ shows
\begin{equation}\label{gresb}
    |I_{par \, par}(m,n;r)| \lqs \exp\left(2\pi ny + \frac{\pi(|m|-m)}{r^2 y}\right) \int_{-\infty}^\infty \frac{dx}{(x^2+y^2)^{k/2}}.
\end{equation}
Clearly, the right side of \eqref{gresb} converges for $k>1$. Since the integrand is holomorphic, \eqref{ipp} is independent of $y>0$. Letting $y \to \infty$ in \eqref{gresb} shows \eqref{ippx1}.
A special case of \cite[3.251.11]{GR} implies
\begin{equation}\label{gres}
    \int_{-\infty}^\infty \frac{dx}{(x^2+y^2)^{s}} = \sqrt{\pi} \frac{\G(s-1/2)}{\G(s)} y^{1-2s} \qquad(\Re(s)>1/2).
\end{equation}
Using \eqref{gres} in \eqref{gresb} with $y=1/\sqrt{n}$  and  $y=1/n$ proves \eqref{ippx2} and \eqref{ippx3} respectively.
\end{proof}

Next we evaluate $I_{par \, par}(m,n;r)$ in terms of the  hypergeometric function ${_0}F_1(;b;z)$. Recall that for each $b\not\in \Z_{\lqs 0}$ it is a holomorphic function of $z\in \C$.

\begin{prop} \label{wpro}
Let $k \in \R_{>1}$. For all $m \in \Z$ and $n \in \Z_{\gqs 1}$
\begin{equation}\label{ipp3}
    I_{par \, par}(m,n;r) =  \frac{(2\pi)^k n^{k-1}}{e^{\pi i k/2}\G(k)} \ {_0}F_1 \left(;k; - \frac{4\pi^2 mn}{r^2} \right).
\end{equation}
\end{prop}
\begin{proof}
The formula \eqref{ipp3} follows directly by a change of variables from
\begin{equation}\label{diri}
    {_0}F_1 \left(;b; z \right) = \frac{e^{\pi i b/2} \G(b)}{(2\pi)^b} \int_{-\infty+it}^{\infty+it} \e\left( -u + \frac{z}{4\pi^2 u}\right) u^{-b} \, du \qquad(t>0, \ \Re(b)>1).
\end{equation}
We can establish \eqref{diri} by linking it to the integral representation of the $J$-Bessel function  in \cite[8.412.2]{GR}. Provided that $\Re(b)>1$, we may deform the contour of integration in \cite[8.412.2]{GR} to a vertical line with positive real part. Multiplying the variable by $i$ then produces
\begin{equation}\label{jbe}
    J_{b-1}(2z) = \frac{z^{b-1}}{2\pi}e^{\pi i b/2} \int_{-\infty+it}^{\infty+it} \e\left( -\frac{u}{2\pi}-\frac{z^2}{2\pi u}\right) u^{-b} \, du \qquad(t>0).
\end{equation}
See also \cite[p. 156]{R77}.
The $J$-Bessel function may be expressed in terms of hypergeometric functions:
\begin{align}\label{jb1}
    J_{b-1}(2z) & = \frac{1}{\G(b)} z^{b-1} {_0}F_1 \left(;b; - z^2 \right)\\
    & = \frac{1}{\G(b)} z^{b-1} e^{-2i z} {_1}F_1 \left(b-\frac 12;  2b-1; 4iz \right) \label{jb2}
\end{align}
as in \cite[p. 200]{AAR}. Formulas  \eqref{jbe} and \eqref{jb1} together prove \eqref{diri}.
\end{proof}

The referee has pointed out that Proposition \ref{wpro} may also be quickly shown  by combining the power series expansion
for \ $\e\left(-\frac{m}{r^2 u} \right)$ with Hankel's formula for the Gamma function, as seen in \cite[p. 156]{R77},
\begin{equation*}
    \int_{-\infty+iy}^{\infty+iy} \e\left( - nu\right) u^{-k}  \, du = \frac{(2\pi)^k n^{k-1}}{e^{\pi i k/2} \G(k)} \qquad (y,n>0, \ k>1).
\end{equation*}

\subsection{Double cosets in the parabolic/parabolic case}
Let $L$ be a complete set of inequivalent representatives for $\G_\ca\backslash \G/\G_\cb$. Partition $L$ into two sets:
\begin{equation*}
    \G(\ca,\cb)_0 := \Big\{ \delta \in L \ \Big| \ \delta \cb = \ca \Big\}, \qquad \G(\ca,\cb) := \Big\{ \delta \in L \ \Big| \ \delta \cb \neq \ca \Big\}.
\end{equation*}
It is easy to see that $\G(\ca,\cb)_0$ has at most one element.

\begin{prop} \label{pp_reps}
With the above notation, a complete set of inequivalent representatives for $\G_\ca\backslash \G$ is given by
\begin{equation} \label{jiu}
    \G(\ca,\cb)_0  \cup \Big\{ \delta \tau \ \Big| \ \delta \in \G(\ca,\cb), \ \tau \in \G_\cb/Z  \Big\}.
\end{equation}
\end{prop}
\begin{proof}
The set $L':=\{ \delta\tau \ | \ \delta \in L, \ \tau \in \G_\cb/Z  \}$ clearly gives a complete set of representatives for $\G_\ca\backslash \G$, but some of its elements may be equivalent modulo $\G_\ca$. Suppose
\begin{equation}\label{jul}
    \G_\ca\delta\tau = \G_\ca \delta'\tau' \quad \text{ for } \quad \delta, \delta' \in L  \quad \text{and}  \quad \tau, \tau' \in \G_\cb/Z.
\end{equation}
  We must have $\delta'=\delta$ because $L$ is defined as a set of inequivalent representatives. Hence there is a $\g \in \G_\ca$ so that $\g \delta\tau = \delta\tau'$. It follows that $\g$ fixes $\ca$ and $\delta \cb$ which can only happen if $\g=\pm 1$ or if $\delta \cb = \ca$.

If $\g=\pm 1$ then $\tau=\tau'$. If $\delta \cb = \ca$ then
$\G_\cb = \delta^{-1} \G_\ca \delta$ and any $\tau \in \G_\cb$ may be written as $\delta^{-1} \g \delta$ for $\g \in \G_\ca$. Therefore, for all $\tau \in \G_\cb$, $\G_\ca\delta\tau = \G_\ca\delta (\delta^{-1} \g \delta) = \G_\ca\delta$. We have shown that \eqref{jul} implies $\delta'=\delta$, and then $\tau=\tau'$ or $\delta \cb = \ca$ and $\G_\ca\delta\tau = \G_\ca \delta'\tau' = \G_\ca \delta$. Hence, with \eqref{jiu}, we have removed all of the equivalent elements from the set $L'$ we started with.
\end{proof}

We may also characterize the sets $\G(\ca,\cb)_0$ and $\G(\ca,\cb)$ with
\begin{align*}
    \G(\ca,\cb)_0  & = \left\{  \delta  \in L \ \left| \ \sa^{-1} \delta \sb = \begin{pmatrix} a  & b \\ c & d \end{pmatrix} \text{ with } c = 0 \right.\right\}, \\
    \G(\ca,\cb)  & = \left\{  \delta  \in L \ \left| \ \sa^{-1} \delta \sb = \begin{pmatrix} a  & b \\ c & d \end{pmatrix} \text{ with } c\neq 0 \right.\right\}
\end{align*}
since
\begin{equation*}
    c=0 \iff \sa^{-1} \delta \sb \infty = \infty \iff  \delta \sb \infty = \sa \infty \iff  \delta \cb = \ca.
\end{equation*}

To describe an example of $\G(\ca,\cb)$ more explicitly, we first recall the Bruhat decomposition in the form given by \cite[Lemma 1]{G83}.
\begin{lemma} \label{brupp}
For $M=\left(\smallmatrix
a & b \\ c & d
\endsmallmatrix\right) \in \SL_2(\R)$ with $c \ne 0$,
\begin{equation} \label{bru}
    \begin{pmatrix} a & b \\ c & d \end{pmatrix} = \sgn(c) \begin{pmatrix} 1 & a/c \\ 0 & 1 \end{pmatrix} \begin{pmatrix} 0 & -1 \\ 1 & 0 \end{pmatrix}
    \begin{pmatrix} \nu & 0 \\ 0 & 1/\nu \end{pmatrix} \begin{pmatrix} 1 & d/c \\ 0 & 1 \end{pmatrix}
\end{equation}
for $\nu=\sideset{_\text{par}}{_\text{par}}{\opv}(M)=|c|$.
\end{lemma}
We see that multiplying \eqref{bru} on the left   by
$\left(\smallmatrix 1 & \ell \\ 0 & 1 \endsmallmatrix\right)$ changes $a/c$ to $\ell+a/c$ and leaves $c$ and $d$ fixed. Similarly, multiplying on the right  by
$\left(\smallmatrix 1 & \ell \\ 0 & 1 \endsmallmatrix\right)$ changes $d/c$ to $\ell+d/c$ and leaves $a$ and $c$ fixed.
Define
\begin{equation*} \label{recab}
R_{\ca\cb}  :=  \left\{ \left. \begin{pmatrix} a  & b \\ c & d \end{pmatrix} \in \sa^{-1}\G \sb \ \right|  \ c\neq 0, \ 0 \lqs  \frac ac  < 1, \ 0 \lqs \frac dc<1   \right\}.
\end{equation*}

\begin{lemma} \label{ac01pp}
We may take $\sa^{-1}\G(\ca,\cb)\sb = R_{\ca\cb}/Z$.
\end{lemma}
\begin{proof}
Let $B=\left\{ \left. \left(\smallmatrix 1 & \ell \\ 0 & 1 \endsmallmatrix\right) \ \right|  \ \ell \in \Z   \right\}$ and suppose that $-I \not\in \G$. Then
\begin{equation*}
    \sa^{-1}(\G_\ca\backslash \G/\G_\cb)\sb = \sa^{-1}\G_\ca\sa \backslash \sa^{-1}\G\sb/ \sb^{-1}\G_\cb \sb = B\backslash \sa^{-1}\G\sb/ B.
\end{equation*}
It follows that $R_{\ca\cb}$ gives a complete set of representatives for $\sa^{-1}\G(\ca,\cb)\sb$.  Suppose that two elements
$\left(\smallmatrix a & b \\ c & d \endsmallmatrix\right)$,
$\left(\smallmatrix a' & b' \\ c' & d' \endsmallmatrix\right)$
of $R_{\ca\cb}$ are equivalent, i.e. $\left(\smallmatrix 1 & \ell \\ 0 & 1 \endsmallmatrix\right)
\left(\smallmatrix a & b \\ c & d \endsmallmatrix\right)
\left(\smallmatrix 1 & \ell' \\ 0 & 1 \endsmallmatrix\right) = \left(\smallmatrix a' & b' \\ c' & d' \endsmallmatrix\right)$. Then $c=c'$ and also $a=a'$, $d=d'$. This proves the lemma when $-I \not\in \G$. If $-I \in \G$ then $\sa^{-1}\G_\ca\sa =  \sb^{-1}\G_\cb\sb = -B \cup B$. Hence $\left(\smallmatrix a & b \\ c & d \endsmallmatrix\right)$ and $\left(\smallmatrix -a & -b \\ -c & -d \endsmallmatrix\right)$ are now equivalent in $R_{\ca\cb}$.
\end{proof}

We also note that if $\left(\smallmatrix a & * \\ c & * \endsmallmatrix\right) \in R_{\ca\cb}$ then $b$ and $d$ are uniquely determined. To see this, suppose $
\g=\left(\smallmatrix a & * \\ c & d \endsmallmatrix\right)$ and $\g'=\left(\smallmatrix a & * \\ c & d' \endsmallmatrix\right)$ are in $R_{\ca\cb}$. Then
$\left(\smallmatrix a & * \\ c & d \endsmallmatrix\right)^{-1} \left(\smallmatrix a & * \\ c & d' \endsmallmatrix\right)= \left(\smallmatrix * & * \\ 0 & * \endsmallmatrix\right) \in \sb^{-1}\G\sb$ and we must have $\g=\g'$. Similarly, if $\left(\smallmatrix * & * \\ c & d \endsmallmatrix\right) \in R_{\ca\cb}$ then $a$ and $b$ are uniquely determined.


\subsection{Kloosterman sums} \label{klpp}

Put $$C_{\ca\cb}:=\Bigl\{|c| \ : \ \left(\smallmatrix a
& b \\ c & d
\endsmallmatrix\right) \in \sa^{-1}\G\sb, \ c\neq 0  \Bigr\}.
$$
 We use $|c|$ instead of $c$ here because it is convenient that $\left(\smallmatrix a
& b \\ c & d
\endsmallmatrix\right)$ and $\left(\smallmatrix -a
& -b \\ -c & -d
\endsmallmatrix\right)$ (if it is in $\sa^{-1}\G\sb$) have the same representative. (We could also have used $c^2$, making the parameter a product of two matrix elements as we do in Sections \ref{sechp}, \ref{secph} and \ref{sechh}, but this goes against the conventional notation.)
For $C \in C_{\ca\cb}$  the Kloosterman sum
\begin{equation}\label{kloopp}
S_{\ca\cb}(m,n;C):= \sum_{\substack{\g \in \G_\ca \backslash \G / \G_\cb \\ \left(\smallmatrix a
& b \\ c & d
\endsmallmatrix\right) = \sa^{-1}\g\sb, \ |c|=C} } \e\left(m\frac{a}{c} + n\frac{d}{c} \right)
\end{equation}
is well defined.
Since $c \neq 0$ we could equivalently  have summed over $\g \in \G(\ca,\cb)$.  In Good's notation \eqref{kloost}, we have $S_{\ca\cb}(m,n;C)= \sideset{_\ca^{0}}{_{\cb}^{0}}{\opS}(m,n;C)$. See also \cite[Eq. (3.13)]{IwTo}, for example.
Note that the sum $S_{\ca\cb}(m,n;C)$ depends on the choice of scaling matrices $\sa$ and $\sb$ in a simple way; we assume the choice is fixed for each cusp. Replacing $\g$ by $\g^{-1}$ in \eqref{kloopp} shows
\begin{equation*}
    S_{\ca\cb}(m,n;C)  =  S_{\cb\ca}(-n,-m;C) = \overline{S_{\cb\ca}(n,m;C)}.
\end{equation*}

Now let $\mathcal N_{\ca\cb}(C) := S_{\ca\cb}(0,0;C)$ be the number of terms in the sum \eqref{kloopp}. Then $\mathcal N_{\ca\cb}(C)$ is always finite and in fact, by \cite[Prop. 2.8]{IwTo},
\begin{equation}\label{klbl}
    \sum_{C \in C_{\ca\cb}, \ C \lqs X} C^{-1} \mathcal N_{\ca\cb}(C) \ll X.
\end{equation}
From \eqref{klbl} we  deduce the bounds
\begin{align}\label{trivk}
    \mathcal N_{\ca\cb}(C) & \ll C^2, \\
    S_{\ca\cb}(m,n;C) & \ll C^2, \label{trivka}\\
    \#\{C \in C_{\ca\cb} \ : \ C \lqs X\}  & \ll X^2 \label{trivkb}
\end{align}
with implied constants depending only on $\G$, $\ca$ and $\cb$.

\subsection{The parabolic expansion of $P_{\ca,m}$}

\begin{proof}[Proof of Theorem \ref{CISpp}]
With $z=x+iy \in \H$, use Proposition \ref{pp_reps} to write
\begin{multline} \label{iiipp}
  \left(P_{\ca,m}|_k \sb\right)(z) = \sum_{\g \in \G_\ca \backslash \G }
    \frac{\e(m(\sa^{-1}\g \sb z))}
    {j(\sa^{-1}\g \sb, z)^{k}}  \\
    = \sum_{\delta \in \G(\ca,\cb)_0} \frac{\e(m(\sa^{-1}\delta \sb z))}
    {j(\sa^{-1}\delta \sb, z)^{k}} + \sum_{\g \in \G(\ca,\cb)} \sum_{\tau \in \G_\cb/Z} \frac{\e(m(\sa^{-1}\g \tau \sb z))}
    {j(\sa^{-1}\g \tau\sb, z)^{k}}
    .
\end{multline}
The first sum in \eqref{iiipp} is just $\e(mz)$ if $\G(\ca,\cb)_0$ is non-empty, which happens exactly when $\ca$ and $\cb$ are $\G$-equivalent.
Write the second sum as
\begin{equation} \label{remppx}
   \sum_{C \in C_{\ca\cb}}
    \sum_{\substack{ \g \in \G(\ca,\cb) \\ \left(\smallmatrix a
& b \\ c & d
\endsmallmatrix\right) = \sa^{-1}\g\sb,
\
 |c|=C }}
\sum_{\tau \in \G_\cb/Z} \frac{\e(m(\sa^{-1}\g\sb \cdot \sb^{-1}\tau \sb z))}
    {j(\sa^{-1}\g \sb \cdot \sb^{-1}\tau \sb, z)^{k}}
\end{equation}
and the inner sum is
\begin{equation} \label{intgpp}
   \sum_{n\in \Z}
    \frac{\e(m(\sa^{-1}\g \sb(z+n)))}
    {j(\sa^{-1}\g \sb , z+n)^{k}}.
\end{equation}
Since
$
\sa^{-1}\g\sb = \left(\smallmatrix a
& b \\ c & d
\endsmallmatrix\right)$ is in $\SL_2(\R)
$ with $c \neq 0$
we have
\begin{equation}\label{gsz}
\sa^{-1}\g\sb z = \frac ac - \frac 1{c(cz+d)}
\end{equation}
and so \eqref{intgpp} equals
\begin{equation} \label{inffpp}
\sum_{n\in \Z} f(n) \qquad \text{for} \qquad
    f(t):=\frac{\e\left( m\left(\frac ac - \frac 1{c^2(z+t+d/c)}\right)\right)}
    {c^k(z+t+d/c)^{k}}.
\end{equation}
Poisson summation gives
\begin{equation} \label{psf}
    \sum_{n\in \Z} f(n) = \sum_{n\in \Z} \hat f(n) \qquad \text{for} \qquad \hat f(n):=\int_{-\infty}^\infty f(t) \e(-nt)\, dt.
\end{equation}
To check this is valid, we may use
the convenient form of \cite[Thm. A, p. 71]{Ra}, which requires $f$ to be twice continuously differentiable on $\R$ and that $\int_{-\infty}^\infty f(t)\, dt$ and $\int_{-\infty}^\infty |f''(t)|\, dt$ exist.
It is straightforward to check that for $k>1$ our $f$ in \eqref{inffpp} meets these conditions. Hence \eqref{psf} holds.

Substituting $u=z+t+d/c$
  and recalling \eqref{ipp} shows that \eqref{remppx} is now
\begin{equation} \label{rempp}
   \sum_{C \in C_{\ca\cb}}
    \sum_{\substack{ \g \in \G(\ca,\cb) \\ \left(\smallmatrix a
& b \\ c & d
\endsmallmatrix\right) = \sa^{-1}\g\sb,
\
 |c|=C }}
\sum_{n \in \Z} e^{2 \pi i n z}\frac 1{ C^{k}} \e\left(m\frac{a}{c} + n\frac{d}{c} \right) I_{\ca\cb}(m,n;C).
\end{equation}
Taking absolute values and using \eqref{ippx1}, \eqref{ippx2}, we find that \eqref{rempp} is majorized by
\begin{equation*}
    \sum_{C \in C_{\ca\cb}}
    \sum_{\substack{ \g \in \G(\ca,\cb) \\ \left(\smallmatrix a
& b \\ c & d
\endsmallmatrix\right) = \sa^{-1}\g\sb,
\
 |c|=C }}
\sum_{n = 1}^\infty e^{-2 \pi  n y} \frac{n^{(k-1)/2}}{C^{k}}  \exp\left(2\pi n^{1/2} \left(1+\frac{|m|-m)}{C^2} \right) \right) \ll \sum_{C \in C_{\ca\cb}} \frac{\mathcal N_{\ca\cb}(C)}{C^{k}}.
\end{equation*}
With \eqref{klbl} and summation by parts, this last is convergent for $k>2$. Changing the order of summation in \eqref{rempp}  is now justified and formula \eqref{ipp3}
 completes the proof.
\end{proof}

\section{Hyperbolic Poincar\'e series and their parabolic Fourier expansions} \label{sechp}
Let $\ca$ be a cusp  and $\eta$ a hyperbolic fixed point pair for $\G$.
In this section we compute coefficients in the parabolic Fourier expansion of $P_{\eta,m}$ at $\ca$:
$$
\left(P_{\eta,m}|_k \sa\right)(z) = \sum_{n=1}^\infty c_\ca(n; P_{\eta,m}) e^{2\pi i n z}.
$$

\subsection{The hyperbolic/parabolic integral}

For $m$, $n \in \Z$ and $r \in \R_{\neq 0}$, the integral we will need is
\begin{equation}\label{ihp}
    I_{\eta \, par}(m,n;r) := \int_{-\infty+iy}^{\infty+iy} \frac{\left( \sgn(r) \frac {u-r}{u+r}\right)^{2\pi i m /\ell_\eta}e^{-2\pi i n u}}
    {( u-r)^{k/2} ( u+r)^{k/2}}  \, du \qquad (y>0, \ k>1).
\end{equation}

\begin{prop} \label{bndhp}
The integral \eqref{ihp} is absolutely convergent.
We   have $I_{\eta \, par}(m,n;r)=0$ for $n\lqs 0$ and
\begin{equation} \label{ad2b}
    I_{\eta \, par}(m,n;r) \ll  n^{k-1} \exp\left(\pi^2 (|m|-m)/\ell_\eta \right) \qquad (n>0)
\end{equation}
for an implied constant depending only on $k>2$.
\end{prop}
\begin{proof}
Notice that $w:=\sgn(r) \frac {u-r}{u+r} \in \H$ and for $y=\Im(u)$  we have the bound
\begin{equation*}
     \left|w^{2\pi i m /\ell_\eta}e^{-2\pi i n u}\right| \lqs \exp\left(\pi^2 (|m|-m)/\ell_\eta+2\pi ny\right).
\end{equation*}
It follows that \eqref{ihp} is absolutely convergent for $k>1$. If we assume $k>2$, write $u-r=x+iy$, note that $|u+r|^{-k/2} \lqs y^{-k/2}$ and
 recall \eqref{gres}, then
\begin{align}
    |I_{\eta \, par}(m,n;r)| & \lqs \exp\left(\pi^2 (|m|-m)/\ell_\eta+2\pi ny\right) y^{-k/2} \int_{-\infty}^\infty \frac{dx}{(x^2+y^2)^{k/4}} \notag\\
    & \ll \exp\left(\pi^2 (|m|-m)/\ell_\eta+2\pi ny\right) y^{1-k}.  \label{yinf}
\end{align}
Since the integrand in \eqref{ihp} is holomorphic,
$I_{\eta \, par}(m,n;r)$ is independent of $y>0$.
Letting $y \to \infty$ in \eqref{yinf} we see that $I_{\eta \, par}(m,n;r)=0$ for $n \leqslant 0$. For $n>0$ let $y =1/n$.
\end{proof}

Bounds for $I_{\eta \, par}(m,n;r)$ when $k \in (1,2]$ are of course possible. The advantage of \eqref{ad2b} for $k>2$ is that it does not depend on $r$; see \eqref{fsumhp3}.

If $m=0$ we can evaluate $I_{\eta \, par}(0,n;r)$ for $n \gqs 1$ by moving the line of integration down past the poles of order $k/2$  at $u= \pm r$ and letting $y \to -\infty$. Evaluating the residues we find for $n \gqs 1$ and even $k \gqs 4$,
\begin{equation}\label{iet}
I_{\eta \, par}(0,n;r) = 2\pi i (-1)^{k/2} \sum_{j=0}^{k/2-1}\frac{(2\pi i n)^{j}}{j!} \binom{k-2-j}{k/2-1}  \left[\frac{e^{-2\pi i n r} }{(2r)^{k-1-j}} + \frac{e^{2\pi i n r} }{(-2r)^{k-1-j}}\right].
\end{equation}
More generally, we may  express $I_{\eta \, par}(m,n;r)$ in terms of the confluent hypergeometric function ${_1}F_1$.

\begin{prop} \label{usop}
Let $k \in \R_{>1}$. For $m \in \Z$ and $n \in \Z_{\gqs 1}$, $I_{\eta\, par}(m,n;r)$  equals
\begin{equation} \label{coul}
     \frac{(2\pi)^k n^{k-1}}{e^{\pi i k/2} \G(k) } \exp\left(\frac{\pi^2 m}{\ell_\eta}(\sgn(r)-1) -2\pi i nr\right)
     \, {_1}F_1 \left(\frac{k}{2}+\frac{2\pi i m}{\ell_\eta};k;4\pi i n r\right).
\end{equation}
Also $I_{\eta\, par}(m,n;r)$ is real-valued for $k$ even.
\end{prop}
\begin{proof}
From \cite[3.384.8]{GR}, we will use the formula
\begin{equation}\label{gri}
    \int_{-\infty}^\infty (\beta+ix)^{-\mu}(\g+ix)^{-\nu} e^{-ipx}\, dx = \frac{2\pi e^{\g p}(-p)^{\mu+\nu-1}}{\G(\mu+\nu)} \, {_1} F_1(\mu;\mu+\nu;(\beta-\g)p)
\end{equation}
 where $\Re(\beta), \Re(\g) >0$,  $\Re(\mu+\nu)>1$  and $p<0$. Rewrite \eqref{ihp}
by letting $u=x+iy$, multiplying through by $i$ and replacing $x$ by $-x$ to get
\begin{equation} \label{coul2}
    I_{\eta \, par}(m,n;r) =  \frac{e^{2\pi ny}}{e^{\pi i k/2}} \int_{-\infty}^{\infty} \frac{\left( \sgn(r) \frac {\g+ix}{\beta+ix}\right)^{2\pi i m /\ell_\eta}e^{2\pi i n x}}
    {(\beta+ix)^{k/2} ( \g+ix)^{k/2}}  \, dx
\end{equation}
for $\beta=y-ir$, $\g=y+ir$ and $\Re(\beta)=\Re(\g)=y>0$. The final step to get \eqref{coul2} into the form \eqref{gri} is to verify by checking the arguments of both sides, that for $r$, $\beta$ and $\g$ as above,
\begin{equation*}
    \left( \sgn(r) \frac {\g+ix}{\beta+ix}\right)^s = e^{\pi i s (1-\sgn(r))/2} (\g+ix)^s (\beta+ix)^{-s} \qquad (x \in \R, \ s\in \C).
\end{equation*}

To see that $I_{\eta\, par}(m,n;r) \in \R$ for $k$ even, use Kummer's transformation
\begin{equation}\label{kummer}
    {_1}F_1(a;c;z)=e^z {_1}F_1(c-a;c;-z),
\end{equation}
to show that the last part of the right side of \eqref{coul},
 \begin{equation} \label{coul3}
    e^{-2\pi i nr}
     \, {_1}F_1(k/2+2\pi i m/\ell_\eta;k;4\pi i n r),
\end{equation}
equals its conjugate. For a second proof, we see that the integral in \eqref{coul2} is real by using the fact that $\g =\overline{\beta}$ and replacing $x$ with $-x$. (We note that \eqref{coul3} takes exactly the form of
  a Coulomb wave function,  used to describe charged particles with a spherically symmetric potential as in \cite[p. 199]{AAR}.)
\end{proof}

Using \eqref{jb2} in \eqref{coul} when $m=0$ shows another version of \eqref{iet}:
\begin{equation*}
    I_{\eta \, par}(0,n;r) = (2\pi )^k \frac{\G((k+1)/2)}{e^{\pi i k/2} \G(k)} \left(\frac{n}{\pi r}\right)^{(k-1)/2} J_{(k-1)/2}(2\pi n r) \qquad(n>0).
\end{equation*}

\subsection{Double cosets in the hyperbolic/parabolic case} \label{sec42}

\begin{lemma}
If $\left(\smallmatrix a
& b \\ c & d
\endsmallmatrix\right) \in \se^{-1}\G\sa$ then $ac\neq 0$.
\end{lemma}
\begin{proof}
Let $\left(\smallmatrix a
& b \\ c & d
\endsmallmatrix\right) = \se^{-1}\g\sa$. Since $\left(\smallmatrix a
& b \\ c & d
\endsmallmatrix\right) \infty = \frac{a}{c}$ we have
\begin{align*}
    ac=0  \iff & \se^{-1}\g\sa \infty = 0 \text{ \ or \ } \infty \\
    \iff & \g \sa \infty = \se 0 \text{ \ or \ } \se\infty \\
    \iff & \g\ca = \eta_1 \text{ \ or \ } \eta_2.
\end{align*}
But the cusp $\g \ca$ cannot be a hyperbolic fixed point, implying $ac\neq 0$.
\end{proof}

Since $\g \ca$ cannot be a hyperbolic fixed point, the analog of $\G(\ca,\cb)_0$ in Proposition \ref{pp_reps} is empty here.
Let $\G(\eta, \ca)$ be a complete set of inequivalent representatives for $\G_\eta\backslash \G/\G_\ca$.

\begin{prop} \label{doubhp} With the above notation, a complete set of inequivalent representatives for $\G_\eta\backslash \G$ is given by
\begin{equation} \label{hpjiu}
    \Big\{ \delta \tau \ \Big| \ \delta \in \G(\eta,\ca), \ \tau \in \G_\ca/Z  \Big\}.
\end{equation}
\end{prop}
\begin{proof}
The set \eqref{hpjiu} gives a complete set of representatives for $\G_\eta\backslash \G$. To see that the representatives are also inequivalent modulo $\G_\eta$, suppose
\begin{equation}\label{hpjul}
    \G_\eta\delta\tau = \G_\eta \delta'\tau' \quad \text{ for } \quad \delta, \delta' \in \G(\eta, \ca)  \quad \text{and}  \quad \tau, \tau' \in \G_\ca/Z.
\end{equation}
  We must have $\delta'=\delta$ because $\G(\eta, \ca)$ is defined as a set of inequivalent representatives. Hence there is a $\g \in \G_\eta$ so that $\g \delta\tau = \delta\tau'$. It follows that $\g$ fixes $\eta$ and $\delta \ca$. Therefore $\g =\pm I$ and $\tau=\tau'$.
\end{proof}

 Good's Lemma \cite[Lemma 1, p 20]{G83} in this hyperbolic/parabolic case is:

\begin{lemma} \label{bruhp}
For $M=\left(\smallmatrix a
& b \\ c & d
\endsmallmatrix\right) \in \SL_2(\R)$ with $ac \neq 0$ we have
\begin{equation} \label{tms}
    \begin{pmatrix} a & b \\ c & d \end{pmatrix} = \frac{\sgn(a)}{\sqrt{2}}\begin{pmatrix} \left| \frac ac \right|^{1/2} & 0 \\ 0 & \left| \frac ac \right|^{-1/2} \end{pmatrix}
    \begin{pmatrix} 1 & -\sgn(ac)  \\ \sgn(ac)  & 1  \end{pmatrix}
    \begin{pmatrix} \nu & 0 \\ 0 & 1/\nu \end{pmatrix}\begin{pmatrix} 1 & \frac{b}{2a}+\frac{d}{2c} \\ 0 & 1 \end{pmatrix}.
\end{equation}
for $\nu = \sideset{_\text{hyp}}{_\text{par}}{\opv}(M) = |2ac|^{1/2}$.
\end{lemma}

For a convenient choice of   $\G(\eta,\ca)$, our representatives for $\G_\eta\backslash \G/\G_\ca$,
we therefore define
\begin{equation*} \label{rec}
R_{\eta\ca}  :=  \left\{ \left. \begin{pmatrix} a  & b \\ c & d \end{pmatrix} \in \se^{-1}\G \sa \ \right|  \ \frac{1}{\varepsilon_\eta} \lqs \left| \frac ac \right| < \varepsilon_\eta, \ 0 \lqs \frac{b}{2a}+\frac{d}{2c} <1   \right\}.
\end{equation*}

\begin{lemma} \label{ac01}
We may take $\se^{-1} \G(\eta,\ca)\sa =  R_{\eta\ca} /Z$.
\end{lemma}
\begin{proof}
Let
$A=\left\{ \left. \left(\smallmatrix \varepsilon^m & 0 \\ 0 & \varepsilon^{-m} \endsmallmatrix\right) \ \right|  \ m \in \Z   \right\}$
for $\varepsilon = \varepsilon_\eta$ and let $B=\left\{ \left. \left(\smallmatrix 1 & \ell \\ 0 & 1 \endsmallmatrix\right) \ \right|  \ \ell \in \Z   \right\}$ as before.
Suppose that $-I \not\in \G$. Then
\begin{equation*}
    \se^{-1}(\G_\eta\backslash \G/\G_\ca)\sa = \se^{-1}\G_\eta\se \backslash \se^{-1}\G\sa/ \sa^{-1}\G_\ca \sa = A\backslash \se^{-1}\G\sa/ B.
\end{equation*}
Start with any $\left(\smallmatrix a
& b \\ c & d
\endsmallmatrix\right) \in \se^{-1}\G \sa $.
If we multiply on the left  by
$\left(\smallmatrix \varepsilon & 0 \\ 0 & 1/\varepsilon \endsmallmatrix\right)$
we obtain
$\left(\smallmatrix a\varepsilon  & b \varepsilon  \\ c/\varepsilon & d/\varepsilon \endsmallmatrix\right)$ so that $|a/c|$ becomes $\varepsilon^2 |a/c|$ and $b/2a+d/2c$ is unaffected. Multiplying $\left(\smallmatrix a
& b \\ c & d
\endsmallmatrix\right)$ on the right by
$\left(\smallmatrix 1 & 1 \\ 0 & 1 \endsmallmatrix\right)$
produces
$\left(\smallmatrix a  & a+b   \\ c & c+d \endsmallmatrix\right)$. Then $b/2a+d/2c$ becomes $b/2a+d/2c + 1$ and $a/c$ remains the same.
It follows that every element of $\se^{-1}\G(\eta,\ca)\sa$ has a  representative in $R_{\eta\ca}$ and, as in the proof of Lemma \ref{ac01pp}, the representative is unique.

If $-I \in \G$ then $-I \in \G_\eta$, $\G_\ca$ so that $\left(\smallmatrix a
& b \\ c & d
\endsmallmatrix\right)$ and
$\left(\smallmatrix -a & -b \\ -c & -d \endsmallmatrix\right)$ are now equivalent in $R_{\eta\ca}$.
\end{proof}

The reasoning after Lemma \ref{ac01pp}
 also shows that if $\left(\smallmatrix a & * \\ c & * \endsmallmatrix\right) \in R_{\eta\ca}$ then $b$ and $d$ are uniquely determined.

\subsection{The hyperbolic/parabolic Kloosterman sum} \label{sec43}

Recall the definition of  $S_{\eta\ca}(m,n;C)$ in  \eqref{kloo}.
It is related to Good's sum \eqref{kloost} by
\begin{equation*}
    S_{\eta\ca}(m,n;C) = \sideset{_\eta^{\delta}}{_{\ca}^{0}}{\opS}(m,n;|2C|^{1/2})  \quad \text{for} \quad \delta = \frac{1-\sgn(C)}{2}.
\end{equation*}
Good showed in \cite[Lemma 6]{G83} that these are finite sums. In this subsection we reprove this and find the analog of the  bound \eqref{trivka}.
First
set
\begin{align*}
    \mathcal N_{\eta\ca}(C) & :=  S_{\eta \ca}(0,0;C) \\
     & = \#\left\{ \g \in \G_\eta \backslash \G / \G_{\ca} \ \left| \ \se^{-1}\g\sa = \begin{pmatrix} a  & b \\ c & d \end{pmatrix} \text{ with } ac = C\right. \right\}.
\end{align*}
Then $\mathcal N_{\eta\ca}(C)$ is well defined and independent of the scaling matrices $\se$ and $\sa$. The next proposition is based on \cite[Prop. 2.8]{IwTo}. It requires the existence of  $M_{\ca\ca}>0$ such that $|c| \gqs M_{\ca\ca}$ for all $c\in C_{\ca\ca}$. For this see \cite[Lemma 1.25]{S71} or \cite[p. 38]{IwTo}.

\begin{prop} \label{schp}
With the above notation
\begin{equation*}
    \sum_{C \in C_{\eta\ca}, \ |C| \lqs X}  \mathcal N_{\eta\ca}(C) \ll X.
\end{equation*}
\end{prop}
\begin{proof}
We may write $\mathcal N_{\eta\ca}(C)$ more explicitly as $\#\left\{ \left. \left(\smallmatrix
a & b \\ c & d
\endsmallmatrix\right) \in R_{\eta\ca}/Z \ \right| \ ac = C\right\}$. Also let
\begin{equation*}
    R(X):=\left\{ \left. \begin{pmatrix} a  & b \\ c & d \end{pmatrix}  \in R_{\eta\ca} \ \right| \ |ac| \lqs X \right\}
    \subset \se^{-1}\G\sa.
\end{equation*}
Suppose $\g=\left(\smallmatrix
a & b \\ c & d
\endsmallmatrix\right)$ and $\delta=\left(\smallmatrix
a' & b' \\ c' & d'
\endsmallmatrix\right)$  are in $R(X)$. Then $\g^{-1} \delta = \left(\smallmatrix
a'' & b'' \\ c'' & d''
\endsmallmatrix\right) \in \sa^{-1}\G\sa$ for
\begin{equation*}
    |c''|=\left| cc'\left(\frac{a}{c}-\frac{a'}{c'}\right) \right|.
\end{equation*}
If $c''=0$ then $\g^{-1} \delta \in \sa^{-1}\G_\ca \sa$  and so $\g=\delta$. Otherwise we have $|c''| \gqs M_{\ca\ca}>0$. Hence
\begin{equation} \label{topbhp}
    \left| \frac{a}{c}-\frac{a'}{c'}\right|  \gqs \frac{M_{\ca\ca}}{|c c'|}.
\end{equation}
For any $\left(\smallmatrix
a & b \\ c & d
\endsmallmatrix\right) \in R(X)$ we have
\begin{equation*}
    \frac{1}{\varepsilon_\eta} \lqs \left| \frac ac \right| < \varepsilon_\eta, \quad |ac| \lqs X \quad \implies \quad \frac{1}{|c|} \gqs \frac{1}{\varepsilon_\eta^{1/2} X^{1/2}}.
\end{equation*}
Therefore \eqref{topbhp} implies
\begin{equation} \label{topbhp2}
    \left| \frac{a}{c}-\frac{a'}{c'}\right|  \gqs \frac{M_{\ca\ca}}{\varepsilon_\eta X}.
\end{equation}

Since each element of $R(X)$ corresponds to a distinct $a/c \in [-\varepsilon_\eta, \varepsilon_\eta]$ with the distance between any two bounded from below by \eqref{topbhp2}, the set $R(X)$ is finite and we may list the fractions  as $a_1/c_1<a_2/c_2< \dots < a_n/c_n$.
 Then
\begin{equation*}
    \sum_{j=1}^{n-1} \left| \frac{a_{j+1}}{c_{j+1}}-\frac{a_{j}}{c_{j}}\right|
    = \sum_{j=1}^{n-1} \left(\frac{a_{j+1}}{c_{j+1}}-\frac{a_{j}}{c_{j}}  \right)
    \lqs 2 \varepsilon_\eta
\end{equation*}
and combining this with \eqref{topbhp2} shows
\begin{equation*}
    \sum_{C \in C_{\eta\ca}, \ |C| \lqs X} \Biggl[ \sum_{\left(\smallmatrix
a & b \\ c & d
\endsmallmatrix\right) \in R_{\eta\ca} ,\, ac = C} \frac{ M_{\ca\ca}}{\varepsilon_\eta X} \Biggr] \lqs 2 \varepsilon_\eta . \qedhere
\end{equation*}
\end{proof}

As a result of Proposition \ref{schp}, for  implied constants depending only on $\G$, $\eta$ and $\ca$,
\begin{equation} \label{impa}
    \mathcal N_{\eta\ca}(C) \ll C, \qquad S_{\eta\ca}(m,n;C) \ll C, \qquad \#\{C \in C_{\eta\ca} \ : \ |C| \lqs X\}   \ll X.
\end{equation}

\subsection{The parabolic expansion of $P_{\eta,m}$}

\begin{theorem} \label{CIShp2}
For $m$, $n \in \Z$, the $n$th parabolic Fourier coefficient at the cusp $\ca$ of the hyperbolic Poincar\'e series $P_{\eta,m}$ is given by
\begin{equation}
c_\ca(n;P_{\eta,m})  =  \sum_{C \in C_{\eta\ca} }
C^{-k/2} I_{\eta\ca}\left(m,n;\frac{1}{2C}\right) \cdot S_{\eta\ca}(m,n;C)
. \label{sum3b}
\end{equation}
\end{theorem}
\begin{proof}
With Proposition \ref{doubhp} and $z=x+iy \in \H$, write the absolutely convergent
\begin{equation*}
    \left(P_{\eta,m}|_k \sa\right)(z) = \sum_{\g \in \G_\eta \backslash \G }
    \frac{(\se^{-1}\g \sa z)^{-k/2+2\pi i m /\ell_\eta}}
    {j(\se^{-1}\g \sa, z)^{k}}
\end{equation*}
as
\begin{equation} \label{fsumhp}
    \sum_{C \in C_{\eta\ca}}
    \sum_{\substack{ \g \in \G(\eta,\ca) \\ \left(\smallmatrix a
& b \\ c & d
\endsmallmatrix\right) = \se^{-1}\g\sa,
\
 ac=C }}
    \sum_{n\in \Z}
    \frac{(\se^{-1}\g \sa(z+n))^{-k/2+2\pi i m /\ell_\eta}}
    {j(\se^{-1}\g \sa , z+n)^{k}}.
\end{equation}
If we let
$
\se^{-1}\g\sa = \left(\smallmatrix a
& b \\ c & d
\endsmallmatrix\right) \in \SL_2(\R)
$ with $ac \neq 0$, then the inner sum is $\sum_{n\in \Z}f_\g(n)$ for
\begin{equation*}
    f_\g(t) = f(t):=\frac{\left( \frac{a(z+t)+b}{c(z+t)+d} \right)^{-k/2+2\pi i m /\ell_\eta}}{(c(z+t)+d)^k} = \frac{\left(\frac ac - \frac 1{c^2(z+t+d/c)}\right)^{-k/2+2\pi i m /\ell_\eta}}
    {c^k(z+t+d/c)^{k}}.
\end{equation*}
As in the proof of Theorem \ref{CISpp}, we may apply Poisson summation if
$\int_{-\infty}^\infty f(t)\, dt$ and $\int_{-\infty}^\infty |f''(t)|\, dt$  exist. The first integral exists for $k>1$ by Proposition \ref{bndhp} with $n=0$. It follows that the second exists too since
differentiating logarithmically  shows
\begin{equation*}
    f''(t)=f(t)\left( \frac{(s+k)(s+k+1)c^2}{(c(z+t)+d)^2} -\frac{2s(s+k)a c}{(a(z+t)+b)(c(z+t)+d)}+\frac{s(s-1)a^2}{(a(z+t)+b)^2}\right)
\end{equation*}
where $s=-k/2+2\pi i m /\ell_\eta$.
Therefore
\begin{equation} \label{psfhp}
    \sum_{n\in \Z} f(n) = \sum_{n\in \Z} \hat f(n) = \sum_{n\in \Z} \int_{-\infty}^\infty f(t) \e(-nt)\, dt.
\end{equation}
For the most symmetric result we substitute
\begin{equation*}
    u  = z+t+\frac{b}{2a}+\frac{d}{2c}
       = z+t+\frac{d}{c} - \frac{1}{2ac}
\end{equation*}
and find that the integral in \eqref{psfhp} equals
\begin{multline} \label{intg2}
    \e\left( n \left(z+ \frac{b}{2a}+\frac{d}{2c}\right)\right) \int_{-\infty+iy}^{\infty+iy} \frac{\left(\frac ac - \frac 1{c^2(u+1/(2ac))}\right)^{-k/2+2\pi i m /\ell_\eta}}
    {c^k (u+1/(2ac))^{k}} e^{-2\pi i n u} \, du\\
    = \frac{\e\left( n \left(z+ \frac{b}{2a}+\frac{d}{2c}\right)\right)} {(ac)^{k/2}}\int_{-\infty+iy}^{\infty+iy} \frac{\left( \frac ac  \frac {u-1/(2ac)}{u+1/(2ac)}\right)^{2\pi i m /\ell_\eta}e^{-2\pi i n u}}
    {( u-1/(2ac))^{k/2} (u+1/(2ac))^{k/2}}  \, du\\
    = e^{2 \pi i n z}\frac{1}{ (ac)^{k/2}}
    \e\left(\frac{m}{\ell_\eta} \log \left|\frac ac\right| + n \left( \frac{b}{2a}+\frac{d}{2c}\right)\right)
    I_{\eta\ca}\left(m,n;\frac{1}{2ac}\right).
\end{multline}
Therefore \eqref{fsumhp} is now
\begin{equation} \label{fsumhp2}
    \sum_{C \in C_{\eta\ca}}
    \sum_{\substack{ \g \in \G(\eta,\ca) \\ \left(\smallmatrix a
& b \\ c & d
\endsmallmatrix\right) = \se^{-1}\g\sa,
\
 ac=C }}
    \sum_{n\in\Z}
    e^{2 \pi i n z}\frac{1}{ C^{k/2}}
    \e\left(\frac{m}{\ell_\eta} \log \left|\frac ac\right| + n \left( \frac{b}{2a}+\frac{d}{2c}\right)\right)
    I_{\eta\ca}\left(m,n;\frac{1}{2C}\right).
\end{equation}
By Proposition \ref{bndhp}, \eqref{fsumhp2} is majorized by
\begin{equation} \label{fsumhp3}
    \sum_{C \in C_{\eta\ca}}
    \sum_{\substack{ \g \in \G(\eta,\ca) \\ \left(\smallmatrix a
& b \\ c & d
\endsmallmatrix\right) = \se^{-1}\g\sa,
\
 ac=C }}
    \sum_{n=1}^\infty
    e^{-2 \pi  n y}\frac{n^{k-1}}{ |C|^{k/2}}
     \exp\left(\pi^2 (|m|-m)/\ell_\eta \right) \ll \sum_{C \in C_{\eta\ca}} \frac{\mathcal N_{\eta\ca}(C)}{ |C|^{k/2}}.
\end{equation}
Using Proposition \ref{schp} and summation by parts shows the last series is convergent for $k>2$. Changing the order of summation in \eqref{fsumhp2} is therefore valid,
and moving the sum over $n$ to the outside completes the proof.
\end{proof}

Theorem \ref{CIShp} follows from Theorem \ref{CIShp2} and Proposition \ref{usop}.

\section{An example in $\G =\SL_2(\Z)$} \label{sect_ex}
\subsection{}
Set $\G =\SL_2(\Z)$. We  consider the results of the last section in the special case where
\begin{equation*}
    \ca=\infty \quad \text{and} \quad \eta=(\eta_1,\eta_2)=(-\sqrt{D},\sqrt{D})
\end{equation*}
for $D$  a positive integer that is not a perfect square.
If $\g=\left(\smallmatrix a
& b \\ c & d
\endsmallmatrix\right) \in \SL_2(\R)$ fixes $\pm \sqrt{D}$ then $cz^2+(d-a)z-b$ has $z=\pm \sqrt{D}$ as its zeros. Therefore $d-a=0$ and $b/c=D$ so that $\g=\left(\smallmatrix
a & Dc \\ c & a
\endsmallmatrix\right)$.
If $\g \in \G$ then
 $(a,c)$ is an integer solution of Pell's equation \eqref{pell}.
Let $(a_0,c_0)$ be the positive integer solution of \eqref{pell} minimizing $a>1$. Set $\varepsilon_D := a_0+\sqrt{D} c_0$ and we see that $\varepsilon_D>1$. Choose $\si= I$ and $\se = \hat\sigma_\eta$ as in \eqref{hypscat}.  Then
\begin{equation} \label{d14s}
    \sigma_\eta= \frac{1}{\sqrt{2}}\begin{pmatrix} D^{1/4} & -D^{1/4} \\ D^{-1/4} & D^{-1/4} \end{pmatrix}, \qquad
    \sigma_\eta^{-1} \begin{pmatrix} a & Dc \\ c & a \end{pmatrix} \sigma_\eta = \begin{pmatrix} a+\sqrt{D} c & 0 \\ 0 & a-\sqrt{D} c \end{pmatrix}
\end{equation}
so that
\begin{equation} \label{deth}
    \sigma_\eta^{-1} \G_\eta \sigma_\eta = \left\langle \begin{pmatrix} \varepsilon_D  & 0 \\ 0 & 1/\varepsilon_D \end{pmatrix}, \ -I\right\rangle
    = \left\{ \left. \pm \begin{pmatrix} \varepsilon_D^n  & 0 \\ 0 & \varepsilon_D^{-n} \end{pmatrix} \right| n \in \Z \right\}.
\end{equation}
(The picture for general hyperbolic points of $\SL_2(\Z)$ is not much different from the above. See, for example \cite[Sect. 3.1]{KZ}.)

For $\left(\smallmatrix
e  &  f \\ g & h
\endsmallmatrix\right) \in \G$, write
\begin{equation*}
    \se^{-1}\begin{pmatrix} e  &  f \\ g & h \end{pmatrix}\si =  \frac{1}{\sqrt{2} D^{1/4}}\begin{pmatrix} e+g\sqrt{D}   &  f+h\sqrt{D}  \\  -e+g\sqrt{D} &  -f+h\sqrt{D}  \end{pmatrix} = \begin{pmatrix} a  &  b \\ c & d \end{pmatrix}.
\end{equation*}
Then $ac=(e^2-Dg^2)/(-2\sqrt{D})$ and $C_{\eta\ci} \subset \Z/(2\sqrt{D})$. Also
\begin{equation*}
    \frac{b}{2a}+\frac{d}{2c} = \frac{f+h\sqrt{D}}{2(e+g\sqrt{D})} - \frac{f-h\sqrt{D}}{2(e-g\sqrt{D})} = \frac{ef-ghD}{e^2-g^2D}.
\end{equation*}
Set $R_D  := \se R_{\eta\ci} \si^{-1}$, so that
\begin{equation}
       R_D= \left\{ \left. \begin{pmatrix} e  &  f \\ g & h \end{pmatrix}\in \SL_2(\Z) \ \right|  \ \frac 1{\varepsilon_D} \lqs  \left|\frac{e+g\sqrt{D} }{e-g\sqrt{D} }\right| < \varepsilon_D, \ 0\lqs \frac{ef-ghD}{e^2-g^2D} < 1 \right\} \label{R_D}
\end{equation}
and
let $R_D(N)$ be the elements of $R_D$ with $e^2-Dg^2=N$.
 Combining   Proposition \ref{usop}, Lemma \ref{ac01} and Theorem \ref{CIShp2} with the work above
shows  the following.

\begin{prop} \label{ci3}
For $m \in \Z$ and $n \in \Z_{\gqs 1}$
\begin{multline*}
c_\infty(n;P_{\eta,m})  =
 \frac{(2\pi i)^k n^{k-1}}{\G(k) }
 \sum_{N\in \Z_{\neq 0} }
\left(\frac{-2\sqrt{D}}
{ N}\right)^{k/2}
 \exp\left(-\frac{\pi^2 m}{\ell_\eta}(\sgn(N)+1) +\frac{2\pi i n\sqrt{D}}{N}\right)
 \\
     \times
      {_1}F_1 \left(\frac{k}{2}+\frac{2\pi i m}{\ell_\eta};k;-\frac{4\pi i n\sqrt{D} }{N}\right)
S_{\eta\infty}\left(m,n;\frac{-N}{2\sqrt{D}}\right)
\end{multline*}
for
\begin{equation} \label{sei}
S_{\eta\infty}\left(m,n;\frac{-N}{2\sqrt{D}}\right) =
 \frac 12
\sum_{\left(\smallmatrix e & f \\ g & h \endsmallmatrix\right) \in R_D(N) }
\e\left(\frac{m}{\ell_\eta} \log \left|\frac{e+g\sqrt{D}}{e-g\sqrt{D}}\right| +   \frac{n(ef-ghD)}{N}  \right).
\end{equation}
\end{prop}

\subsection{A more explicit form for $S_{\eta\infty}\left(m,n;-N/(2\sqrt{D})\right)$}
Recall the statement of Theorem \ref{final_k} and the definitions preceding it.
To prove this result, we begin by
examining the sets $R_D$ and $R_D(N)$ in more detail.

Given  $e$, $g \in \Z$ with $(e,g)=1$, how many ways, if any, can we complete the matrix $\left(\smallmatrix e  & * \\ g & *
\endsmallmatrix\right)$ to an element of $R_D$?
If $g=0$ then it can be quickly seen that the only way to complete $\left(\smallmatrix e  & * \\ 0 & *
\endsmallmatrix\right)$ is to $\pm I \in R_D(1)$.
For $g\neq 0$  write $\overline{e}$ for the inverse of $e \bmod |g|$, chosen with $0\lqs \overline{e}<|g|$ say. We find the solution
\begin{equation} \label{efghsol}
    (e,f_0,g,h_0)=\left(e,\frac{e\overline{e}-1}{g},g,\overline{e}\right) \qquad \text{to} \qquad eh-fg=1
\end{equation}
 and any other solution must be of the form $(e,f_0+te,g,h_0+tg)$ for $t\in \Z$. With these solutions, the second condition in $R_D$, \eqref{R_D},  becomes
\begin{equation*}
    0\lqs \frac{ef_0-gh_0 D}{e^2-g^2D} +t < 1 .
\end{equation*}
It follows from the above arguments that  there is at most one way to complete  $\left(\smallmatrix e  & * \\ g & *
\endsmallmatrix\right)$ to an element of $R_D$. The next result gives the summands in \eqref{sei} in terms of just $e$ and $g$.

\begin{lemma}
For $g \neq 0$
\begin{equation} \label{elg}
    \e\left( \frac{n(ef-ghD)}{N}  \right) = \e\left( -\frac{n e g^{-1}}{N}  \right)
\end{equation}
and for $g=0$ the left side of \eqref{elg} is $1$.
\end{lemma}
\begin{proof}
When $g=0$ then we must have $e=\pm 1$ and $N=1$ so that $ef-ghD \equiv 0 \bmod N$.
When $g \neq 0$, we have from \eqref{efghsol} that
\begin{equation} \label{elg2}
    ef-ghD \equiv ef_0-gh_0D   \equiv  e\frac{e\overline{e}-1}g-g\overline{e}D
     \equiv  \frac{N\overline{e} -e}g \bmod N.
\end{equation}
Note that $(g,N)=1$ since $(e,g)=1$ and $g$ has an inverse $\bmod N$. Writing \eqref{elg2} as $x \bmod N$ then implies $-e \equiv gx \bmod N$
and the lemma follows.
\end{proof}

We now examine the first condition in  \eqref{R_D}.

\begin{lemma} \label{ellipse!}
For $\varepsilon$, $e$, $g \in \R$ with  $e^2-Dg^2=N \neq 0$ and $\varepsilon>1$ we have
\begin{equation} \label{iff}
    \frac 1{\varepsilon} \lqs  \left|\frac{e+g\sqrt{D} }{e-g\sqrt{D} }\right| \lqs \varepsilon \iff e^2+Dg^2 \lqs \left(\varepsilon+\frac{1}{\varepsilon}\right)\frac{|N|}2.
\end{equation}
\end{lemma}
\begin{proof}
The right side of \eqref{iff} is equivalent to
\begin{multline*}
    (e+g\sqrt{D})^2+(e-g\sqrt{D})^2 \lqs \left(\varepsilon+\frac{1}{\varepsilon}\right)|N|\\
    \iff \frac{(e+g\sqrt{D})^2}{(e-g\sqrt{D})^2}+1 \lqs \left(\varepsilon+\frac{1}{\varepsilon}\right)\frac{|N|}{(e-g\sqrt{D})^2}\\
    \iff \left|\frac{e+g\sqrt{D} }{e-g\sqrt{D} }\right|^2+1 \lqs \left(\varepsilon+\frac{1}{\varepsilon}\right)\left|\frac{e+g\sqrt{D} }{e-g\sqrt{D} }\right|
\end{multline*}
which is equivalent to the left side of \eqref{iff}.
\end{proof}

Since $\varepsilon_D = a_0+c_0\sqrt{D}$, we may write $\varepsilon_D+1/\varepsilon_D = 2a_0$ in \eqref{iff}.
Recall $R^*_D(N)$ defined in \eqref{rdnst}. We see from \eqref{R_D} and Lemma \ref{ellipse!} that $R_D(N)$ corresponds exactly to all pairs $(e,g) \in R^*_D(N)$  such that
\begin{equation} \label{hi}
    \left|\frac{e+g\sqrt{D} }{e-g\sqrt{D} }\right| \neq \varepsilon_D.
\end{equation}

\begin{lemma} \label{przw}
We have equality in \eqref{hi}  if and only if $(e,g)=\pm(u_+,v_+)$ or $\pm(u_-,v_-)$.
\end{lemma}
\begin{proof}
There can be equality in \eqref{hi} only for two possible values of $N$, as we see next.
We have
\begin{equation} \label{holds2}
     \frac{e+g\sqrt{D} }{e-g\sqrt{D} } = \pm \varepsilon_D=\pm(a_0+c_0\sqrt{D}) \iff e = \frac{a_0\pm 1}{c_0}g.
\end{equation}
If  $\gcd(e,g)=1$ then
\begin{align*}
    e = \frac{a_0+ 1}{c_0}g & \iff  e = \frac{u_+}{v_+} g  \iff (e,g)=\pm(u_+,v_+), \\
    e = \frac{a_0- 1}{c_0}g & \iff  e = \frac{u_-}{v_-} g  \iff (e,g)=\pm(u_-,v_-).
\end{align*}
We also note that $D_+>0$ and $D_-<0$ since $D_+= 2a_0 v_+^2/c_0^2$, $D_-= -2a_0 v_-^2/c_0^2$.
\end{proof}

The points $\pm(u_+,v_+)$ lie on both the ellipse $e^2+Dg^2=a_0|N|$ and the hyperbola $e^2-Dg^2=N$ for $N=D_+$. Similarly for
$\pm(u_-,v_-)$ when $N=D_-$. See Figure \ref{bfig}.

\begin{lemma}
If $(e,g,N)=(\pm u_+, \pm v_+,D_+)$ or $(\pm u_-, \pm v_-,D_-)$ then
\begin{equation}\label{sun}
    \e  \left(\frac{m}{\ell_\eta} \log \left|\frac{e+g\sqrt{D}}{e-g\sqrt{D}}\right| -  \frac{n e g^{-1}}{N} \right) = (-1)^{m+c_0 \cdot n}.
\end{equation}
\end{lemma}
\begin{proof}
Lemma \ref{przw} implies
\begin{equation*}
    \e  \left(\frac{m}{\ell_\eta} \log \left|\frac{e+g\sqrt{D}}{e-g\sqrt{D}}\right|  \right) =
    \e  \left(\frac{m}{\ell_\eta} \log \varepsilon_D  \right) =
    \e  \left(\frac{m}{2}   \right) = (-1)^m.
\end{equation*}
Writing $a_0+ 1 = \lambda u_+$ and $c_0 = \lambda v_+$ for $\lambda \in \Z_{\gqs 1}$ we find
\begin{align}
    a_0^2-Dc_0^2=1 & \implies (\lambda u_+ -1)^2 - D \lambda^2 (v_+)^2 = 1 \notag\\
    & \implies \lambda D_+ = 2u _+. \label{udp}
\end{align}
Consider $u_+(v_+)^{-1} \bmod D_+$.
If $c_0$ is even then $a_0$ must be odd and so $\lambda$ is even. Hence \eqref{udp} implies $D_+$ divides $u_+$ and therefore $u_+(v_+)^{-1} \equiv 0 \bmod D_+$ and
\begin{equation*}
    \e  \left( -  \frac{n e g^{-1}}{N} \right) = \e  \left( -  \frac{n u_+(v_+)^{-1}}{D_+} \right) = 1 = (-1)^{c_0 \cdot n}.
\end{equation*}
On the other hand,
if $c_0$ is odd then $\lambda$ is odd and $D_+$ is even. Hence \eqref{udp} implies $u_+ \equiv D_+/2 \bmod D_+$. In this case we must also have $v_+$ and $(v_+)^{-1}$ odd so that
$u_+(v_+)^{-1} \equiv D_+/2 \bmod D_+$ and
\begin{equation*}
    \e  \left( -  \frac{n e g^{-1}}{N} \right) = \e  \left( -  \frac{n u_+(v_+)^{-1}}{D_+} \right) = (-1)^n = (-1)^{c_0 \cdot n}.
\end{equation*}

This completes the proof for  $(e,g,N)=(u_+, v_+,D_+)$, $(-u_+, -v_+,D_+)$.
The proof for $(e,g,N)=(\pm u_-,\pm v_-,D_-)$ is similar.
\end{proof}

\begin{proof}[Proof of Theorem \ref{final_k}]
We see now that the sum for $S_{\eta\infty}$ in \eqref{sei} may be replaced by the sum over $(e,g) \in R_D^*(N)$ in \eqref{kloost4} except that the extra summands with $(e,g)=\pm(u_+,v_+)$, $\pm(u_-,v_-)$ must be removed from $R_D^*(N)$. This is accomplished by the term $-\psi_D(m,n;N)$ in \eqref{kloost4}.
The factor $1/2$ in both sums comes from the equivalence of $\left(\smallmatrix e  & f \\ g & h
\endsmallmatrix\right)$ and $\left(\smallmatrix -e  & -f \\ -g & -h
\endsmallmatrix\right)$.
\end{proof}

\section{Parabolic Poincar\'e series and their hyperbolic Fourier expansions} \label{secph}
The results in the section are similar to those in Section \ref{sechp}, switching $\eta$ with $\ca$.
Note the relation
\begin{equation} \label{ccre}
    c_\ca(m;P_{\eta,n}) = \overline{c_\eta(n;P_{\ca,m})} \left[\frac{(2\pi)^k \ell_\eta m^{k-1}e^{-2\pi^2n/\ell_\eta}}
    {|\G(k/2+2\pi i n/\ell_\eta)|^2 }
     \right] \qquad (m\in \Z_{\gqs 0}, n\in \Z)
\end{equation}
coming from \eqref{pin} and \eqref{epe} applied to $\s{P_{\ca,m}}{P_{\eta,n}}= \overline{\s{P_{\eta,n}}{P_{\ca,m}}}$.
However,  \eqref{ccre} is not quite symmetrical. For $m \lqs 0$ we have that $c_\ca(m;P_{\eta,n})=0$ since $P_{\eta,n} \in S_k(\G)$, but we don't expect $c_\eta(n;P_{\ca,m})$ to be zero since $P_{\ca,m} \in M^!_k(\G)$ for $m<0$ and $P_{\ca,0} \in M_k(\G)$.

\subsection{The parabolic/hyperbolic integral}
For $m$, $n \in \Z$ and $r \in \R_{\neq 0}$ define
\begin{equation} \label{lirr}
    I_{par \, \eta}(m,n;r):= \int_{-\infty+i y}^{\infty+i y}
    \frac{\e\left(m\left(\frac{\sgn(r)  e^{u} - 1}{2r(\sgn(r) e^{u} +  1)} \right)\right) e^{u(k/2-2\pi i n  /\ell_{\eta})}}
    { (\sgn(r) e^{u} +  1)^{k}}  \, \frac{du}{\ell_{\eta}} \qquad(0<y<\pi, \ k>0).
\end{equation}
This is the integral that appears in the proof of Theorem \ref{CISph} and we develop its properties here first.

\begin{prop} \label{bndph}
The integral \eqref{lirr} is absolutely convergent and we have the estimates
\begin{align}
    I_{par \, \eta}(m,n;r) & \ll   \exp\bigl(\pi e(|m|-m)/|r| \bigr)/\ell_{\eta}  &(n=0), \label{iphx1}\\
    I_{par \, \eta}(m,n;r) & \ll n^{k/2} \exp\left(\pi^2 n^{1/2} \left(\frac{1}{\ell_{\eta}}+\frac{|m|-m)}{|r|} \right) \right)/\ell_{\eta}  & (n >0 ), \label{iphx2}\\
    I_{par \, \eta}(m,n;r) & \ll |n|^{k/2} e^{2\pi^2 n/\ell_{\eta}} \exp\left(\pi^2 |n|^{1/2} \left(\frac{1}{\ell_{\eta}}+\frac{|m|-m)}{|r|} \right) \right)/\ell_{\eta}  & (n < 0), \label{iphx3}
\end{align}
for  implied constants depending only on $k>0$.
\end{prop}
\begin{proof}
Let $u=x+iy$ to get
\begin{align}
    \left|\e\left(m\left(\frac{\sgn(r)  e^{u} - 1}{2r(\sgn(r) e^{u} +  1)} \right)\right) \right| & =
    \exp\left( \frac{2\pi m}{r\left|\sgn(r) e^{x+iy}+1\right|^2} \Im \left(\sgn(r) e^{x-iy}+1 \right)\right) \notag\\
    & =
    \exp\left( \frac{-2\pi m \sin(y)}{|r|} \frac{e^x}{\left|e^x +\sgn(r) e^{-iy}\right|^2}\right). \label{lab}
\end{align}
Hence
\begin{equation*}
    \left| I_{par \, \eta}(m,n;r) \right| \lqs \frac{e^{2\pi  n y/\ell_{\eta}}}{\ell_{\eta}}\int_{-\infty}^\infty
    \exp\left( \frac{-2\pi m \sin(y)}{|r|} \frac{e^x}{\left|e^x +\sgn(r) e^{-iy}\right|^2}\right)
    \left( \frac{e^x}{\left|e^x +\sgn(r) e^{-iy}\right|^2} \right)^{k/2} \, dx.
\end{equation*}
Also
\begin{equation} \label{123a}
    \frac{e^x}{\left|e^x +\sgn(r) e^{-iy}\right|^2} \lqs \begin{cases}
    e^{1-|x|} &\text{ \ when \ } |x|\gqs 1 \\ e\sin^{-2}(y) &\text{ \ when \ } |x|\lqs 1.
    \end{cases}
\end{equation}
It follows that \eqref{lab} is bounded by
\begin{equation*}
    \exp\left( \frac{\pi (|m|-m) \sin(y)}{|r|} e^{1-|x|} \right)  \text{ \ when \ } |x|\gqs 1, \quad  \exp\left( \frac{\pi (|m|-m) \sin(y)}{|r|} \frac{e}{\sin^2(y)}  \right)\text{ \ when \ } |x|\lqs 1.
\end{equation*}
Altogether, for an implied constant depending only on $k>0$,
\begin{equation} \label{uiii}
     I_{par \, \eta}(m,n;r) \ll \frac{e^{2\pi  n y/\ell_{\eta}}}{\ell_{\eta}}
    \left( \exp\left( \frac{\pi (|m|-m) \sin(y)}{|r|}  \right)
    + \exp\left( \frac{\pi (|m|-m) e}{|r| \sin(y)} \right)  \frac{1}{\sin^k(y)}
    \right)
\end{equation}
proving convergence. We have that \eqref{lirr} is independent of $y$ by Cauchy's theorem.
Finally, letting $y=\pi n^{-1/2}/2$, $y=\pi/2$ and $y=\pi(1-|n|^{-1/2}/2)$ in \eqref{uiii} for $n>0$, $n=0$ and $n<0$, respectively, and using
\begin{equation} \label{sini}
    2y/\pi \lqs \sin(y), \ \sin(\pi-y) \lqs y \quad \text{for} \quad 0\lqs y\lqs \pi/2,
\end{equation}
completes the proof.
\end{proof}

\begin{prop} Let $k\in \R_{>0}$.
For $m$, $n \in \Z$ and $r \in \R_{\neq 0}$ we have
\begin{multline}\label{iiiph}
     I_{par \, \eta}(m,n;r) = \frac{1}{\ell_{\eta}}
    \exp\left( \frac{\pi i}2\left( \frac k2 -\frac{2\pi i n}{\ell_{\eta}}  \right) (1-\sgn(r)) -\frac{\pi i m}{r}\right)
    \\
    \times
    B\left(\frac k2+\frac{2\pi i n}{\ell_{\eta}}, \frac k2-\frac{2\pi i n}{\ell_{\eta}}\right)
    {_1}F_1\left(\frac k2-\frac{2\pi i n}{\ell_{\eta}};k; \frac{2\pi i m}{r}\right).
\end{multline}
Also $I_{par \, \eta}(m,n;r)$ is real-valued when $k$ is even.
\end{prop}
\begin{proof}
Let $u=t+iy$ for $y=\pi/2$ and then write $v=e^t$ so that \eqref{lirr} becomes
 \begin{equation} \label{lirrr}
     \exp\left( \frac{\pi i}2\left( \frac k2 -\frac{2\pi i n}{\ell_{\eta}} +\frac{2m}{r} \right)\right)
    \int_{0}^{\infty}
    \frac{\exp\left(\frac{-2\pi i m}{r(\sgn(r)i v +  1)} \right) v^{k/2-2\pi i n  /\ell_{\eta}-1}}
    { (\sgn(r)i v +  1)^{k}}  \, \frac{dv}{\ell_{\eta}}.
\end{equation}
Substitute $x=1/(\sgn(r)i v +  1)$ and the integral in \eqref{lirrr} is now
\begin{equation*}
    \frac{1}{\sgn(r)i \ell_{\eta}} \int_0^1 \exp\left(-\frac{2\pi i m}{r} x \right) x^{k-2} \left( \frac{1-x}{\sgn(r)i x}\right)^{k/2-2\pi i n/\ell_{\eta}-1} \, dx
\end{equation*}
where the path of integration runs along a semicircle centered at $1/2$. Except for the endpoints, we have $-\pi/2< \sgn(r)\arg(x) < 0$ and $0< \sgn(r)\arg(1-x) < \pi/2$. Hence
\begin{equation*}
    \left( \frac{1-x}{\sgn(r)i x}\right)^{w} = (1-x)^w (\sgn(r) i)^{-w} x^{-w} \qquad(w\in \C).
\end{equation*}
Finally, move the contour of integration to the interval $[0,1] \subset \R$ and use
\begin{equation*}
    \int_0^1  x^{\mu-1}(1-x)^{\nu -1} e^{\beta x} \, dx = B(\mu,\nu) \ {_1}F_1(\mu;\mu+\nu;\beta)
\end{equation*}
when $\Re(\mu)$, $\Re(\nu) >0$ from \cite[3.383.1]{GR},
along with an application of  Kummer's transformation  \eqref{kummer}, to show \eqref{iiiph}. It now follows from \eqref{iiiph}, as in Proposition \ref{usop}, that $I_{par \, \eta}(m,n;r) \in \R$ for $k$ even.
\end{proof}

\subsection{Double cosets and Kloosterman sums in the parabolic/hyperbolic case}
All of the results in Subsections \ref{sec42} and \ref{sec43} translate directly to the parabolic/hyperbolic case here by means of the map $\se^{-1}\G \sa \to \sa^{-1}\G\se$ given by $\g \mapsto \g^{-1}$. We summarize the main things we need:

\begin{enumerate}
\item If $\left(\smallmatrix a
& b \\ c & d
\endsmallmatrix\right) \in \sa^{-1}\G\se$ then $cd \neq 0$.
\item Let $\G(\ca,\eta)$ be a complete set of inequivalent
representatives for $\G_\ca\backslash \G/\G_\eta$.
Then
\begin{equation} \label{phjiu}
    \Big\{ \delta \tau \ \Big| \ \delta \in \G(\ca,\eta), \ \tau \in \G_\eta/Z  \Big\}
\end{equation}
is a complete set of inequivalent
representatives for
$\G_\ca\backslash \G$.
\item
In this case \cite[Lemma 1]{G83} says:
\begin{lemma} \label{bruph}
For $M=\left(\smallmatrix a
& b \\ c & d
\endsmallmatrix\right) \in \SL_2(\R)$ with $cd \neq 0$ we have
\begin{equation} \label{ph}
    \begin{pmatrix} a & b \\ c & d \end{pmatrix} = \frac{\sgn(d)}{\sqrt{2}}
    \begin{pmatrix} 1 & \frac{a}{2c}+\frac{b}{2d} \\ 0 & 1 \end{pmatrix}
    \begin{pmatrix} 1/\nu & 0 \\ 0 & \nu \end{pmatrix}
    \begin{pmatrix} 1 & -\sgn(cd)  \\ \sgn(cd)  & 1  \end{pmatrix}
    \begin{pmatrix} \left| \frac dc \right|^{-1/2} & 0 \\ 0 & \left| \frac dc \right|^{1/2} \end{pmatrix}
\end{equation}
for $\nu = \sideset{_\text{par}}{_\text{hyp}}{\opv}(M) = |2cd|^{1/2}$.
\end{lemma}
\item
Define
\begin{equation*}
R_{\ca\eta}  :=  \left\{ \left. \begin{pmatrix} a  & b \\ c & d \end{pmatrix} \in \sa^{-1}\G \se \ \right|   \ 0 \lqs \frac{a}{2c}+\frac{b}{2d} <1, \ \frac{1}{\varepsilon_\eta} \lqs \left| \frac dc \right| < \varepsilon_\eta   \right\}
\end{equation*}
and we may take $\sa^{-1} \G(\ca,\eta)\se =  R_{\ca\eta} /Z$.
\item Put $C_{\ca\eta}=\left\{cd \ \left| \ \left(\smallmatrix a
& b \\ c & d
\endsmallmatrix\right) \in \sa^{-1}\G\se \right. \right\}$. We have $C_{\ca\eta}=-C_{\eta\ca}$.
\item For $C \in C_{\ca\eta}$  define
\begin{equation}\label{kloostph}
S_{\ca\eta}(m,n;C):= \sum_{\substack{ \g \in \G_\ca \backslash \G / \G_{\eta} \\ \left(\smallmatrix a
& b \\ c & d
\endsmallmatrix\right) = \sa^{-1}\g\se, \ cd=C} } \e\left(m\left(\frac{a}{2c}+\frac{b}{2d}\right)  +  \frac{n}{\ell_{\eta}} \log \left| \frac{c}{d} \right|\right).
\end{equation}
It is related to Good's sum by
\begin{equation*}
    S_{\ca\eta}(m,n;C) =
    \sideset{_\ca^{0}}{_{\eta}^{\delta'}}{\opS}(m,n;|2C|^{1/2})  \qquad \text{for} \qquad \delta' = \frac{1+\sgn(C)}{2}.
\end{equation*}
Also
\begin{equation}
    S_{\ca\eta}(m,n;C)  = S_{\eta\ca}(-n,-m;-C) = \overline{S_{\eta\ca}(n,m;-C)}, \label{po2}
\end{equation}
so the formula  in Theorem \ref{final_k} for $\ca=\ci$ and $\eta=(-\sqrt{D},\sqrt{D})$ also evaluates $ S_{\ca\eta}(m,n;C)$.
\item With
\begin{equation*}
    \mathcal N_{\ca\eta}(C)  :=  S_{\ca \eta}(0,0;C)
      = \#\left\{ \g \in \G_\ca \backslash \G / \G_{\eta} \ \left| \ \sa^{-1}\g\se = \begin{pmatrix} a  & b \\ c & d \end{pmatrix} \text{ with } cd= C\right. \right\}
\end{equation*}
it is clear that $\mathcal N_{\ca\eta}(C) = \mathcal N_{\eta\ca}(-C)$. Therefore
\begin{gather} \label{pkl}
    \sum_{C \in C_{\ca\eta}, \ |C| \lqs X}  \mathcal N_{\ca\eta}(C) \ll X, \\
    \mathcal N_{\ca\eta}(C) \ll C, \qquad S_{\ca\eta}(m,n;C) \ll C, \qquad \#\{C \in C_{\ca\eta} \ : \ |C| \lqs X\}   \ll X. \label{impaph}
\end{gather}
\end{enumerate}

\subsection{The hyperbolic expansion of $P_{\ca,m}$}
In the next result we prove a formula for $c_{\eta}(n;P_{\ca,m})$ using the same approach as in Theorems \ref{CISpp} and \ref{CIShp2}. An alternative derivation may be given using \er{kcq} as the starting point.
\begin{theorem} \label{CISph}
For all $m$, $n \in \Z$, the $n$th hyperbolic Fourier coefficient at $\eta$ of the parabolic Poincar\'e series $P_{\ca,m}$ is given by
\begin{equation} \label{sum3ph}
c_{\eta}(n;P_{\ca,m})  =   \sum_{C \in C_{\ca\eta}}
|C|^{-k/2}
 I_{\ca\eta}(m,n;C) S_{\ca\eta}(m,n;C).
\end{equation}
\end{theorem}
\begin{proof}
With \eqref{phjiu} and $z \in \H$, write the absolutely convergent
\begin{equation} \label{sparph}
    \left(P_{\ca,m}|_k \se\right)(z) = \sum_{\g \in \G_\ca \backslash \G }
    \frac{\e(m(\sa^{-1}\g \se z))}
    {j(\sa^{-1}\g \se, z)^{k}}
\end{equation}
as
\begin{equation} \label{fsumph}
    \sum_{C \in C_{\ca\eta}}
    \sum_{\substack{ \g \in \G(\ca,\eta) \\ \left(\smallmatrix a
& b \\ c & d
\endsmallmatrix\right) = \sa^{-1}\g\se,
\
 cd=C }}
    \sum_{n\in \Z}
    \frac{\e(m(\sa^{-1}\g \se (e^{n \ell_{\eta}+A})))}
    {j(\sa^{-1}\g \se , e^{n \ell_{\eta}+A})^{k} e^{-n \ell_{\eta} k/2}}
\end{equation}
where we let $z=e^A$ for $0<\Im A< \pi$. With $
\sa^{-1}\g\se = \left(\smallmatrix a
& b \\ c & d
\endsmallmatrix\right) \in \SL_2(\R)
$
for $cd \neq 0$,  the inner sum is $\sum_{n\in \Z}f_\g(n)$ for
\begin{equation*}
    f_\g(t) = f(t):=\frac{\exp\left( 2\pi i m \frac{a e^{t \ell_{\eta}+A}+b}{c e^{t \ell_{\eta}+A}+d} +t \ell_{\eta} k/2\right)}{(c e^{t \ell_{\eta}+A}+d)^k}.
\end{equation*}
As in the proof of Theorem \ref{CISpp}, we may apply Poisson summation if
$\int_{-\infty}^\infty f(t)\, dt$ and $\int_{-\infty}^\infty |f''(t)|\, dt$  exist. The first integral exists for $k>0$ by similar arguments to Proposition \ref{bndph} with $n=0$. It follows that the second also exists since,
with $g(t):=c e^{t \ell_{\eta}+A}+d$,
\begin{equation*}
    \frac{f''(t)}{\ell_{\eta}^2 }=f(t)\left[
    \frac{2\pi i m e^{t \ell_{\eta}+A}}{g(t)^2}
    \left( \frac{2\pi i m e^{t \ell_{\eta}+A}}{g(t)^2} +
    \frac{2d(k+1)}{g(t)}-k-1\right)
    +\frac{d k(k+1)}{g(t)}\left(\frac{d}{g(t)} +1\right)
    +\frac{k^2}{4}
    \right].
\end{equation*}
With Poisson summation, as in \eqref{psfhp},
the inner sum in \eqref{fsumph} is now
\begin{equation} \label{intyy2ph}
    \sum_{n\in \Z}\int_{-\infty}^\infty
    \frac{\e\left( m\left(\frac ac - \frac 1{c(c e^{\ell_{\eta} t + A}+d)}\right)\right)}
    {(c e^{\ell_{\eta} t + A}+d)^{k}} e^{t\ell_{\eta}(k/2 -2\pi i n/\ell_{\eta})}\, dt.
\end{equation}
Let $u=\ell_{\eta} t +A+ \log \left| \frac{c}{d} \right|$ in \eqref{intyy2ph} and use $\frac ac = \frac{a}{2c}+\frac{b}{2d}+\frac{1}{2cd}$ to get that the integral equals
\begin{multline} \label{intyy3ph}
z^{-k/2+2\pi i n /\ell_{\eta}} \left| \frac{d}{c} \right|^{k/2}
\e\left(m\left(\frac{a}{2c}+\frac{b}{2d}\right)  -  \frac{n}{\ell_{\eta}} \log \left| \frac{d}{c} \right|\right) \\
    \times \int_{-\infty+i \Im A}^{\infty+i \Im A}
    \frac{\e\left(m\left(\frac{c \left| \frac{d}{c} \right| e^u -d}{2cd (c\left| \frac{d}{c} \right| e^u +d)} \right)\right) }
    {(c\left| \frac{d}{c} \right| e^u +d)^{k} }
    e^{u(k/2-2\pi i n /\ell_{\eta})}
    \, \frac{du}{\ell_{\eta}}.
\end{multline}
Writing the integral in \eqref{intyy3ph} with
 \eqref{lirr}, we have shown that \eqref{fsumph} equals
\begin{equation} \label{fsumph2}
    \sum_{C \in C_{\ca\eta}}
    \sum_{\substack{ \g \in \G(\ca,\eta) \\ \left(\smallmatrix a
& b \\ c & d
\endsmallmatrix\right) = \sa^{-1}\g\se,
\
 cd=C }}
    \sum_{n\in \Z}
    z^{-k/2+2\pi i n /\ell_{\eta}}
    \e\left(m\left(\frac{a}{2c}+\frac{b}{2d}\right)  -  \frac{n}{\ell_{\eta}} \log \left| \frac{d}{c} \right|\right)
    \frac{I_{\ca\eta}(m,n;C)}{|C|^{k/2}}.
\end{equation}
For $z=e^A$ with $0<\Im A<\pi$ as before, we have $z^{-k/2+2\pi i n /\ell_{\eta}} \ll e^{-2\pi n \Im A/\ell_{\eta}}$. With Proposition \ref{bndph} we then have
\begin{equation*}
    z^{-k/2+2\pi i n /\ell_{\eta}} I_{\ca\eta}(m,n;C) \ll e^{-\varepsilon |n|}
\end{equation*}
for $\varepsilon >0$, depending on $z$. Therefore
 \eqref{fsumph2} is majorized by
$
     \sum_{C \in C_{\ca\eta}} \mathcal  |C|^{-k/2} N_{\ca\eta}(C)
$
and thus convergent for $k>2$ by \eqref{pkl}. This proves that changing the order of summation in \eqref{fsumph2} is  valid,
and moving the sum over $n$ to the outside completes the proof.
\end{proof}

Set
\begin{align}
S^\star_{\ca\eta}(m,n;C) & := \exp\left(-\pi^2 n  (\sgn(C)+1)/\ell_\eta - \pi i m/C \right) S_{\ca\eta}(m,n;C) \label{poy}
\\
& \phantom{:} = \sum_{\substack{\g \in \G_\ca \backslash \G / \G_\eta  \\ \left(\smallmatrix a
& b \\ c & d
\endsmallmatrix\right) = \sa^{-1}\g\se, \ cd=C} } \e\left( m  \frac{b}{d} + \frac{n}{\ell_\eta} \log \left(-\frac cd\right) \right)
\notag
\end{align}
and note the relation with the sum \eqref{kloos}
\begin{equation}
S^\star_{\ca\eta}(m,n;C)  = \overline{S^\star_{\eta\ca}(n,m;-C)}. \label{po3}
\end{equation}
Rewriting $I_{\ca\eta}(m,n;C) S_{\ca\eta}(m,n;C)$ with \eqref{iiiph} and \eqref{poy} gives a more explicit version of Theorem \ref{CISph}:
\begin{theorem} \label{CISph2}
For all $m$, $n \in \Z$
\begin{multline} \label{sum3phb}
    c_{\eta}(n;P_{\ca,m})
     =
     \frac{e^{2\pi^2n/\ell_\eta}}{\ell_\eta}
    B\left(\frac k2+\frac{2\pi i n}{\ell_{\eta}}, \frac k2-\frac{2\pi i n}{\ell_{\eta}}\right)
    \\
    \times
    \sum_{C \in C_{\ca\eta}}
    {_1}F_1\left(\frac k2-\frac{2\pi i n}{\ell_{\eta}};k; \frac{2\pi i m}{C}\right)
    \frac{S^\star_{\ca\eta}(m,n;C)}{C^{k/2}}.
\end{multline}
\end{theorem}
Theorem \ref{CISph2} combined with the identity \eqref{ccre} and the symmetry \eqref{po3} implies Theorem  \ref{CIShp}. Theorem  \ref{CIShp} does not imply Theorem \ref{CISph2} in the same way.

\section{Computations} \label{num}
In this section we restrict our attention to $\G = \SL_2(\Z)$, its cusp at $\ci$ with scaling matrix $\si=I$ and its hyperbolic pairs $\eta=(-\sqrt{D},\sqrt{D})$ with scaling matrix $\hat\sigma_\eta$ given by \eqref{hypscat}.

\subsection{Parabolic coefficients} \label{numa}
We have $C_{\ci\ci}=\Z_{\gqs 1}$ in the notation of Subsection \ref{klpp}. With $c\in C_{\ci\ci}$, using for example Lemma \ref{ac01pp} and the sentences following it, we obtain the classical Kloosterman sum
\begin{equation*}
    S_{\ci\ci}(m,n;c)=\sum_{0 \lqs d <c, \ (c,d)=1, \ ad \equiv 1 \bmod c} \e\left(m\frac ac + n\frac dc\right).
\end{equation*}
 It is necessarily real-valued since each term with $a$, $d$ gets added to (or equals) its conjugate with $c-a$, $c-d$. For all $m,n \in \Z$, Theorem \ref{CISpp} then gives
\begin{multline}
c_\ci(n;P_{\ci,m})  = \delta_{mn}+\begin{cases}\displaystyle \frac{(2\pi i)^k n^{k-1}}{\G(k)} \sum_{C=1}^\infty
 {_0}F_1 \left(;k; - \frac{4\pi^2 mn}{C^2} \right)
 \frac{S_{\ci\ci}(m,n;C)}{C^k} \text{ \ \ \ if \ $n \gqs 1$}.
 \end{cases}
 \label{sum3pp2}
\end{multline}
This is usually stated using $J$-Bessel functions (recall \eqref{jb1}) when $m \gqs 1$ and $I$-Bessel functions  when $m \lqs -1$. See for example  \cite[Thm. 5.3.2]{R77}.

With \eqref{sum3pp2}, we can investigate the Poincar\'e series $P_{\ci,m}$ numerically.
The coefficient $c_\ci(n;P_{\ci,m})$ evaluates to zero when $m\in \Z_{\gqs 1}$ and $k\in\{4,6,8,10,14\}$ since $\dim S_k(\G)=0$ in these cases. The space $S_{12}(\G)$ is one-dimensional, containing $\Delta(z):=q\prod_{n=1}^\infty (1-q^n)^{24}=\sum_{m=1}^\infty \tau(m) q^m$. It follows that, when $k=12$, each $P_{\ci,m}$ for $m\in \Z_{\gqs 1}$ must equal $\lambda_m \Delta$ for some $\lambda_m \in \R$. Since $\lambda_m = c_\ci(1;P_{\ci,m})$, we find for example:
\begin{equation}\label{lams}
    \lambda_1 \approx 2.840287, \ \lambda_2 \approx -0.0332846, \ \lambda_3 \approx 0.004040443, \ \lambda_4 \approx -0.0009968.
\end{equation}
This is consistent with \eqref{pin}, which implies $\lambda_m = \G(11)\tau(m)/(\s{\Delta}{\Delta} (4\pi m)^{11})$ for $m\in \Z_{\gqs 1}$.

 If $m=0$ then $P_{\ci,0}$ is the Eisenstein series $E_k(z):=1-2k/B_k\sum_{n=1}^\infty \sigma_{k-1}(n) q^n$ where $B_k$ is the $k$th Bernoulli number and $\sigma_{k-1}(n) := \sum_{d|n}d^{k-1}$.

For $m \in \Z_{\lqs -1}$, each $P_{\ci,m}$ is related to the $j$-function, defined as
$$j(z):=E_4^3(z)/\Delta(z) =q^{-1}+744+196884 q+ \cdots \in M_0^!(\G).$$
For example, in weight $k=12$, Rankin \cite[(4.4)]{rank} constructed $\bigl(j(z)+264\bigr)E^2_6(z)= q^{-1}-598428q+\cdots$  with integer coefficients. Then we must have
\begin{equation}\label{rank}
    P_{\ci,-1}(z)=\bigl(j(z)+264 \bigr)E^2_6(z) + \lambda_{-1} \Delta(z)
\end{equation}
for some $\lambda_{-1} \in \R$. Computing, we find $\lambda_{-1} \approx 600270.8947$, agreeing with \cite[(4.12)]{rank}.  Following Rankin's method we similarly have
\begin{equation*}
    P_{\ci,-2}(z)=\bigl(j(z)^2-480j(z)+205128 \bigr)E^2_6(z) + \lambda_{-2} \Delta(z) \quad \text{for} \quad \lambda_{-2} \approx 321214058.075.
\end{equation*}
It would be interesting to  identify the $\lambda_{m}$s resulting from continuing this sequence.
For more on parabolic Poincar\'e series and weakly holomorphic forms, see for example \cite{DJ08}, \cite{Rh12} and their contained references.

With Theorems \ref{CIShp} and \ref{final_k} we may calculate the parabolic Fourier coefficients at $\ci$ of the series $P_{\eta,m}$ with $\eta=(-\sqrt{D},\sqrt{D})$. In weight $k=12$ each is again a constant times $\Delta$. This constant, $c_\ci(1;P_{\eta,m})$, is given numerically  for $-2\lqs m \lqs 2$ and $D=2$, $3$ and $5$ in Table \ref{pk12}.
\begin{table}[h]
\begin{center}
\begin{tabular}{r|rrr}
$m$ & $D=2$ & $D=3$ & $D=5$   \\ \hline
$2$ & $23.43$ &  $7.93$ & $-130.37$  \\
$1$ & $252.41$ &  $114.79$ & $-311.81$  \\
$0$ & $1529.46$ &  $-1665.07$ & $1857.25$  \\
$-1$ & $-68190.34$ &  $78417.86$ & $9515.95$  \\
$-2$ & $1709726.97$ &  $-12443941.21$ & $-121422.56$
\end{tabular}
\caption{Computations of $c_\ci(1;P_{\eta,m})$ for $\eta=(-\sqrt{D},\sqrt{D})$ with $k=12$.} \label{pk12}
\end{center}
\end{table}
To see why all the entries in the table are real it is simplest to use \eqref{sum3b} in Theorem \ref{CIShp2}. We know that the factor $I_{\eta\ci}(m,n;1/(2C))$ there is real by Proposition \ref{usop}. The factor $S_{\eta\ci}(m,n;C)$ is also real since in the formula \eqref{kloost4} each term with $e$, $g$ is added to (or equals) its conjugate with $e$, $-g$.

\subsection{Hyperbolic coefficients}
\label{numb}

With Theorems \ref{CISph} or \ref{CISph2} we may numerically compute the hyperbolic expansion coefficients at $\eta=(-\sqrt{D},\sqrt{D})$ of $P_{\ci,m}$ for $m\in \Z$. As above, the Kloosterman sums are computed with Theorem \ref{final_k}, now combined with the symmetries \eqref{po2} or \eqref{po3}, and the coefficients are necessarily real.

The first column of Table \ref{pk12b} shows part of the hyperbolic expansion of the weight $k=12$ series $P_{\ci,1}$ at $\eta=(-\sqrt{2},\sqrt{2})$. As we saw  earlier, $P_{\ci,1} = \lambda_1 \Delta$ for $\lambda_1$ given in \eqref{lams}. Hence, on renormalizing, we obtain the expansion \eqref{dl2}. As in the parabolic and elliptic cases, we suspect that these coefficients should have some arithmetic significance, but this remains to be determined.

A noticeable feature of these hyperbolic coefficients, first shown by Hiramatsu in \cite[Thm. 1]{Hir}, is that they  have exponential decay as $n \to -\infty$. A slightly more precise version of his result,  appearing in \cite{IMO}, is that for all $f \in S_k(\G)$ we have
\begin{equation} \label{h1}
c_{\eta}(m;f) \ll |m|^{k/2}\times
\begin{cases}1 & \text{if \ \ $  m>0$},\\ e^{-2\pi^2|m|/\ell_\eta} & \text{if \ \ $m<0$}.
\end{cases}
\end{equation}
This is the analog of the usual Hecke bound for parabolic Fourier coefficients.
\begin{table}[h]
\begin{center}
\begin{tabular}{r|d{11} d{6} d{4}}
 $n$  &  \multicolumn{1}{c}{$P_{\ci,1}$}  &        \multicolumn{1}{c}{$P_{\ci,0}$}  &     \multicolumn{1}{c}{$P_{\ci,-1}$}    \\ \hline
 $3$  &  -0.0039  &       -10417.11  &  -798957.50   \\
 $2$  &  0.2114    &   445.10  &  3632.46    \\
 $1$  &  0.0418  &   -7.88  &   -4.4001   \\
 $0$  &  0.00165  &   0.106  &  0.0017   \\
 $-1$  &  -0.000155  &   0.0292  &  0.0163    \\
 $-2$  &  0.00000290   &   0.00610  &  0.0498   \\
 $-3$  &  0.000000000198   &   0.000528  &  0.0405
\end{tabular}
\caption{Coefficients  $c_\eta(n;P_{\ci,m})$ for $m=1$, $0$, $-1$ and $\eta=(-\sqrt{2},\sqrt{2})$ with $k=12$.} \label{pk12b}
\end{center}
\end{table}
The second and third columns in Table \ref{pk12b} give the larger hyperbolic coefficients of $E_{12} \in M_{12}(\G)$ and  $P_{\ci,-1}  \in M^!_{12}(\G)$ as seen in \eqref{rank}.

\section{Hyperbolic Poincar\'e series and their hyperbolic Fourier expansions} \label{sechh}
Returning to a general $\G$, let $\eta$ and $\eta'$ be two pairs of hyperbolic fixed points: $\eta=(\eta_1,\eta_2)$ and $\eta'=(\eta'_1,\eta'_2)$.
We describe the hyperbolic Fourier expansion of $P_{\eta,m}$ at $\eta'$,
$$
\left(P_{\eta,m}|_k \sep \right)(z) = \sum_{n \in \Z} c_{\eta'}(n; P_{\eta,m}) z^{-k/2+2\pi i n /\ell_{\eta'}},
$$
in this section. Here, the group $\G$ may or may not have parabolic elements.

\subsection{The hyperbolic/hyperbolic integral}
The integral we will need shortly in \eqref{intyy4} is the following one.
Let $r \in \R_{\neq 0,1}$ and $\alpha$, $\beta =\pm1$ satisfy $\alpha \beta =\sgn(r)$. For $m,$ $n \in \Z$ put
\begin{equation} \label{iefa}
    I_{\eta\eta'}(m,n;r,\alpha,\beta):= \int_{-\infty+i y}^{\infty+i y}
    \frac{
    \left(\bigl| \frac{r}{r-1} \bigr|^{1/2}
    \cdot
    \frac{\alpha   e^{u} + \sgn(r-1) \left| \frac{r-1}{r} \right|^{1/2} }
    {  e^{u}  + \beta  \left| \frac{r}{r-1} \right|^{1/2}} \right)^{2\pi i m /\ell_\eta}
    e^{u(k/2-2\pi i n /\ell_{\eta'})}
    }
    {\scriptstyle
    \left( \alpha  e^{u} + \sgn(r-1) \left| \frac{r-1}{r} \right|^{1/2}   \right)^{k/2}
    \left( e^{u}  + \beta  \left| \frac{r}{r-1} \right|^{1/2} \right)^{k/2}}  \, \frac{du}{\ell_{\eta'}}
\end{equation}
 where $0<y<\pi$ and $k>0$.  We next establish good bounds for $I_{\eta\eta'}(m,n;r,\alpha,\beta)$ with respect to $n$. These bounds will be required at the end of the proof of Theorem \ref{CIShh2}.

\begin{prop} \label{bndhh}
The integral \eqref{iefa} is absolutely convergent. We have
\begin{align}
    I_{\eta \eta'}(m,n;r,\alpha,\beta) & \ll   \exp\left(\pi^2 (|m|-m)/\ell_{\eta} \right)/\ell_{\eta'}  &(n=0), \label{ihhx1}\\
    I_{\eta \eta'}(m,n;r,\alpha,\beta) & \ll n^{k/2} \exp\left(\frac{\pi^2(|m|-m)}{\ell_{\eta}} + \frac{\pi^2 n^{1/2}}{\ell_{\eta'}}  \right)/\ell_{\eta'}  & (n>0), \label{ihhx2}\\
    I_{\eta \eta'}(m,n;r,\alpha,\beta) & \ll |n|^{k/2} e^{2\pi^2 n/\ell_{\eta'}} \exp\left(\frac{\pi^2(|m|-m)}{\ell_{\eta}} + \frac{\pi^2 |n|^{1/2}}{\ell_{\eta'}}  \right)/\ell_{\eta'}  & (n <0), \label{ihhx3}
\end{align}
for  implied constants depending only on $k>0$.
\end{prop}
\begin{proof}
Note that $w:=\frac{\alpha   e^{u} + \sgn(r-1) \left| \frac{r-1}{r} \right|^{1/2} }
    {  e^{u}  + \beta  \left| \frac{r}{r-1} \right|^{1/2}} \in \H$ so that $0 < \arg w < \pi$. Hence, with $u=x+iy$,
\begin{equation*}
     \left|w^{2\pi i m /\ell_\eta}e^{-2\pi i n u/\ell_{\eta'}}\right| \lqs e^{\pi^2 (|m|-m)/\ell_\eta} \cdot e^{2\pi ny/\ell_{\eta'}}.
\end{equation*}
The remaining part of the integrand in \eqref{iefa} is bounded by
\begin{equation} \label{rej}
    \frac{e^{xk/2}}{\left|  \bigl| \frac{r}{r-1} \bigr|^{1/2} e^{x} +\alpha \sgn(r-1)e^{-iy}\right|^{k/2} \left| \bigl| \frac{r-1}{r} \bigr|^{1/2} e^{x} + \beta e^{-iy} \right|^{k/2}}.
\end{equation}
Let $t=x+\frac 12 \log\bigl| \frac{r}{r-1} \bigr|$ and $u=x-\frac 12 \log\bigl| \frac{r}{r-1} \bigr|$.
Using \eqref{123a},
\begin{align*}
    \frac{e^{x}}{\left|  \bigl| \frac{r}{r-1} \bigr|^{1/2} e^{x} +\alpha \sgn(r-1)e^{-iy}\right|^{2}} & \lqs
    \left| \frac{r}{r-1} \right|^{1/2} \times \begin{cases}
    e^{1-|t|} & \text{ if } |t| \gqs 1\\
    e \sin^{-2}(y) & \text{ if } |t| \lqs 1,
    \end{cases}
    \\
    \frac{e^{x}}{\left| \bigl| \frac{r-1}{r} \bigr|^{1/2} e^{x} + \beta e^{-iy} \right|^{2}} & \lqs
    \left| \frac{r-1}{r} \right|^{1/2} \times \begin{cases}
    e^{1-|u|} & \text{ if } |u| \gqs 1\\
    e \sin^{-2}(y) & \text{ if } |u| \lqs 1.
    \end{cases}
\end{align*}
Therefore, \eqref{rej} is bounded by
\begin{align*}
    \left( e^{2-|t|-|u|} \right)^{k/4} \lqs e^{(2-2|x|)k/4} \qquad & \text{if} \qquad |t| \gqs 1 \quad \text{and} \quad |u| \gqs 1,\\
    \left( e^2 \sin^{-4}(y) \right)^{k/4} = e^{k/2} \sin^{-k}(y) \qquad & \text{if} \qquad |t| \lqs 1 \quad \text{or} \quad |u| \lqs 1.
\end{align*}
Altogether, for an implied constant depending only on $k>0$,
\begin{equation*}
    I_{\eta \eta'}(m,n;r,\alpha,\beta) \ll e^{\pi^2 (|m|-m)/\ell_\eta} \cdot e^{2\pi ny/\ell_{\eta'}}\left( 1+ \sin^{-k}(y)\right).
\end{equation*}
Therefore \eqref{iefa} is absolutely convergent. Since the integrand is holomorphic, it  is independent of $y$ with $0<y<\pi$.
Recalling \eqref{sini} and choosing $y=\pi n^{-1/2}/2$, $y=\pi/2$ and $y=\pi(1-|n|^{-1/2}/2)$ for $n>0$, $n=0$ and $n<0$, respectively,  finishes the proof.
\end{proof}

\begin{prop} \label{long}
Let $k \in \R_{>0}$.
If $r \not\in (0,1)$ or if $r \in (0,1)$ and $\alpha=1$ then
\begin{multline} \label{ffv2}
    I_{\eta\eta'}(m,n;r,\alpha,\beta)=\frac{\sgn(r)^{k/2}}{\ell_{\eta'}} \left| \frac{r-1}{r} \right|^{k/4} e^{2\pi^2 n/\ell_{\eta'}}
    \\
    \times
    \e\left(\frac{m}{2\ell_\eta}\left[ \log \left| \frac{r}{r-1} \right| +\pi i (1-\alpha)\right]
    + \frac{n}{2\ell_{\eta'}}\left[ \log \left| \frac{r-1}{r} \right| +\pi i (1+\beta)\right] \right)
    \\
    \times
     B\left( \frac{k}{2} - \frac{2\pi i n}{\ell_{\eta'}}, \frac{k}{2} + \frac{2\pi i n}{\ell_{\eta'}}\right)
     {_2}F_1\left( \frac{k}{2} - \frac{2\pi i m}{\ell_{\eta}}, \frac{k}{2} + \frac{2\pi i n}{\ell_{\eta'}};k;\frac{1}{r}\right).
\end{multline}
Also, if $r \not\in (0,1)$ or if $r \in (0,1)$ and $\alpha=-1$ then
\begin{multline} \label{ffv3}
    I_{\eta\eta'}(m,n;r,\alpha,\beta)= \frac{\sgn(r-1)^{k/2}}{\ell_{\eta'}} \left| \frac{r}{r-1} \right|^{k/4} e^{2\pi^2 n/\ell_{\eta'}}
    \\
    \times
    \e\left(\frac{m}{2\ell_\eta}\left[ \log \left| \frac{r}{r-1} \right| +\pi i (1-\alpha)\right]
    + \frac{n}{2\ell_{\eta'}}\left[ \log \left| \frac{r}{r-1} \right| +\pi i (1+\alpha \sgn(r-1))\right] \right)
    \\
    \times
     B\left( \frac{k}{2} - \frac{2\pi i n}{\ell_{\eta'}}, \frac{k}{2} + \frac{2\pi i n}{\ell_{\eta'}}\right)
     {_2}F_1\left( \frac{k}{2} + \frac{2\pi i m}{\ell_{\eta}}, \frac{k}{2} + \frac{2\pi i n}{\ell_{\eta'}};k;\frac{1}{1-r}\right).
\end{multline}
\end{prop}
\begin{proof}
Writing $u=t+i\pi/2$ and then $x=e^t$ in \eqref{iefa} gives
\begin{equation} \label{ksg}
    \frac{I_{\eta\eta'}(m,n;r,\alpha,\beta)}{(-i\alpha)^{k/2} e^{\pi^2n/\ell_{\eta'}}}=  \left| \frac{r}{r-1} \right|^{\pi i m /\ell_\eta} \int_{0}^{\infty}
    \frac{
    \left(\alpha \frac{  x - \alpha\sgn(r-1) \left| \frac{r-1}{r} \right|^{1/2} i }
    {  x  - \beta  \left| \frac{r}{r-1} \right|^{1/2} i} \right)^{2\pi i m /\ell_\eta}
    x^{k/2-2\pi i n /\ell_{\eta'}-1}
    }
    {\scriptstyle
    \left( x - \alpha\sgn(r-1) \left| \frac{r-1}{r} \right|^{1/2} i  \right)^{k/2}
    \left( x  - \beta  \left| \frac{r}{r-1} \right|^{1/2} i \right)^{k/2}}  \, \frac{dx}{\ell_{\eta'}}.
\end{equation}
With $\alpha = \pm 1$ and $x>0$ as above, we have
\begin{equation} \label{ksg2}
    \left( \alpha \frac{x+iu}{x+iv} \right)^w = e^{\pi i w (1-\alpha)/2} (x+iu)^w (x+iv)^{-w} \qquad (u,v \in \R, \ w\in \C)
\end{equation}
if $\alpha \frac{x+iu}{x+iv} \in \H$
since $|\arg(x+iu)|$, $|\arg(x+iv)| <\pi/2$. We may apply the identity \eqref{ksg2} to \eqref{ksg} since the quotient to be exponentiated is in $\H$, as can be verified by a direct check or by noting that it originates as $\se^{-1}\g \sep z$ in the proof of Theorem \ref{CIShh2}. Therefore
\begin{multline} \label{ffv}
    I_{\eta,\eta'}(m,n;r,\alpha,\beta)= \frac{(-i\alpha)^{k/2}}{\ell_{\eta'}} e^{\pi^2n/\ell_{\eta'}}
    \e\left(\frac{m}{2\ell_\eta}\left[ \log \left| \frac{r}{r-1} \right| +\pi i (1-\alpha)\right]\right)
    \\
    \times
    \int_{0}^{\infty}
    \left(   x +a \right)^{-k/2+2\pi i m /\ell_\eta}
    \left(  x  +b \right)^{-k/2-2\pi i m /\ell_\eta}
    x^{k/2-2\pi i n /\ell_{\eta'}-1}
      \, dx
\end{multline}
for $a=- \alpha \,\sgn(r-1) \bigl| \frac{r-1}{r} \bigr| ^{1/2} i$ and $b=- \beta  \bigl|  \frac{r}{r-1} \bigr|^{1/2} i$. The evaluation of the integral in \eqref{ffv} has some subtleties so we give it in the following lemma. Recall that ${_2}F_1(a,b;c;1-z)$ is a multi-valued function of $z$ in general, and by convention we take the principal branch with $-\pi<\arg z\lqs \pi$.

\begin{lemma}
Suppose $a,$ $b\in \C_{\neq 0}$ with $|\arg(a)|$, $|\arg(b)| <\pi$. For any $\mu,$ $\rho,$ $\nu \in \C$ with $0<\Re(\nu)<\Re(\mu+\rho)$ we have
\begin{multline} \label{phyg}
\int_0^\infty (x+a)^{-\mu} (x+b)^{-\rho} x^{\nu-1}\, dx  = B(\nu,\mu+\rho-\nu)\\
\times \begin{cases}
a^{\nu-(\mu+\rho)}  {_2}F_1(\rho,\mu+\rho-\nu;\mu+\rho;1-b/a) & \text{ \ if \ \ } -\pi < \arg b -\arg a \lqs \pi \\
b^{\nu-(\mu+\rho)}  {_2}F_1(\mu,\mu+\rho-\nu;\mu+\rho;1-a/b) & \text{ \ if \ \ } -\pi \lqs \arg b -\arg a < \pi
\end{cases}
\end{multline}
using the principal value of the hypergeometric function in \eqref{phyg}.
\end{lemma}
\begin{proof}
From \cite[3.197.1]{GR} we have
\begin{equation}\label{grihh}
    \int_0^\infty (x+a)^{-\mu} (x+b)^{-\rho} x^{\nu-1}\, dx = a^{-\mu} b^{\nu-\rho} B(\nu,\mu+\rho-\nu) {_2}F_1(\mu,\nu;\mu+\rho;1-b/a).
\end{equation}
 If $\arg b -\arg a \in (-\pi,\pi]$ then the right side of \eqref{grihh}
requires the principal branch of ${_2}F_1$. For   $\arg b -\arg a \not\in (-\pi,\pi]$ we require values of ${_2}F_1$ on the branch reached by crossing the branch-cut from above or below. Applying the Pfaff transformation \cite[Thm. 2.2.5]{AAR} to ${_2}F_1$ converts \eqref{grihh} into
\begin{align} \label{pfff}
     & a^{\nu-(\mu+\rho)} B(\nu,\mu+\rho-\nu) {_2}F_1(\rho,\mu+\rho-\nu;\mu+\rho;1-b/a) \\
   \text{or} \quad  & b^{\nu-(\mu+\rho)} B(\nu,\mu+\rho-\nu) {_2}F_1(\mu,\mu+\rho-\nu;\mu+\rho;1-a/b) \label{pfff2}
\end{align}
by switching $a$ and $b$. Clearly we remain in the principal branch of ${_2}F_1$ in \eqref{pfff} for $-\pi < \arg b -\arg a \lqs \pi$ and the principal branch of ${_2}F_1$ in \eqref{pfff2} for the overlapping range $-\pi < \arg a -\arg b \lqs \pi$. This proves the lemma.
\end{proof}

In our case, with $a$ and $b$ given after \er{ffv}, we have $\arg a$, $\arg b = \pm \pi/2$. Therefore,
$-\pi<\arg b-\arg a \lqs \pi$ unless $\beta=1$ and $\alpha \sgn(r-1)=-1$ which is equivalent to
\begin{equation}\label{gca}
     \alpha = \beta = -\sgn(r-1) =1,
\end{equation}
since it is not possible to have $-\alpha = \beta = \sgn(r-1) =1$. Note that \er{gca} implies $r$ is in the interval $(0,1)$. Hence we have $-\pi<\arg b-\arg a \lqs \pi$ if $r \not\in (0,1)$ or if $r \in (0,1)$ and $\alpha=-1$. In this case we may evaluate the integral in \eqref{ffv} using the top option in \er{phyg}, with for example
\begin{equation*}
    a^{\nu-(\mu+\rho)}= (-i \alpha)^{-k/2} \sgn(r-1)^{k/2} \left| \frac{r}{r-1} \right|^{k/4}
    \e\left(\frac{n}{2\ell_{\eta'}}\left[ \log \left| \frac{r}{r-1} \right| +\pi i \alpha \sgn(r-1)\right] \right).
\end{equation*}
The result is \er{ffv3}.
Similarly, if $r \not\in (0,1)$ or if $r \in (0,1)$ and $\alpha=1$, we may evaluate the integral in \eqref{ffv} using the bottom option in \er{phyg} and the result is \er{ffv2}.
\end{proof}

\subsection{Double cosets in the hyperbolic/hyperbolic case}
We need some preliminary material to understand the double cosets appearing in the Kloosterman sum $S_{\eta\eta'}$.
Let $L$ be a complete set of inequivalent representatives for $\G_\eta\backslash \G/\G_{\eta'}$. Partition $L$ into two subsets:
\begin{equation*}
    \G(\eta,\eta')_0 := \Big\{ \delta \in L \ \Big| \ \delta \eta' = \eta \text{ or } \eta^* \Big\}, \qquad \G(\eta,\eta') := \Big\{ \delta \in L \ \Big| \ \delta \eta' \neq \eta \text{ or } \eta^* \Big\}.
\end{equation*}

\begin{lemma} \label{2elts}
There exist $a,$ $b\in \R_{\neq 0}$ such that $\G(\eta,\eta')_0$ is a subset of
\begin{equation}\label{gg0n}
    \left\{ \se \left(\smallmatrix
    a & 0 \\ 0 & \frac1a
\endsmallmatrix\right) \sep^{-1}, \ \se \left(\smallmatrix
    0 & b \\ -\frac{1}{b} & 0
\endsmallmatrix\right) \sep^{-1}  \right\}.
\end{equation}
Then $\G(\eta,\eta')_0$ contains the first element of \eqref{gg0n} if  $\eta' \equiv \eta \bmod \G$ and the second if  $\eta' \equiv \eta^* \bmod \G$. The numbers $a$ and $b$ depend on the choice of the scaling matrices $\se$ and $\sep$.
\end{lemma}
\begin{proof}
If $\G(\eta,\eta')_0$ contains $\delta$  and $\tau$ such that  $\delta \eta'=\eta$ and $\tau \eta'=\eta$ then $\tau^{-1} \delta \in \G_{\eta'}$ and hence $\tau=\delta$. Similarly, if $\g$, $\tau \in \G(\eta,\eta')_0$ with $\g \eta'=\eta^*$ and $\tau \eta'=\eta^*$ then we must have $\tau=\g$ also. Therefore $\G(\eta,\eta')_0$ contains at most one element $\delta$ satisfying $\delta \eta'=\eta$ and at most one  $\g$ satisfying $\g \eta'=\eta^*$. If $\G(\eta,\eta')_0$ contains such a $\delta$ and such a $\g$ then they must be distinct since $\eta \neq \eta^*$.
The computations $\se^{-1} \delta \sep 0 = 0$ and $\se^{-1} \delta \sep \ci = \ci$ show $\delta$ takes the form of the first element of \eqref{gg0n} and similarly for $\g$ taking the form of the second element.
\end{proof}

It may be shown that
\begin{equation}\label{lorc}
    \ell_\eta = \ell_{\eta'} \qquad \text{if} \qquad \eta' \equiv \eta \bmod \G \qquad \text{or} \qquad \eta' \equiv \eta^*  \bmod \G.
\end{equation}
Then we see that the effect of a different choice of $L$ on the $a$ and $b$ in Lemma \ref{2elts} is multiplication by a factor of the form $\pm e^{m\ell_\eta/2}$ for $m \in \Z$. In other words, for fixed scaling matrices $\se$ and $\sep$, the sets $\log(a^2)+\ell_\eta\Z$ and $\log(b^2)+\ell_\eta\Z$ are independent of $L$.

\begin{prop} \label{hh_reps}
With the above notation,
\begin{equation} \label{hhjiu}
    \G(\eta,\eta')_0 \cup \Big\{ \g \tau \ \Big| \ \g \in \G(\eta,\eta'), \ \tau \in \G_{\eta'}/Z\Big\}
\end{equation}
is a complete set of inequivalent representatives for $\G_\eta\backslash \G$.
\end{prop}
\begin{proof}
The set $L':=\{ \delta\tau \ | \ \delta \in L, \ \tau \in \G_{\eta'}/Z  \}$ clearly gives a complete set of representatives for $\G_\eta\backslash \G$. To see which of its elements are equivalent modulo $\G_\eta$, suppose
\begin{equation}\label{hhjul}
    \G_\eta\delta\tau = \G_\eta \delta'\tau' \quad \text{ for } \quad \delta, \delta' \in L  \quad \text{and}  \quad \tau, \tau' \in \G_{\eta'}/Z.
\end{equation}
Arguing as in Proposition \ref{pp_reps},  we must have $\delta'=\delta$ and there exists $\g \in \G_\eta$ so that $\g \delta\tau = \delta\tau'$. It follows that $\g$ fixes $\eta$ and $\delta \eta'$. This can happen if $\g=\pm I$, in which case $\tau=\tau'$. The other possibility is that $\delta \eta' = \eta$ or $\eta^*$. In these cases  $\G_\eta\delta\tau = \G_\eta \delta'\tau' = \G_\eta \delta$. Hence, with \eqref{hhjiu}, we have removed all of the equivalent elements from the set $L'$ we started with.
\end{proof}

To give another characterization of the sets $\G(\eta,\eta')_0$ and $\G(\eta,\eta')$, we first prove the following two results, contained in \cite[Lemma 6 (iii)]{G83}.

\begin{lemma} \label{fix}
Suppose  $\g$ and $\delta$ are  hyperbolic elements of $\G$ and that $\g$ fixes $\eta_1$, $\eta_2$ while $\delta$ fixes $\eta_3$, $\eta_2$. Then $\eta_1=\eta_3$.
\end{lemma}
\begin{proof}
Suppose $\eta_1 \neq \eta_3$ and let $\eta=(\eta_1,\eta_2)$. We have $\se^{-1} \g \se = \left(\smallmatrix u
& 0 \\ 0 & 1/u
\endsmallmatrix\right)$ and $\se^{-1} \delta \se = \left(\smallmatrix v
& w \\ 0 & 1/v
\endsmallmatrix\right)$ for some $u$, $v$, $w$ in $\R_{\neq 0}$. Then  $\se^{-1} \g^k \delta \g^{-k} \se = \left(\smallmatrix v
& w \cdot u^{2k} \\ 0 & 1/v
\endsmallmatrix\right)$ for $k \in \Z$.
Applying these elements to $i \in \H$ gives $v^2 i + vw \cdot u^k$, with infinitely many points contained in a compact neighborhood of $v^2 i$.
But this is impossible since $\se^{-1} \G \se$ is a discrete group. Hence we must have  $\eta_1=\eta_3$.
\end{proof}

\begin{lemma} \label{86}
For $\delta \in \G$, write $\se^{-1} \delta \sep = \left(\smallmatrix a
& b \\ c & d
\endsmallmatrix\right)$. Then
\begin{gather}\label{enu1}
    b=0 \text{ or } c=0 \iff  b=c=0 \iff \delta \eta' = \eta \\
    a=0 \text{ or } d=0 \iff  a=d=0  \iff \delta \eta' = \eta^*. \label{enu2}
\end{gather}
\end{lemma}
\begin{proof} Write $\eta=(\eta_1,\eta_2)$, $\eta'=(\eta'_1,\eta'_2)$ and let $\g \in \G_\eta$ and $\g' \in \G_{\eta'}$.
Suppose that $b=0$. This implies
\begin{align*}
    \se^{-1} \delta \sep 0 = 0 & \implies \delta \eta'_1 = \eta_1 \\
    & \implies (\delta \g' \delta^{-1}) \eta_1 = \eta_1.
\end{align*}
Since $\delta \g' \delta^{-1}$ and $\g$ both fix $\eta_1$, they must both fix $\eta_2$ by Lemma \ref{fix}. It follows that $\delta \eta'=\eta$ and $c=0$. Similarly, starting with $c=0$ instead of $b=0$ we also find that $\delta \eta'=\eta$ and $b=0$.

Conversely, if $\delta \eta'=\eta$ then $\sep$ must be of the form $\delta^{-1} \se \left(\smallmatrix t
& 0 \\ 0 & 1/t
\endsmallmatrix\right)$ for some $t \in \R_{\neq 0}$. Hence $\se^{-1} \delta \sep$ has $b=c=0$.
This finishes the proof of \eqref{enu1}.

If $a=0$ or $d=0$  or $\delta \eta' = \eta^*$ we may choose  $\sigma_{(\eta')^*} = \sep S$. Applying \eqref{enu1} to
\begin{equation*}
    \se^{-1} \delta \sigma_{(\eta')^*}  = \se^{-1} \delta \sep S = \left(\smallmatrix a'
& b' \\ c' & d'
\endsmallmatrix\right)
\end{equation*}
implies
\begin{equation*}
    b'=0 \text{ or } c'=0 \iff  b'=c'=0 \iff \delta (\eta')^* = \eta
\end{equation*}
which is equivalent to \eqref{enu2}
\end{proof}

The next corollary follows directly.

\begin{cor} \label{hh_reps2} We have
\begin{align*}
    \G(\eta,\eta')_0  & = \left\{  \delta  \in L \ \left| \ \se^{-1} \delta \sep = \begin{pmatrix} a  & b \\ c & d \end{pmatrix} \text{ with } abcd = 0 \right.\right\}, \\
    \G(\eta,\eta')  & = \left\{  \delta  \in L \ \left| \ \se^{-1} \delta \sep = \begin{pmatrix} a  & b \\ c & d \end{pmatrix} \text{ with } abcd\neq 0 \right.\right\}.
\end{align*}
\end{cor}

  Good's decomposition of $\SL_2(\R)$ in this hyperbolic/hyperbolic case, see \cite[Lemma 1]{G83} and \cite[Lemma 1]{G85}, says the following.
\begin{lemma} \label{bruhh}
Let $M=\left(\smallmatrix a
& b \\ c & d
\endsmallmatrix\right) \in \SL_2(\R)$.
\begin{enumerate}
\item
When $|ad|+|bc| \neq 1$ we have
\begin{multline} \label{hh}
    \begin{pmatrix} a & b \\ c & d \end{pmatrix} = \frac{\sgn(a)}{2}
    \begin{pmatrix} \left| \frac{ab}{cd} \right|^{1/4} & 0 \\ 0 & \left| \frac{ab}{cd} \right|^{-1/4} \end{pmatrix}
    \begin{pmatrix} 1 & -\sgn(ac)  \\ \sgn(ac)  & 1  \end{pmatrix}\\
    \times
    \begin{pmatrix} \nu & 0 \\ 0 & 1/\nu \end{pmatrix}
    \begin{pmatrix} 1 & \sgn(cd)  \\ -\sgn(cd)  & 1  \end{pmatrix}
    \begin{pmatrix} \left| \frac{bd}{ac} \right|^{-1/4} & 0 \\ 0 & \left| \frac{bd}{ac} \right|^{1/4} \end{pmatrix}
\end{multline}
for $\nu = \sideset{_\text{hyp}}{_\text{hyp}}{\opv}(M) = |ad|^{1/2}+ |bc|^{1/2} $.
\item
When $|ad|+|bc| = 1$ and $abcd \neq 0$ we have
\begin{equation} \label{hh2}
    \begin{pmatrix} a & b \\ c & d \end{pmatrix} = -\sgn(c)
    \begin{pmatrix} \left| \frac{ab}{cd} \right|^{1/4} & 0 \\ 0 & \left| \frac{ab}{cd} \right|^{-1/4} \end{pmatrix}
    \begin{pmatrix} \cos \theta/2 & \sin \theta/2  \\ -\sin \theta/2  & \cos \theta/2  \end{pmatrix}
    \begin{pmatrix} \left| \frac{bd}{ac} \right|^{-1/4} & 0 \\ 0 & \left| \frac{bd}{ac} \right|^{1/4} \end{pmatrix}
\end{equation}
for $\theta = \theta(M) = 2\cos^{-1}\left(-\sgn(ac)|ad|^{1/2}\right)$ and $0<\theta<2\pi$.
\end{enumerate}
\end{lemma}
\begin{proof}
 Let $r=ad$ so that $bc=r-1$ and $|ad|+|bc| \neq 1$ is equivalent to $r \notin [0,1]$. The identity \eqref{hh} in (i) follows from a direct calculation, reducing to $|r-1|+\sgn(r)=|r|$ or $|r-1|+\sgn(r-1)=|r|$. Part (ii) corresponds to $r \in (0,1)$ and may be easily verified also.
\end{proof}

Based on the above decomposition we define
\begin{equation} \label{ree}
R_{\eta\eta'}  :=  \left\{ \left. \begin{pmatrix} a  & b \\ c & d \end{pmatrix} \in \se^{-1}\G \sep \ \right|  \ abcd \neq 0, \ \frac{1}{\varepsilon_\eta} \lqs \left| \frac{ab}{cd} \right|^{1/2} < \varepsilon_\eta, \ \frac{1}{\varepsilon_{\eta'}} \lqs \left| \frac{bd}{ac} \right|^{1/2} < \varepsilon_{\eta'}  \right\}
\end{equation}
and a similar proof to Lemma \ref{ac01} shows the next result.
\begin{lemma} \label{ac01hh}
We may take $\se^{-1} \G(\eta,\eta') \sep = R_{\eta\eta'}/Z$.
\end{lemma}

\subsection{The hyperbolic/hyperbolic Kloosterman sum}

Recall from \eqref{cetet} that $C_{\eta\eta'}=\left\{ad \ \left| \ \left(\smallmatrix a
& b \\ c & d
\endsmallmatrix\right) \in \se^{-1}\G\sep, \ abcd \neq 0 \right. \right\}$.
For $C \in C_{\eta\eta'}$ and $\alpha$, $\beta=\pm1$ define
\begin{equation}\label{kloosthh2}
S_{\eta\eta'}(m,n;C,\alpha,\beta):= \sum_{\substack{ \g \in \G_\eta \backslash \G / \G_{\eta'}, \ \left(\smallmatrix a
& b \\ c & d
\endsmallmatrix\right) = \se^{-1}\g\sep
\\ ad=C, \ \sgn(ac)=\alpha, \ \sgn(cd)=\beta} } \e\left(\frac{m}{2\ell_{\eta}}  \log \left| \frac{ab}{cd} \right| +  \frac{n}{2\ell_{\eta'}} \log \left| \frac{ac}{bd} \right|\right).
\end{equation}
Then $S_{\eta\eta'}(m,n;C,\alpha,\beta)$ is related to Good's generalized Kloosterman sum \eqref{kloost} by
\begin{equation} \label{hhgo}
    S_{\eta\eta'}(m,n;C,\alpha,\beta) =
    \sideset{_\eta^{\delta}}{_{\eta'}^{\delta'}}{\opS}(m,n;|C|^{1/2}+|C-1|^{1/2})  \quad \text{for} \quad \delta = \frac{1-\alpha}{2}, \ \delta' = \frac{1+\beta}{2}
\end{equation}
when $C$ is not in the interval $(0,1)$. For $C \in (0,1)$, Good made the right side of \eqref{hhgo} zero and treated this case separately with another sum: $\sideset{_\eta^{}}{_{\eta'}^{}}{\opsmall}(m,n;\theta)$ for $\theta$ as in \eqref{hh2}. See \cite[(5.11)]{G83}.

To show that \eqref{kloosthh2} is a finite sum, and to bound it, we start with the following analog of \cite[Lemma 1.24]{S71}.

\begin{lemma} \label{fic}
Given $M>0$, there are only finitely many double cosets $\G_\eta \g \G_{\eta'}$ where $\left(\smallmatrix
a & b \\ c & d
\endsmallmatrix\right) = \se^{-1}\g\sep$ has $|abcd| \lqs M$. Note that $|abcd|$ is independent of the double coset representative and also the choice of scaling matrices.
\end{lemma}
\begin{proof}
There are at most two double cosets with $abcd=0$ by Lemma \ref{2elts} and Corollary \ref{hh_reps2}. Assume now that $abcd \neq 0$. Since $\left\{ \left( \left. \smallmatrix \varepsilon_\eta^m
& 0 \\ 0 & \varepsilon_\eta^{-m}
\endsmallmatrix\right)
\ \right| \ m\in \Z \right\} \subseteq \se^{-1}\G_\eta\se$ and $\left\{\left( \left. \smallmatrix
\varepsilon_{\eta'}^n & 0 \\ 0 & \varepsilon_{\eta'}^{-n}
\endsmallmatrix\right)
\ \right| \ n\in \Z \right\} \subseteq \sep^{-1}\G_{\eta'}\sep$ we may choose representatives $\delta$ for $\G_\eta \g \G_{\eta'}$ satisfying
\begin{equation*}
    \se^{-1}\delta\sep = \begin{pmatrix} \varepsilon_\eta^m & 0 \\ 0 & \varepsilon_\eta^{-m} \end{pmatrix}
\begin{pmatrix} a  & b \\ c & d \end{pmatrix}
\begin{pmatrix} \varepsilon_{\eta'}^n & 0 \\ 0 & \varepsilon_{\eta'}^{-n} \end{pmatrix}
\end{equation*}
so that
\begin{equation*}
    \se^{-1}\delta \se \left( \se^{-1} \sep i \right)= \varepsilon_\eta^{2m} \frac{a \varepsilon_{\eta'}^{2n}i + b}{c \varepsilon_{\eta'}^{2n}i + d}.
\end{equation*}
We will show that  distinct double cosets satisfying the statement of the lemma give distinct elements in the discrete group $\se^{-1}\G \se$ mapping $z_0:=\se^{-1} \sep i \in \H$ into a compact set  $K \subset \H$ of the form
\begin{equation*}
    K = \left\{ r e^{i \theta} \ \left| \ 1 \lqs r \lqs \varepsilon_\eta^2, \ \theta_1 \lqs \theta \lqs \theta_2\right. \right\}
\end{equation*}
with $\theta_1$, $\theta_2$ depending only on $\eta'$ and $M$. This forces the number of double cosets to be finite.

Choose $n \in \Z$ so that $\lambda:=\varepsilon_{\eta'}^{2n}$ satisfies
\begin{equation} \label{qkl}
    \left| \frac{bd}{ac}\right|^{1/2} \lqs \lambda < \left| \frac{bd}{ac}\right|^{1/2} \varepsilon_{\eta'}^{2}.
\end{equation}
We have
\begin{equation*}
    \arg\left( \frac{a \lambda i + b}{c \lambda i + d}\right) = \arg\left( ac \lambda +bd/\lambda +i \right)
\end{equation*}
and, using \eqref{qkl},
\begin{equation*}
    \left| ac \lambda +bd/\lambda \right| \lqs M^{1/2}(1+ \varepsilon_{\eta'}^{2}).
\end{equation*}
Hence $\arg(\se^{-1}\delta \se z_0)$ is bounded between constants $\theta_1$, $\theta_2$ that depend only on $\eta'$ and $M$. Choose $m \in \Z$ so that $1 \lqs |\se^{-1}\delta \se z_0| < \varepsilon_\eta^2$ and $\se^{-1}\delta \se z_0$ is contained in the compact set $K$ as required.
\end{proof}

\begin{cor} \label{811}  Given two hyperbolic fixed point pairs $\eta$ and $\eta'$ for $\G$, there exists $M_{\eta\eta'}>0$ with the following properties. For all $\left(\smallmatrix
a & b \\ c & d
\endsmallmatrix\right) \in \se^{-1}\G\sep$ we have
\begin{enumerate}
   \item $ |bc| \gqs M_{\eta\eta'}$ \ if \  $bc \neq 0$,
   \item $ |ad| \gqs M_{\eta\eta'}$ \ if \  $ad \neq 0$,
    \item $ |abcd| \gqs M_{\eta\eta'}^2$ \ if \ $abcd \neq 0$.
\end{enumerate}
\end{cor}
\begin{proof}
Consider a double coset $\G_\eta \g \G_{\eta'}$ with $\left(\smallmatrix
a & b \\ c & d
\endsmallmatrix\right) = \se^{-1}\g\sep$. If $\delta \in \G_\eta \g \G_{\eta'}$ has $\left(\smallmatrix
a' & b' \\ c' & d'
\endsmallmatrix\right) = \se^{-1}\delta\sep$
then $b'c'=bc$. So distinct values of $bc$ correspond to different double cosets. Take any $N>0$ and we want to examine the possible values for $|bc| \in [0,N]$ where $\left(\smallmatrix
a & b \\ c & d
\endsmallmatrix\right) \in \se^{-1}\G\sep$.
Clearly $|abcd|=|bc(bc+1)|\lqs N(N+1)$. It follows from Lemma \ref{fic} that there are only finitely many values for $|bc| \in [0,N]$. Hence the nonzero ones are bounded from below, proving part (i). The proof of (ii) is similar and we may take $M_{\eta\eta'}$ as the smaller of the two lower bounds. Part (iii) is a consequence of (i) and (ii).
\end{proof}

We next set
\begin{equation*}
    \mathcal N_{\eta\eta'}(C)  :=  \#\left\{ \g \in \G_\eta \backslash \G / \G_{\eta'} \ \left| \ \se^{-1}\g\sep = \begin{pmatrix} a  & b \\ c & d \end{pmatrix} \text{ with } ad = C\right. \right\}.
\end{equation*}
Then $\mathcal N_{\eta\eta'}(C)$ is well defined and independent of the scaling matrices $\se$ and $\sep$. It bounds the number of terms in the sum \eqref{kloosthh2}, though at the outset it may be infinite.

\begin{prop} \label{812}
With the above notation
\begin{equation*}
    \sum_{C \in C_{\eta\eta'}, \ |C| \lqs X}  \mathcal N_{\eta\eta'}(C) \ll X^{3/2}.
\end{equation*}
\end{prop}
\begin{proof}
We may write $\mathcal N_{\eta\eta'}(C)$ more explicitly as $\#\left\{ \left. \left(\smallmatrix
a & b \\ c & d
\endsmallmatrix\right) \in R_{\eta\eta'}/Z \ \right| \ ad = C\right\}$. Also let
\begin{equation*}
    R(X):=\left\{ \left. \begin{pmatrix} a  & b \\ c & d \end{pmatrix}  \in R_{\eta\eta'} \ \right| \ |ad| \lqs X \right\} \subset \se^{-1}\G \sep.
\end{equation*}
Suppose $\g=\left(\smallmatrix
a & b \\ c & d
\endsmallmatrix\right)$ and $\delta=\left(\smallmatrix
a' & b' \\ c' & d'
\endsmallmatrix\right)$  are in $R(X)$. Then $\g \delta^{-1} = \left(\smallmatrix
a'' & b'' \\ c'' & d''
\endsmallmatrix\right) \in \se^{-1}\G\se$ for
\begin{equation*}
    |b''c''|=\left| dd' aa'\left(\frac{c}{d}-\frac{c'}{d'}\right) \left(\frac{b}{a}-\frac{b'}{a'}\right)\right|.
\end{equation*}
If $b''c''=0$ then $\g \delta^{-1} \in  \se^{-1}\G_\eta\se$ by Lemma \ref{86} and so we must have $\g=\delta$. Otherwise, it follows that $|b''c''| \gqs M_{\eta\eta}>0$ by Corollary \ref{811}. Hence
\begin{equation} \label{topb}
    \left| \frac{c}{d}-\frac{c'}{d'}\right| \left|\frac{b}{a}-\frac{b'}{a'}\right| \gqs \frac{M_{\eta\eta}}{X^2}.
\end{equation}

We next determine how large $ \left| \frac{c}{d}\right|$ and  $ \left| \frac{b}{a}\right|$ can be for $\left(\smallmatrix
a & b \\ c & d
\endsmallmatrix\right) \in R(X)$. Combining the inequalities in the definition \eqref{ree} implies
\begin{equation} \label{inq1}
    \frac{1}{\varepsilon_\eta \varepsilon_{\eta'}} \lqs \left| \frac{b}{c}\right| \lqs \varepsilon_\eta \varepsilon_{\eta'}, \qquad
    \frac{1}{\varepsilon_\eta \varepsilon_{\eta'}} \lqs \left| \frac{a}{d}\right| \lqs \varepsilon_\eta \varepsilon_{\eta'}.
\end{equation}
We also know that
\begin{equation} \label{inq2}
   M_{\eta\eta'} \lqs \left| bc \right| \lqs X+1, \qquad
    M_{\eta\eta'} \lqs \left| ad \right| \lqs X.
\end{equation}
Together \eqref{inq1} and \eqref{inq2} prove
\begin{equation} \label{inq3}
    \left| \frac{c}{d}\right|, \ \left| \frac{b}{a}\right| \lqs \frac{\varepsilon_\eta \varepsilon_{\eta'} (X+1)^{1/2}}{M_{\eta\eta'}^{1/2}}.
\end{equation}
 Use \eqref{inq3} in \eqref{topb} to bound $\left|\frac{b}{a}-\frac{b'}{a'}\right|$ and show
\begin{equation} \label{inq4}
    \left| \frac{c}{d}-\frac{c'}{d'}\right| \gqs \frac{M_{\eta\eta}}{X^2} \times \frac{ M_{\eta\eta'}^{1/2}}{2\varepsilon_\eta \varepsilon_{\eta'} (X+1)^{1/2}}
\end{equation}
for any two distinct $\left(\smallmatrix
a & b \\ c & d
\endsmallmatrix\right)$ and $\left(\smallmatrix
a' & b' \\ c' & d'
\endsmallmatrix\right)$  in $R(X)$. Since we have seen in \eqref{inq3} that $c/d$ is restricted to a finite interval, it follows from \eqref{inq4} that $R(X)$ has a finite number of elements, say $n$. List the corresponding fractions as $c_1/d_1<c_2/d_2< \dots < c_n/d_n$.  Then
\begin{equation} \label{gfd}
    \sum_{j=1}^{n-1} \left| \frac{c_{j+1}}{d_{j+1}}-\frac{c_{j}}{d_{j}}\right|
    = \sum_{j=1}^{n-1} \left(\frac{c_{j+1}}{d_{j+1}}-\frac{c_{j}}{d_{j}}  \right)
    \lqs 2\frac{\varepsilon_\eta \varepsilon_{\eta'} (X+1)^{1/2}}{M_{\eta\eta'}^{1/2}}
\end{equation}
using \eqref{inq3}.
With \eqref{gfd} and the inequality of the arithmetic and geometric means, we have
\begin{equation}\label{amgm}
    \prod_{j=1}^{n-1} \left| \frac{c_{j+1}}{d_{j+1}}-\frac{c_{j}}{d_{j}}\right| \ll \left(\frac{(X+1)^{1/2}}{n-1}\right)^{n-1}
\end{equation}
and the same bound holds for $\prod_{j=1}^{n-1} \left| \frac{b_{j+1}}{a_{j+1}}-\frac{b_{j}}{a_{j}}\right|$ by a similar argument.
Combining these bounds with \eqref{topb} shows
\begin{equation*}
    \left(\frac{1}{X^2}\right)^{n-1} \ll \left(\frac{X+1}{(n-1)^2}\right)^{n-1}
\end{equation*}
and therefore $n \ll X^{3/2}$,
as desired.
\end{proof}

\begin{cor}
For  implied constants depending only on $\G$, $\eta$ and $\eta'$
\begin{equation*} 
    \mathcal N_{\eta\eta'}(C) \ll C^{3/2}, \qquad S_{\eta\eta'}(m,n;C,\alpha,\beta) \ll C^{3/2}, \qquad \#\{C \in C_{\eta\eta'} \ : \ |C| \lqs X\}   \ll X^{3/2}.
\end{equation*}
\end{cor}

\subsection{The hyperbolic expansion of $P_{\eta,m}$}

\begin{theorem} \label{CIShh2}
Recall the numbers  $a$ and $b$  from Lemma \ref{2elts}. For $m$, $n \in \Z$, the $n$th hyperbolic Fourier coefficient at $\eta'$ of the hyperbolic Poincar\'e series $P_{\eta,m}$ is given by
\begin{align}
c_{\eta'}(n;P_{\eta,m})  =  & \sum_{C \in C_{\eta\eta'}, \ \alpha,\beta=\pm 1}
 I_{\eta\eta'}(m,n;C,\alpha,\beta) \frac{S_{\eta\eta'}(m,n;C,\alpha,\beta)}{|C(C-1)|^{k/4}} \notag\\
  \label{xab1}
     & \qquad + \begin{cases}
 (a^2)^{2\pi i n/\ell_{\eta'}}  \text{ \ if \ }\eta' \equiv \eta \bmod \G  \text{ and } n=m,
 \end{cases}\\
   & \qquad + \begin{cases}
 (-1)^{k/2} e^{2\pi^2 n/\ell_{\eta'}}(b^2)^{-2\pi i n/\ell_{\eta'}}   \text{ \ if \ }\eta' \equiv \eta^* \bmod \G  \text{ and }  n=-m.
 \end{cases} \label{yab1}
\end{align}
\end{theorem}
\begin{proof}
We have
\begin{equation} \label{spar}
    \left(P_{\eta,m}|_k \sep\right)(z) = \sum_{\g \in \G_\eta \backslash \G }
    \frac{(\se^{-1}\g \sep z)^{-k/2+2\pi i m /\ell_\eta}}
    {j(\se^{-1}\g \sep, z)^{k}}
\end{equation}
which is absolutely   convergent for $z$  in $\H$ and $k>2$.
We use the set of representatives for $\G_\eta \backslash \G$ given by Proposition \ref{hh_reps}.
The elements of $\G(\eta,\eta')_0$, as described in Lemma \ref{2elts}, easily yield the contributions \er{xab1} and \er{yab1} -- using \er{lorc} and for \er{yab1} that $(-1/z)^s = e^{\pi i s} z^{-s}$ for all $z\in \H$ and $s\in \C$.

Write the remaining terms in \eqref{spar} as
\begin{multline} \label{rem}
  \sum_{\g \in \G(\eta,\eta')} \sum_{\tau \in \G_{\eta'}/Z }
    \frac{(\se^{-1}\g \tau \sep z)^{-k/2+2\pi i m /\ell_\eta}}
    {j(\se^{-1}\g \sep, z)^{k}} \\
    = \sum_{C \in C_{\eta\eta'}}
    \sum_{\alpha,\beta = \pm 1}
    \sum_{\substack{ \g \in \G(\eta,\eta'), \ \left(\smallmatrix a
& b \\ c & d
\endsmallmatrix\right) = \se^{-1}\g\sep
\\
 ad=C, \ \sgn(ac)=\alpha, \ \sgn(cd)=\beta} }
\sum_{\tau \in \G_{\eta'}/Z }
    \frac{(\se^{-1}\g \tau \sep z)^{-k/2+2\pi i m /\ell_\eta}}
    {j(\se^{-1}\g \tau\sep, z)^{k}}.
\end{multline}
The inner series is
\begin{multline} \label{intyy}
   \sum_{n\in \Z}
    \frac{(\se^{-1}\g \sep (e^{n \ell_{\eta'}} z))^{-k/2+2\pi i m /\ell_\eta}}
    {j(\se^{-1}\g \sep , e^{n \ell_{\eta'}} z)^{k} e^{-n \ell_{\eta'} k/2}}  \\
   =  \sum_{n\in \Z} \int_{-\infty}^\infty
    \frac{(\se^{-1}\g \sep (e^{\ell_{\eta'} t + A}))^{-k/2+2\pi i m /\ell_\eta}}
    {j(\se^{-1}\g \sep , e^{\ell_{\eta'} t + A})^{k}} e^{\ell_{\eta'} t k/2 -2\pi i n t} \, dt
\end{multline}
where  $z=e^A$ for $0<\Im A< \pi$ and we used Poisson summation which may be justified as in Theorem \ref{CISph}.
Here $
\se^{-1}\g\sep = \left(\smallmatrix a
& b \\ c & d
\endsmallmatrix\right)
$
with $abcd \neq 0$ and
the integral in \eqref{intyy}  equals
\begin{equation} \label{intyy2}
    \int_{-\infty}^\infty
    \frac{\left(\frac{a e^{\ell_{\eta'} t + A} +b}{c e^{\ell_{\eta'} t + A}+d} \right)^{2\pi i m /\ell_\eta}}
    {(a e^{\ell_{\eta'} t + A} +b)^{k/2}(c e^{\ell_{\eta'} t + A}+d)^{k/2}} e^{t \ell_{\eta'}(  k/2 -2\pi i n/\ell_{\eta'})} \, dt.
\end{equation}
Substitute $u=\ell_{\eta'} t +A+\frac 12 \log \left| \frac{ac}{bd} \right|$ and \eqref{intyy2} equals
\begin{multline} \label{intyy3}
    z^{-k/2+2\pi i n /\ell_{\eta'}} \left| \frac{ac}{bd} \right|^{-k/4}
    \e\left(\frac{m}{2\ell_{\eta}}  \log \left| \frac{ab}{cd} \right| +  \frac{n}{2\ell_{\eta'}}  \log \left| \frac{ac}{bd} \right|\right)\\
    \times
    \int_{-\infty+i \Im A}^{\infty+i \Im A}
    \frac{\left( \left| \frac{cd}{ab} \right|^{1/2}  \frac{a \left| \frac{bd}{ac} \right|^{1/2} e^u +b}{c \left| \frac{bd}{ac} \right|^{1/2} e^u +d} \right)^{2\pi i m /\ell_\eta} e^{u(k/2-2\pi i n /\ell_{\eta'})}}
    {(a \left| \frac{bd}{ac} \right|^{1/2} e^u +b)^{k/2} (c \left| \frac{bd}{ac} \right|^{1/2} e^u +d)^{k/2}}  \, \frac{du}{\ell_{\eta'}}.
\end{multline}
The integrand is holomorphic for $0<\Im u<\pi$ and therefore independent of $\Im A$ provided $0<\Im A<\pi$.
The equalities
\begin{align*}
    a \left| \frac{bd}{ac} \right|^{1/2} e^{u} +b  & = \left| \frac{b}{c} \right|^{1/2} |ad|^{1/2}
    \left( \sgn(a)  e^{u} + \sgn(b) \left| \frac{bc}{ad} \right|^{1/2} \right),\\
    c \left| \frac{bd}{ac} \right|^{1/2} e^{u} +d  & =
    \left| \frac{d}{a} \right|^{1/2} |bc|^{1/2}
    \left( \sgn(c)  e^{u}  + \sgn(d)  \left| \frac{ad}{bc} \right|^{1/2}\right)
\end{align*}
show the integral in \eqref{intyy3} is
\begin{equation} \label{intyy4}
    \left| \frac{ac}{bd} \right|^{k/4} \frac{1}{|abcd|^{k/4}}
    \int_{-\infty+i y}^{\infty+i y}
    \frac{
    \left( \left| \frac{ad}{bc} \right|^{1/2}
    \frac{\sgn(a)  e^{u} + \sgn(b) \left| \frac{bc}{ad} \right|^{1/2} }
    {\sgn(c)  e^{u}  + \sgn(d)  \left| \frac{ad}{bc} \right|^{1/2}} \right)^{2\pi i m /\ell_\eta}
    e^{u(k/2-2\pi i n /\ell_{\eta'})}
    }
    {\scriptstyle
    \left( \sgn(a)  e^{u} + \sgn(b) \left| \frac{bc}{ad} \right|^{1/2} \right)^{k/2}
    \left( \sgn(c)  e^{u}  + \sgn(d)  \left| \frac{ad}{bc} \right|^{1/2} \right)^{k/2}}  \, \frac{du}{\ell_{\eta'}}
\end{equation}
for any $y$ with $0<y<\pi$. Finally, multiplying through by $\sgn(c)$, \eqref{intyy3} is now
\begin{equation} \label{intyy5}
    z^{-k/2+2\pi i n /\ell_{\eta'}}
    \e\left(\frac{m}{2\ell_{\eta}} \log \left| \frac{ab}{cd} \right| +  \frac{n}{2\ell_{\eta'}}  \log \left| \frac{ac}{bd} \right|\right)
    \frac{I_{\eta \eta'}(m,n;C,\sgn(ac),\sgn(cd))}{|C(C-1)|^{k/4}}.
\end{equation}
Hence \eqref{rem} is
\begin{multline} \label{rem2}
\sum_{C \in C_{\eta\eta'}}
    \sum_{\alpha,\beta = \pm 1}
    \sum_{\substack{ \g \in \G(\eta,\eta'), \ \left(\smallmatrix a
& b \\ c & d
\endsmallmatrix\right) = \se^{-1}\g\sep
\\
 ad=C, \ \sgn(ac)=\alpha, \ \sgn(cd)=\beta} }
\sum_{n \in \Z }
    z^{-k/2+2\pi i n /\ell_{\eta'}}
    \\
    \times
    \e\left(\frac{m}{2\ell_{\eta}} \log \left| \frac{ab}{cd} \right| +  \frac{n}{2\ell_{\eta'}}  \log \left| \frac{ac}{bd} \right|\right)
    \frac{I_{\eta \eta'}(m,n;C,\alpha,\beta)}{|C(C-1)|^{k/4}}.
\end{multline}
With Proposition \ref{bndhh} we  have
\begin{equation*}
    z^{-k/2+2\pi i n /\ell_{\eta'}} I_{\eta\eta'}(m,n;C,\alpha,\beta) \ll e^{-\varepsilon |n|}
\end{equation*}
for $\varepsilon >0$ depending on $z$. Therefore
 \eqref{rem2} is majorized by a constant times
$
     \sum_{C \in C_{\eta\eta'}} \mathcal  |C|^{-k/2} N_{\eta\eta'}(C)
$
and thus convergent for $k>3$ by Proposition \ref{812}. This proves that changing the order of summation in \eqref{rem2} is  valid.
Rearranging completes the proof.
\end{proof}

\begin{proof}[Proof of Theorem \ref{CIShh}]
Set
\begin{multline} \label{mst}
    S^\star_{\eta\eta'}(m,n;C,\alpha) := S_{\eta\eta'}(m,n;C,\alpha,\beta)\\
    \times
    \e\left(\frac{m}{2\ell_\eta}\left[ \log \left| \frac{C}{C-1} \right| +\pi i (1-\alpha)\right]
    + \frac{n}{2\ell_{\eta'}}\left[ \log \left| \frac{C-1}{C} \right| +\pi i (1+\beta)\right] \right)
\end{multline}
for $\beta=\alpha \sgn(C)$ and this agrees with our earlier definition \eqref{klooshh}.
Combining Theorem \ref{CIShh2} with Proposition \ref{long} and \eqref{mst} gives Theorem \ref{CIShh}.
\end{proof}

By choosing the scaling matrices $\se$ and $\sep$ we can make $a$ and $b$ in the statements of Theorems \ref{CIShh} and \ref{CIShh2}  explicit as follows.
\begin{prop} \label{trpr}
\begin{enumerate}
\item Suppose $\eta$ and $\eta^*$ are not $\G$-equivalent. If $\eta' =  \rho \eta$ for some $\rho \in \G$ put $\sep :=  \rho \se$ and if $\eta' =  \rho \eta^*$ for some $\rho \in \G$ put $\sep :=  \rho \se S$.  In this case  \eqref{xab1}, \eqref{yab1} become
\begin{align} \label{xab2}
    + & \begin{cases}
 1  \text{ \ if \ }\eta' \equiv \eta \bmod \G  \text{ and } n=m,
 \end{cases}\\
  + & \begin{cases}
 (-1)^{k/2} e^{2\pi^2 n/\ell_{\eta'}}  \text{ \ if \ }\eta' \equiv \eta^* \bmod \G  \text{ and }  n=-m.
 \end{cases} \label{yab2}
\end{align}
\item Suppose $\eta$ and $\eta^*$ are  $\G$-equivalent. If $\tau \eta^* = \eta$ for $\tau \in \G$. It follows that $\se^{-1}\tau \se = \left(\smallmatrix 0 & t \\ -\frac1t & 0
\endsmallmatrix\right)$ for some $t\in \R_{\neq 0}$. If $\eta' \equiv \eta \bmod \G$ with $\eta' = \rho \eta$, choose $\sep = \rho\se$.
Then  \eqref{xab1}, \eqref{yab1} become
\begin{align} \label{xab3}
    + & \begin{cases}
 1  \text{ \ if \ }\eta' \equiv \eta \bmod \G  \text{ and } n=m,
 \end{cases}\\
  + & \begin{cases}
 (-1)^{k/2} e^{2\pi^2 n/\ell_{\eta'}}(t^2)^{-2\pi i n/\ell_{\eta'}}  \text{ \ if \ }\eta' \equiv \eta^* \bmod \G  \text{ and }  n=-m.
 \end{cases} \label{yab3}
\end{align}
\end{enumerate}
\end{prop}

\begin{proof} 
To prove part (ii), note that
  $\G(\eta,\eta')_0=\{\delta, \ \g\}$ (if $\eta' \equiv \eta \bmod \G$) with
\begin{equation*}
    \delta \eta'=\eta, \quad \delta=\se \left(\smallmatrix
    a & 0 \\ 0 & \frac1a
\endsmallmatrix\right) \sep^{-1}, \quad  \g \eta'=\eta^*, \quad \g=\se \left(\smallmatrix
    0 & b \\ -\frac{1}{b} & 0
\endsmallmatrix\right) \sep^{-1}.
\end{equation*}
Clearly $\delta \rho \in \G_\eta$. Then $$\left(\smallmatrix
    a & 0 \\ 0 & \frac1a
\endsmallmatrix\right) = \se^{-1}\delta\sep = \se^{-1}\delta \rho\se \in \se^{-1}\G_\eta\se
$$ and \eqref{xab3} follows from \eqref{xab1}.
To show \eqref{yab3}, we note that $\g \rho \tau \eta^*=\eta^*$ implying $\g \rho \tau \in \G_\eta$. Hence
\begin{align*}
    \begin{pmatrix}  -\frac bt & 0 \\ 0 & -\frac tb  \end{pmatrix} = \begin{pmatrix}  0 & b \\ -\frac 1b & 0 \end{pmatrix}
    \begin{pmatrix}  0 & t \\ -\frac 1t & 0 \end{pmatrix} & = (\se^{-1} \g \sep)(\se^{-1} \tau \se) \\
    & = \se^{-1} \g \rho \tau \se \in \se^{-1} \G_\eta \se
\end{align*}
so that $b^2 = t^2 e^{r \ell_\eta}$ for some $r\in \Z$.
Then \eqref{yab3} follows from \eqref{yab1}.
The proof of part (i) is similar.
\end{proof}

With \eqref{epe},  the identity $\s{P_{\eta',n}}{P_{\eta,m}} = \overline{\s{P_{\eta,m}}{P_{\eta',n}}}$ implies
\begin{equation}\label{epe2}
c_{\eta'}(n;P_{\eta,m}) \frac{ \ell_{\eta'}   e^{-2\pi^2 n/\ell_{\eta'}}}{\bigl|\G\bigl(\frac{k}{2}+\frac{2\pi i n}{\ell_{\eta'}} \bigr)\bigr|^2  }
=
  \overline{c_{\eta}(m;P_{\eta',n})} \frac{\ell_\eta   e^{-2\pi^2 m/\ell_\eta}}{ \bigl|\G\bigl(\frac{k}{2}+\frac{2\pi i m}{\ell_\eta} \bigr)\bigr|^2  }.
\end{equation}
To check that our formulas satisfy this symmetry, first note that $C_{\eta\eta'} = C_{\eta'\eta}$ and
\begin{equation*}
    S_{\eta\eta'}(m,n;C,\alpha,\beta) = \overline{S_{\eta'\eta}(n,m;C,-\beta,-\alpha)}.
\end{equation*}
It follows that \eqref{epe2} is a consequence of Theorem \ref{CIShh2} if we can show that
\begin{equation} \label{ver}
    I_{\eta\eta'}(m,n;C,\alpha,\beta)
    \frac{ \ell_{\eta'}   e^{-2\pi^2 n/\ell_{\eta'}}}{\bigl|\G\bigl(\frac{k}{2}+\frac{2\pi i n}{\ell_{\eta'}} \bigr)\bigr|^2  }
     = \overline{I_{\eta'\eta}(n,m;C,-\beta,-\alpha)}
     \frac{\ell_\eta   e^{-2\pi^2 m/\ell_\eta}}{ \bigl|\G\bigl(\frac{k}{2}+\frac{2\pi i m}{\ell_\eta} \bigr)\bigr|^2  }.
\end{equation}
When $I_{\eta\eta'}(m,n;C,\alpha,\beta)$ is given by \eqref{ffv2} then \eqref{ver} is straightforward to verify. When $I_{\eta\eta'}(m,n;C,\alpha,\beta)$ is given by \eqref{ffv3}, the final step of the verification of \eqref{ver} requires Euler's transformation formula, \cite[(2.2.7)]{AAR}:
\begin{equation*}
    {_2}F_1(a,b;c;1-z) = z^{c-a-b}{_2}F_1(c-a,c-b;c;1-z) \qquad (-\pi<\arg z\lqs \pi).
\end{equation*}

\subsection{Examples} \label{last}
As in Sections \ref{sect_ex} and \ref{num}, we  take the example $\G=\SL_2(\Z)$ with $\eta=\eta'=(-\sqrt{D},\sqrt{D})$ and $\sigma_\eta=\sigma_{\eta'}$ given by $\hat\sigma_\eta$.
For $\left(\smallmatrix
e  &  f \\ g & h
\endsmallmatrix\right) \in \G$, write
\begin{equation*}
    \se^{-1}\begin{pmatrix} e  &  f \\ g & h \end{pmatrix}\se =  \frac{1}{2}\begin{pmatrix} e+g\sqrt{D} +f/\sqrt{D}+h  &  -e-g\sqrt{D} +f/\sqrt{D}+h  \\  -e+g\sqrt{D} -f/\sqrt{D}+h &  e-g\sqrt{D} -f/\sqrt{D}+h  \end{pmatrix} = \begin{pmatrix} a  &  b \\ c & d \end{pmatrix}.
\end{equation*}
Then
\begin{equation} \label{hse}
    ad=\frac 12+\frac 14\left(e^2-Dg^2-\frac{f^2-Dh^2}{D} \right).
\end{equation}
Recall the determination of $\varepsilon_D$ and $ \sigma_\eta^{-1} \G_\eta \sigma_\eta$ in \eqref{deth}. Set $H_D:= \sigma_\eta R_{\eta\eta} \sigma^{-1}_\eta$ (for $R_{\eta\eta}$ defined in \eqref{ree}) to get explicitly
\begin{multline}
       H_D= \left\{  \begin{pmatrix} e  &  f \\ g & h \end{pmatrix}\in \SL_2(\Z) \ \left|  \ \frac 1{\varepsilon_D} \lqs  \left|\frac{(e+g\sqrt{D})^2-(f+h\sqrt{D})^2/D }{(e-g\sqrt{D})^2-(f-h\sqrt{D})^2/D  }\right|^{1/2} < \varepsilon_D, \right. \right.
       \\
        \left.\frac 1{\varepsilon_D} \lqs  \left|\frac{(h-g\sqrt{D})^2-(f-e\sqrt{D})^2/D }{(h+g\sqrt{D})^2-(f+e\sqrt{D})^2/D  }\right|^{1/2} < \varepsilon_D
        \right\} \label{H_D}
\end{multline}
and
let $H_D(C)$ be the elements of $H_D$ with $ad$, given by \eqref{hse}, equalling $C$.

\begin{lemma}
If $\left(\smallmatrix
e  &  f \\ g & h
\endsmallmatrix\right) \in H_D(C)$ for $C \neq 0$, $1$ then
\begin{equation} \label{blas}
    e^2+Dg^2+\frac{f^2+Dh^2}{D} \lqs \left(\varepsilon_D^2+\varepsilon_D^{-2}\right)\left(|C|+|C-1|\right).
\end{equation}
\end{lemma}
\begin{proof}
As in \eqref{inq1}, the inequalities in \eqref{ree} imply $\varepsilon_D^{-2} \lqs |a/d|, \ |b/c| \lqs \varepsilon_D^2$. Arguing as in Lemma \ref{ellipse!},
\begin{align*}
    \varepsilon_D^{-2} \lqs |a/d| \lqs \varepsilon_D^2 & \iff |a/d|^2+1 \lqs \left(\varepsilon_D^2+\varepsilon_D^{-2}\right)|a/d| \\
    & \iff a^2+d^2 \lqs \left(\varepsilon_D^2+\varepsilon_D^{-2}\right)|C|.
\end{align*}
Similarly for $|b/c|$, implying $a^2+b^2+c^2+d^2 \lqs \left(\varepsilon_D^2+\varepsilon_D^{-2}\right)\left(|C|+|C-1|\right)$ which is equivalent to \eqref{blas}.
\end{proof}

So we may calculate the sums in \eqref{mull} as sums over $\left(\smallmatrix
e  &  f \\ g & h
\endsmallmatrix\right) \in H_D(C)$, restricting our attention to entries satisfying \eqref{blas}.
For example, the hyperbolic coefficients at $\eta=(-\sqrt{2},\sqrt{2})$ of $P_{\eta,0}$ with weight $k=12$ are computed in Table \ref{tblh} using Theorem \ref{CIShh} and summing over all $C$ with $|C-1/2|\lqs 20$.
\begin{table}[h]
\begin{center}
\begin{tabular}{c|c|c|c|c|c|c|c}
$n$ & $-3$ & $-2$  & $-1$  & $0$   & $1$  & $2$  & $3$ \\ \hline
$c_{\eta}(n;P_{\eta,0})$ &
$1.0677 \times 10^{-7}$ &
$0.0015600$ &
$-0.083234$ &
$0.88859$ &
$22.4859$ &
$113.849$ &
$-2.105$
\end{tabular}
\caption{Hyperbolic coefficients of $P_{\eta,0}$ for $\eta=(-\sqrt{2},\sqrt{2})$ and $k=12$} \label{tblh}
\end{center}
\end{table}
Since $P_{\eta,0} \approx 1529.46 \Delta \approx  1529.46 P_{\ci,1}/2.840287$, (using Table \ref{pk12} and $\lambda_1$ from \eqref{lams}), we may verify that the coefficients  in Table \ref{tblh} and the first column of Table \ref{pk12b} agree.
It would be interesting to see if the sum $S_{\eta\eta'}(m,n;C,\alpha,\beta)$ has a simple explicit expression similar to that of $S_{\eta\ci}(m,n;C)$ in Theorem \ref{final_k}.

We finally note that Theorem \ref{CIShh}  may be used it to detect  when the negative Pell equation
\er{npell}
has integer solutions. To explain this, let
 $\eta=\eta'=(-\sqrt{D},\sqrt{D})$ for  $\G=\SL_2(\Z)$, as before, and define $\Phi(D)$ as the right side of \eqref{mull} (without \eqref{xab1q}, \eqref{yab1q}) for $m=n=0$, $k=10$:
 \begin{equation*}
    \Phi(D) := \frac{1}{1260\log \varepsilon_D}
    \left(\sum\nolimits_1+\sum\nolimits_2+\sum\nolimits_3\right).
 \end{equation*}
The fundamental solution $(a_0,c_0)$ to the Pell equation \eqref{pell} is built into  $\Phi(D)$ through $\varepsilon_D := a_0 + \sqrt{D} c_0$.
\begin{prop}
The function $\Phi(D)$ takes only the values $0$ and $-1$. The negative Pell equation \eqref{npell} has integer solutions if and only if $\Phi(D)=0$.
\end{prop}
\begin{proof}
Note that $S_{10}(\G)=\{0\}$ and that the hyperbolic pairs $(-\sqrt{D},\sqrt{D})$ and $(\sqrt{D},-\sqrt{D})$ are equivalent in $\SL_2(\Z)$ exactly when \eqref{npell} has integer solutions. With Proposition \ref{trpr}, Theorem \ref{CIShh} yields
\begin{equation*}
    0 = \Phi(D) + \begin{cases} 1 & \text{if} \quad \eta \not\equiv \eta^* \bmod \G, \\
    1+(-1)^5 & \text{if} \quad \eta \equiv \eta^* \bmod \G.
    \end{cases} \qedhere
\end{equation*}
\end{proof}

Examples of $\Phi(D)$ for some small values of $D$ are shown in Table \ref{phi}. They were found by computing $\sum\nolimits_1$, $\sum\nolimits_2$ and $\sum\nolimits_3$ in \eqref{mull} for all $C$ with $|C-1/2|\lqs 2$, using the techniques from earlier in this section.
 If \eqref{npell} has a solution then there is a fundamental one, $(x_0,y_0)$, and all other solutions $(x_n,y_n)$ are given by $x_n+\sqrt{D}y_n =(x_0+\sqrt{D}y_0)^{n+1}$ for $n+1$ odd. See for example \cite{MS12} and its contained references.
When $ (x_0,y_0)$ exists it is given by
\begin{equation*}
      x_0= {\textstyle\frac 12 }\left( \varepsilon^{1/2}_D - \varepsilon^{-1/2}_D \right),
    \quad y_0= {\textstyle\frac 1{2\sqrt{D}} }\left( \varepsilon^{1/2}_D + \varepsilon^{-1/2}_D \right).
\end{equation*}

\begin{table}[h]
\begin{center}
\begin{tabular}{c|c|c|c|c|c|c}
$D$         & $2$       &  $3$          & $5$       & $7$     & $11$ & $13$   \\ \hline
$\varepsilon_D$ & $3+2\sqrt{2} $ &  $2+\sqrt{3}$    &     $9+4\sqrt{5}$      &   $8+3\sqrt{7}$        &    $10+3\sqrt{11}$   &  $649+180\sqrt{13}$  \\ \hline
$\Phi(D)$   & $0.0$     &  $-0.99998$   &  $0.0$    &  $-1.00005$        &$-0.99997$ &  $0.0$ \\ \hline
$(x_0,y_0)$ & $(1,1)$   &  $ $          &  $(2,1)$  &           &     &   $(18,5)$
\end{tabular}
\caption{Solutions of the negative Pell equation} \label{phi}
\end{center}
\end{table}

 Table \ref{phi}, at least, serves as a check of Theorem \ref{CIShh}.
It is well known that \eqref{npell} has solutions if and only if the continued fraction expansion of $\sqrt{D}$ has an odd period. Recently, a very simple criterion was given in \cite{MS12}: for  $D \equiv 1,2 \bmod 4$, equation \eqref{npell} has solutions if and only if $a_0 \equiv -1 \bmod 2D$.


{\small
\bibliography{hc}

\begin{thebibliography}{BKK15}

\bibitem[AAR99]{AAR}
George~E. Andrews, Richard Askey, and Ranjan Roy.
\newblock {\em Special functions}, volume~71 of {\em Encyclopedia of
  Mathematics and its Applications}.
\newblock Cambridge University Press, Cambridge, 1999.

\bibitem[BKK15]{BKK}
Kathrin Bringmann, Ben Kane, and Winfried Kohnen.
\newblock Locally harmonic {M}aass forms and the kernel of the {S}hintani lift.
\newblock {\em Int. Math. Res. Not. IMRN}, (11):3185--3224, 2015.

\bibitem[dAP07]{Pr}
Wladimir de~Azevedo~Pribitkin.
\newblock Uninhibited {P}oincar\'e series.
\newblock {\em Int. J. Number Theory}, 3(3):335--347, 2007.

\bibitem[DJ08]{DJ08}
W.~Duke and Paul Jenkins.
\newblock On the zeros and coefficients of certain weakly holomorphic modular
  forms.
\newblock {\em Pure Appl. Math. Q.}, 4(4, Special Issue: In honor of
  Jean-Pierre Serre. Part 1):1327--1340, 2008.

\bibitem[Goo83]{G83}
Anton Good.
\newblock {\em Local analysis of {S}elberg's trace formula}, volume 1040 of
  {\em Lecture Notes in Mathematics}.
\newblock Springer-Verlag, Berlin, 1983.

\bibitem[Goo85]{G85}
A.~Good.
\newblock Dirichlet and {P}oincar\'e series.
\newblock {\em Glasgow Math. J.}, 27:39--56, 1985.

\bibitem[GR07]{GR}
I.~S. Gradshteyn and I.~M. Ryzhik.
\newblock {\em Table of integrals, series, and products}.
\newblock Elsevier/Academic Press, Amsterdam, seventh edition, 2007.
\newblock Translated from the Russian, Translation edited and with a preface by
  Alan Jeffrey and Daniel Zwillinger, With one CD-ROM (Windows, Macintosh and
  UNIX).

\bibitem[Hir70]{Hir70}
Toyokazu Hiramatsu.
\newblock Eichler maps and hyperbolic {F}ourier expansion.
\newblock {\em Nagoya Math. J.}, 40:173--192, 1970.

\bibitem[Hir72]{Hir}
Toyokazu Hiramatsu.
\newblock Remarks on hyperbolic {P}oincar\'e series.
\newblock {\em Comment. Math. Univ. St. Paul.}, 20:9--14, 1971/72.

\bibitem[IMO]{IMO}
{\"O}zlem Imamo{\=g}lu, Yves Martin, and Cormac O'Sullivan.
\newblock Kernels and {D}irichlet type series associated to the hyperbolic and
  elliptic {F}ourier coefficients of modular forms.
\newblock In preparation.

\bibitem[IO09]{IO09}
{\"O}zlem Imamo{\=g}lu and Cormac O'Sullivan.
\newblock Parabolic, hyperbolic and elliptic {P}oincar\'e series.
\newblock {\em Acta Arith.}, 139(3):199--228, 2009.

\bibitem[Iwa97]{IwTo}
Henryk Iwaniec.
\newblock {\em Topics in classical automorphic forms}, volume~17 of {\em
  Graduate Studies in Mathematics}.
\newblock American Mathematical Society, Providence, RI, 1997.

\bibitem[Kat85]{Ka85}
Svetlana Katok.
\newblock Closed geodesics, periods and arithmetic of modular forms.
\newblock {\em Invent. Math.}, 80(3):469--480, 1985.

\bibitem[Kat92]{Ka92}
Svetlana Katok.
\newblock {\em Fuchsian groups}.
\newblock Chicago Lectures in Mathematics. University of Chicago Press,
  Chicago, IL, 1992.

\bibitem[Kow10]{K10}
Emmanuel Kowalski.
\newblock Poincar\'e and analytic number theory.
\newblock In {\em The scientific legacy of {P}oincar\'e}, volume~36 of {\em
  Hist. Math.}, pages 73--85. Amer. Math. Soc., Providence, RI, 2010.

\bibitem[KZ84]{KZ}
W.~Kohnen and D.~Zagier.
\newblock Modular forms with rational periods.
\newblock In {\em Modular forms ({D}urham, 1983)}, Ellis Horwood Ser. Math.
  Appl.: Statist. Oper. Res., pages 197--249. Horwood, Chichester, 1984.

\bibitem[MS12]{MS12}
R.~A. Mollin and A.~Srinivasan.
\newblock Pell equations: non-principal {L}agrange criteria and central norms.
\newblock {\em Canad. Math. Bull.}, 55(4):774--782, 2012.

\bibitem[OR13]{OSR}
Cormac O'Sullivan and Morten~S. Risager.
\newblock Non-vanishing of {T}aylor coefficients and {P}oincar\'e series.
\newblock {\em Ramanujan J.}, 30(1):67--100, 2013.

\bibitem[Pet30]{P30}
Hans Petersson.
\newblock Theorie der automorphen {F}ormen beliebiger reeller {D}imension und
  ihre {D}arstellung durch eine neue {A}rt {P}oincar\'escher {R}eihen.
\newblock {\em Math. Ann.}, 103(1):369--436, 1930.

\bibitem[Pet32]{P32}
Hans Petersson.
\newblock \"{U}ber die {E}ntwicklungskoeffizienten der automorphen {F}ormen.
\newblock {\em Acta Math.}, 58(1):169--215, 1932.

\bibitem[Pet41]{P41}
Hans Petersson.
\newblock Einheitliche {B}egr\"undung der {V}ollst\"andigkeitss\"atze f\"ur die
  {P}oincar\'eschen {R}eihen von reeller {D}imension bei beliebigen
  {G}renzkreisgruppen von erster {A}rt.
\newblock {\em Abh. Math. Sem. Hansischen Univ.}, 14:22--60, 1941.

\bibitem[Poi12]{P11}
H.~Poincar{\'e}.
\newblock Fonctions modulaires et fonctions fuchsiennes.
\newblock {\em Ann. Fac. Sci. Toulouse Sci. Math. Sci. Phys. (3)}, 3:125--149,
  1912.
\newblock See also Oeuvres, volume II, pages 592--618.

\bibitem[Rad73]{Ra}
Hans Rademacher.
\newblock {\em Topics in analytic number theory}.
\newblock Springer-Verlag, New York, 1973.
\newblock Edited by E. Grosswald, J. Lehner and M. Newman, Die Grundlehren der
  mathematischen Wissenschaften, Band 169.

\bibitem[Ran77]{R77}
Robert~A. Rankin.
\newblock {\em Modular forms and functions}.
\newblock Cambridge University Press, Cambridge-New York-Melbourne, 1977.

\bibitem[Ran96]{rank}
R.~A. Rankin.
\newblock On certain meromorphic modular forms.
\newblock In {\em Analytic number theory, {V}ol.\ 2 ({A}llerton {P}ark, {IL},
  1995)}, volume 139 of {\em Progr. Math.}, pages 713--721. Birkh\"auser
  Boston, Boston, MA, 1996.

\bibitem[Rho12]{Rh12}
Robert~C. Rhoades.
\newblock Linear relations among {P}oincar\'e series via harmonic weak {M}aass
  forms.
\newblock {\em Ramanujan J.}, 29(1-3):311--320, 2012.

\bibitem[Shi71]{S71}
Goro Shimura.
\newblock {\em Introduction to the arithmetic theory of automorphic functions}.
\newblock Publications of the Mathematical Society of Japan, No. 11. Iwanami
  Shoten, Publishers, Tokyo; Princeton University Press, Princeton, N.J., 1971.
\newblock Kan{\^o} Memorial Lectures, No. 1.

\bibitem[Sie65]{Si65}
Carl~Ludwig Siegel.
\newblock {\em Lectures on advanced analytic number theory}.
\newblock Notes by S. Raghavan. Tata Institute of Fundamental Research Lectures
  on Mathematics, No. 23. Tata Institute of Fundamental Research, Bombay, 1965.

\bibitem[vP10]{vP10}
Anna-Maria von Pippich.
\newblock {\em The arithmetic of elliptic {E}isenstein series}.
\newblock PhD thesis, Humboldt-Universit\"at zu Berlin, 2010.

\end{thebibliography}
}

\end{document}